\documentclass[11pt,reqno]{amsart}
\usepackage{amsmath,amssymb,amsthm,amsfonts,latexsym, amsthm, amscd, mathrsfs, stmaryrd}
\usepackage[colorlinks=true,linkcolor=blue,citecolor=blue]{hyperref}
\usepackage{graphicx,color,colortbl}
\usepackage{upgreek}
\usepackage[mathscr]{euscript}
\usepackage{hhline}
\numberwithin{equation}{section}
\usepackage{dynkin-diagrams}
\usepackage{epic}
\allowdisplaybreaks
\numberwithin{equation}{section}
\usepackage{float}
\usepackage{tikz}
\usetikzlibrary{
    cd,
	shapes,
	arrows,
	positioning,
	decorations.markings,
	decorations.pathmorphing,
	circuits.logic.US,
	circuits.logic.IEC,
	fit,
	calc,
	plotmarks,
	matrix
}

\usepackage{enumerate}

\newtheorem{proposition}{Proposition}[section]
\newtheorem{lemma}[proposition]{Lemma}
\newtheorem{corollary}[proposition]{Corollary}
\newtheorem{theorem}[proposition]{Theorem}
\theoremstyle{definition}
\newtheorem{definition}[proposition]{Definition}

\theoremstyle{remark}
\newtheorem{remark}[proposition]{Remark}
\newtheorem{innercustomthm}{{\bf Theorem}}
\newenvironment{customthm}[1]
  {\renewcommand\theinnercustomthm{#1}\innercustomthm}
  {\endinnercustomthm}

\newcommand{\arxiv}[1]{\href{http://arxiv.org/abs/#1}{\tt arXiv:\nolinkurl{#1}}}

\newcommand{\Rmnum}[1]{\expandafter\@slowromancap\romannumeral #1@}
\def \g{\mathfrak{g}}

\def \dm{\diamond}

\def \sl{\mathfrak{sl}}
\def \N{\mathbb{N}}
\def \Q{\mathbb{Q}}

\def \Z{\mathbb{Z}}
\def \I{\mathbb{I}}
\def \Br{\mathrm{Br}}

\def \b{\underline{b}}
\def \B{ b}

\def \wI{\I_{\circ}}
\def \wItau{\I_{\circ,\tau}}
\def \bI{\I_{\bullet}}

\def \diag{\mathrm{diag}}

\def \tfX{\widetilde{\Upsilon}}
\def \cR{\mathcal{R}}
\def \bs{\mathbf{r}} 
\def \bF{\mathbb{F}}
\def \bw{w_\bullet}

\def \tbU{\tU_{\bullet}}

\def \ba{\mathbf{a}}
\def \tk{\widetilde{k}}
\def \tT{\widetilde{\mathscr T}}
\def \tTD{\widetilde{T}}

\def \tTT{\widetilde{\mathbf{T}}}

\newcommand{\tTa}[1]{\tTT'_{#1,-1}}
\newcommand{\tTb}[1]{\tTT''_{#1,+1}}

\def \Id{\mathrm{Id}}
\def \tpsi{\psi} 
\def \tPsi{\widetilde{\Psi}}

\newcommand{\U}{\mathbf{U}}
\newcommand{\Ui}{(\U)^\imath}

\newcommand{\tU}{\widetilde{{\mathbf U}} }
\newcommand{\tUi}{\widetilde{{\mathbf U}}^\imath}

\newcommand{\qbinom}[2]{\begin{bmatrix} #1\\#2 \end{bmatrix} }

\def \ad{\text{ad}}
\def \wad{{}^{\omega\psi}\ad}
\def \sad{{}^{\sigma}\ad}
\def \wsad{{}^{\sigma\omega\psi}\ad}
\def \ov{\overline}
\def \un{\underline}

\newcommand{\nc}{\newcommand}
\nc{\greentext}[1]{\textcolor{green}{#1}}
\nc{\redtext}[1]{\textcolor{red}{#1}}
\nc{\bluetext}[1]{\textcolor{blue}{#1}}
\nc{\brown}[1]{\browntext{ #1}}
\nc{\green}[1]{\greentext{ #1}}
\nc{\red}[1]{\redtext{ #1}}
\nc{\blue}[1]{\bluetext{ #1}}

\def \Q {\mathbb Q}
\def \TT{\mathbf T}

\def \bvs{{\boldsymbol{\varsigma}}}

\def \vs{\varsigma}
\def \U{\mathbf U}
\def \Ui{\mathbf{U}^\imath}
\def \reW{W^\circ}

\def \fwItau{\wItau^{\mathrm{fin}}}
\def \ty{\widetilde{y}}

\newcommand{\odd}{\bar{1}}

\newcommand{\dv}[2]{{B}_{#1}^{{(#2)}}}
\newcommand{\edvi}[1]{{B}_{i,\bar{0}}^{(#1)}}
\newcommand{\odvi}[1]{{B}_{i,\bar{1}}^{(#1)}}

\newcommand{\DE}[1]{E^{(#1)}}
\newcommand{\DF}[1]{F^{(#1)}}
\begin{document}
\title[Relative braid group symmetries of Kac-Moody type]{Relative braid group symmetries on $\imath$quantum groups of Kac-Moody type}

\author[Weinan Zhang]{Weinan Zhang}
\address{Department of Mathematics, University of Virginia, Charlottesville, VA 22904, USA}
\email{wz3nz@virginia.edu}



\begin{abstract}
Recently, relative braid group symmetries on $\imath$quantum groups of arbitrary finite types have been constructed by Wang and the author. In this paper, generalizing that finite-type construction, we establish relative braid group symmetries on $\imath$quantum groups of locally quasi-split Kac-Moody type. We formulate root vectors for $\imath$quantum groups in both recursive forms and closed $\imath$divided power forms. The higher rank formulas of relative braid group symmetries are given by root vectors. We show that the relative braid group symmetries send root vectors to root vectors.
\end{abstract}

\maketitle
\setcounter{tocdepth}{1}
\tableofcontents

\section{Introduction}
\subsection{Background}
Lusztig's braid group actions \cite{Lus90a,Lus90b} are of crucial importance in the theory of quantum groups. These braid group actions are fundamental in the formulation of root vectors, which allows the constructions of PBW bases and canonical bases. They have further applications in various areas, ranging from geometric representation theory to categorification and Hall algebras.

Let $\U$ be the Drinfeld-Jimbo quantum group associated to a generalized Cartan matrix $(c_{ij})_{i,j\in \I}$. The quantum symmetric pair $(\U,\Ui_\bvs)$ associated to a Satake diagram $(\I=\wI\cup\bI,\tau)$ of Kac-Moody type was formulated in \cite{Ko14}, generalizing Letzter's work for finite type \cite{Let02}. We call a symmetric pair {\em quasi-split} if $\bI=\varnothing$, and {\em split} if in addition $\tau =\Id$. The $\imath$quantum group $\Ui_\bvs$, arising from the quantum symmetric pair, is a coideal subalgebra of $\U$ depending on parameters $\bvs$.

Let $\tU=\langle E_i,F_i,K_i,K_i'|i\in \I \rangle $ be the Drinfeld double quantum group, where $K_iK_i'$ are central. The Drinfeld-Jimbo quantum group $\U$ is recovered from $\tU$ by taking central reductions: $\U=\tU/(K_iK_i'-1|i\in \I)$. The {\em universal} quantum symmetric pair $(\tU,\tUi)$ naturally arises from a Hall algebra approach \cite{LW22}, where $\tUi$ is a coideal subalgebra of the Drinfeld double $\tU$. In the {\em universal} $\imath$quantum group $\tUi$, the parameters are replaced by certain central elements, and $\Ui_\bvs$ is recovered by taking central reductions on $\tUi$.

It was conjectured in \cite[Conjecture 1.2]{KP11} that there exist relative braid group actions associated to the relative Weyl group of the underlying symmetric pair on the $\imath$quantum group $\Ui_\bvs$ (of finite type). Kolb-Pellegrini {\em loc. cit.} formulated relative braid group actions via computer computation, for quasi-split finite types and type AII with some specific parameters. Relative braid group actions of type AIII and split affine rank one type were constructed in \cite{Dob20} and \cite{BK20}, respectively.

In the Hall algebra approach, Lu-Wang \cite{LW21a} used reflection functors to construct relative braid group actions on $\tUi$, for quasi-split finite types under the assumption that the Cartan integers $c_{i,\tau i}$ are all even. Since then, it has become clear that $\tUi$, rather than $\Ui_\bvs$, provides the right framework for a conceptual construction of relative braid group actions. In a recent work \cite{LW21b}, they extended the constructions to quasi-split Kac-Moody types under the same assumption.

The quasi $K$-matrix for a general quantum symmetric pair was originally formulated in \cite{BW18a}, as an intertwiner between the bar involution on $\U$ and the bar involution on $\Ui_\bvs$; a proof in greater generality was supplied in \cite{BK19}. An insightful reformulation in \cite{AV20} of quasi $K$-matrix avoided a direct use of bar involutions, and this allows one to remove restrictions on the parameter $\bvs$. A new simple reformulation of the quasi $K$-matrix for arbitrary parameters was given in \cite{WZ22} as an intertwiner between anti-involutions.

A complete answer for Kolb-Pellegrini's longstanding conjecture was given recently by Wang and the author \cite{WZ22}, where relative braid group actions on $\imath$quantum groups of arbitrary finite types were constructed.
The construction was based on new intertwining properties of quasi $K$-matrices. Properties of our relative braid group symmetries are parallel to those well-known properties for Lusztig's braid group symmetries. Compatible relative braid group actions were also constructed on $\U$-modules for the first time. This construction is generalized to the quasi-split affine rank one $\imath$quantum group whose relative Weyl group is of type $A_2^{(2)}$ in \cite{LWZ22}.

\subsection{Goal}

Let $\fwItau$ be a fixed set of representatives in $\wI$ of the finite type real rank one Satake subdiagrams; see \eqref{def:fwItau}. The relative Weyl group $W^\circ$ is a Coxeter group with Coxeter generators $\bs_i, i\in\fwItau$, which can be identified with a subgroup of the Weyl group $W$ \cite{Lus03}; see \eqref{def:Wcirc}. Let $\bw$ be the longest element in the Weyl group associated to $\bI$.

The goal of this paper is to construct relative braid group symmetries associated to any vertex $i\in\fwItau$ such that $i=\bw i$ on $\imath$quantum groups of Kac-Moody type. Here $\bw i=i$ represents that the simple root $\alpha_i$ is fixed by $\bw$, i.e., the vertex $i$ is not connected to any vertices in $\bI$. 
If the Satake diagram is of quasi-split type (i.e., $\bI=\varnothing$), then any $i\in\fwItau$ satisfy $\bw i=i$. In particular, we will obtain relative braid group actions on $\imath$quantum groups of quasi-split Kac-Moody type, extending the previous work \cite{WZ22} on finite type.

Unlike quantum groups, there are 8 distinct finite-type rank one quantum symmetric pairs. Structures vary significantly for $\tUi$ associated to different rank one symmetric pairs, and hence one often has to consider different rank one cases separately. Under the condition $i=\bw i$, there are three distinct rank one cases
\begin{itemize}
\item[(i)] $i=\tau i,$
\item[(ii)] $c_{i,\tau i}=0,$
\item[(iii)]$c_{i,\tau i}=-1,$
\end{itemize}
(all other types of vertices in $\fwItau$ satisfy $i\neq \bw i$).
Note that, when $c_{i,\tau i}\leq -2,i=\bw i$, the underlying Satake subdiagram $\{i,\tau i\}$ is not finite type and the relative Weyl group associated to the subdiagram $\{i,\tau i\}$ is trivial by definition.

We  will construct relative braid group symmetries $\tTT'_{i,e},\tTT''_{i,e}$ for all three types (i)-(iii). Following \cite{LW21a, LW21b, WZ22}, such symmetries will be constructed on the universal $\imath$quantum groups $\tUi$ first, and then descend to $\Ui_\bvs$ by central reductions.

For type (i) and (ii), conjectural formulations for the relative braid group symmetries were given in \cite[Conjecture 6.5]{CLW21a} and \cite[Conjecture 3.7]{CLW21b}, respectively. These conjectures were confirmed via Hall algebras in \cite{LW21b} for quasi-split types. In this paper, we shall prove these conjectures in full generality, removing restrictions that $\bI= \varnothing$  and Cartan integers $c_{i,\tau i}$ are  even.
For type (iii), we are able to establish relative braid group symmetries for the first time.

\subsection{Motivation}

The major difficulty in the generalization is to establish higher rank relative braid group formulas of $\tTa{i}(B_j),\tTb{i}(B_j)$ for $j\neq i,\tau i$; cf. \cite[Conjecture 5.13]{WZ22}. Those higher rank formulas in \cite[Tables 3-4]{WZ22} for finite type were established by lengthy case-by-case computations, using a complete list of finite-type rank two Satake diagrams. 
We will provide a uniform construction of the higher rank formulas in the Kac-Moody setting, independent of the rank two Satake diagrams.

Our construction was inspired by the higher rank formulas of Lusztig's braid group symmetries on quantum groups in the Kac-Moody setting. In \cite[37.2.1]{Lus94}, Lusztig introduced (rank two) elements including $x_{i,j;1,m,e},x'_{i,j;1,m,e}$, $y_{i,j;1,m,e},y'_{i,j;1,m,e}$ in quantum groups. 
Alternatively, these root vectors are determined by recursive relations, according to Lusztig; see Lemma~\ref{lem:Lusxy}.
The uniform higher rank formulas for braid group symmetries on quantum groups are naturally given by these root vectors; see \eqref{eq:TEF}.

In this paper, we extend this picture from quantum groups to $\imath$quantum groups. We introduce root vectors in $\imath$quantum groups in both closed $\imath$divided-power formulas and recursive formulas. We show that there exist the relative braid group symmetries on $\imath$quantum groups, whose higher rank formulas are uniformally given by root vectors.
\subsection{Main results}
\subsubsection{New symmetries on $\imath$quantum groups $\tUi$}

By \cite[Propositions~6.20-6.21]{LW21b}, braid group symmetries on $\U$ can be lifted to $\tU$, and we denote those lifts by $\tTD'_{i,e},\tTD''_{i,e}$ for $i\in \I,e=\pm 1$.
 We need a rescaled version $\tT'_{i,e},\tT''_{i,e}$ of $\tTD'_{i,e},\tTD''_{i,e}$ on $\tU$ following \cite{WZ22}; see \eqref{def:tT}. We have found such a rescaling necessary via examples; the rescaling is also compatible with the one appeared in the factorization of quasi $K$-matrices \cite[\S 3.4]{DK19}, which we will use later in the proof of the braid relations (see \S~\ref{sec:braidrel}). Since $\tT'_{i,e},\tT''_{i,e}$ satisfy the braid relations, we have well-defined automorphisms $\tT'_{w,e},\tT''_{w,e}$ of $\tU$, for any $w\in W$.

The $\imath$quantum group $\tUi$ is the subalgebra of the Drinfeld double $\tU$ generated by elements $B_i,\tk_i,i\in \wI$ and the subalgebra $\tbU=\langle E_j,F_j,K_j,K_j'|j\in \bI\rangle$; see \eqref{def:iQG}. We define the quasi $K$-matrix for $(\tU,\tUi)$ following \cite{WZ22}; see Proposition~\ref{prop:fX1}. Let $\tfX_i$ denote the quasi $K$-matrices associated to the rank one subdiagram $(\I_{\bullet,i}=\bI\cup\{i,\tau i\},\tau|_{\I_{\bullet,i}})$ for $i\in \wI$.

 \begin{customthm} {\bf A}
  [Theorem~\ref{thm:ibraid}]
  \label{thm:A}
Let $i\in \fwItau$ be of type (i)-(iii).
\begin{itemize}
\item[(1)] There exist automorphisms $\tTa{i},\tTb{i}$ on $\tUi$ such that
\begin{align}
\label{eq:intro2}
\tTa{i}(x) \tfX_i &=\tfX_i \tT_{\bs_i,-1}'(x),\qquad \forall x\in \tUi,\\
\tTb{i}(x) \tT''_{\bs_i,+1}(\tfX_i)^{-1} &=\tT''_{\bs_i,+1}(\tfX_i)^{-1}\tT_{\bs_i,+1}''(x),\qquad \forall x\in \tUi.
\label{eq:intro3}
\end{align}

\item[(2)] $\tTa{i},\tTb{i}$ are mutually inverses. Moreover, $\tTa{i}=\sigma^\imath \circ \tTb{i}\circ \sigma^\imath $, where $\sigma^\imath$ is the anti-involution on $\tUi$ defined in Proposition~\ref{prop:sigma}.
\end{itemize}

\end{customthm}

Two more symmetries $\tTT'_{i,+1},\tTT''_{i,-1}$ are defined by conjugations of $\tTT'_{i,-1},\tTT''_{i,+1}$ using the bar involution $\psi^\imath$ on $\tUi$; see \eqref{def:tTT2} and Corollary~\ref{thm:sigma}.

The main step in the proof of Theorem~\ref{thm:A} is finding explicit elements $\tTa{i}(x)$ and $\tTb{i}(x)$ satisfying \eqref{eq:intro2}-\eqref{eq:intro3} for generators $x$ of $\tUi$. Once this is done, we automatically obtain endomorphisms $\tTa{i},\tTb{i}$ on $\tUi$ which satisfy \eqref{eq:intro2}-\eqref{eq:intro3}. Using \eqref{eq:intro2}-\eqref{eq:intro3}, one can then show that $\tTa{i},\tTb{i}$ are mutually inverses.

 The elements $\tTa{i}(x),\tTb{i}(x)$ for generators $x=B_i,B_{\tau i}$ or generators $x=\tk_j,j\in \wI$ or $x\in \tbU$ were uniformly formulated in \cite{WZ22}. We recall those formulations in Proposition~\ref{prop:pfrkone} and Proposition~\ref{prop:rkone1}. We shall discuss the construction of higher rank formulas for $\tTa{i},\tTb{i}$ in the next subsection.

\subsubsection{Higher rank formulas for relative braid group symmetries}
\label{sec:intro2}

 We recall Lusztig's rank 2 root vectors in quantum groups in \eqref{def:f}. Since $\bs_i$ are not simple reflections in $W$ when $i$ are of types (ii)-(iii), we also need some rank 3 root vectors to describe the action of $\tTD'_{\bs_i,e} $ in $\tU$, e.g., elements $y_{i,\tau i,j;m_1,m_2}$ in Definition~\ref{def:qsxy} and elements $y_{i,\tau i,j;a,b,c}$ in Definition~\ref{def:qsqsxy}.

It turns out the recursive definitions for these (rank 2 or 3) root vectors in $\tU$ admit nontrivial generalizations to $\tUi$, which allows us to define the following root vectors

\begin{itemize}
\item[(i)] $\B_{i,j;m},\b_{i,j;m}\in \tUi$ in Definition~\ref{def:s}-\ref{def:s'} when $i=\tau i=\bw i$,
\item[(ii)] $\B_{i,\tau i,j;m_1,m_2},\b_{i,\tau i,j;m_1,m_2}\in \tUi $ in Definition~\ref{def:qsBij}-\ref{def:qsBij'} when $c_{i,\tau i}=0,i=\bw i$,
\item[(iii)] $\B_{i,\tau i,j;a,b,c},\b_{i,\tau i,j;a,b,c}\in \tUi $ in Definition~\ref{def:qsqsBij}-\ref{def:qsqsBij'} when $c_{i,\tau i}=-1,i=\bw i$.
\end{itemize}

Our root vectors admit explicit closed formulas. 
For type (i), the $\imath$divided powers $\dv{i,\ov{p}}{m}$ for $ \ov{p}\in \Z/2\Z$ were formulated in \cite{BW18a,BeW18,CLW21a} (see \eqref{def:idv}), arising from the theory of $\imath$canonical bases.
Closed formulas of $\B_{i,j;m},\b_{i,j;m}$ are given in terms of $\imath$divided powers in Proposition~\ref{thm:idv}; that is, the elements $\B_{i,j;m},\b_{i,j;m}$ coincide with elements $\ty'_{i,j;1,m,\ov{p},\ov{t},1},\ty_{i,j;1,m,\ov{p},\ov{t},1}$ introduced in \cite[\S 6.1]{CLW21a}.

For type (ii)-(iii), these root vectors in $\tUi$ are constructed for the first time and their closed formulas are given in terms of usual divided powers $B_i^{(m)}$. The divided power formulations for $\B_{i,\tau i, j;m_1,m_2},\b_{i,\tau i, j;m_1,m_2}$ and $\B_{i,\tau i, j;a,b,c},\b_{i,\tau i, j;a,b,c}$ are respectively provided in Proposition~\ref{thm:dvij} and Theorem~\ref{thm:qsqsdv}.

The desired higher rank formulas for $\tTa{i}(B_j),\tTb{i}(B_j)$ are given by root vectors in $\tUi$, as formulated in Theorem~\ref{thm:B} below.

 \begin{customthm} {\bf B}
  \label{thm:B}
  [Theorem~\ref{thm:rktwo1}]
Let $i\in \fwItau, j\in \wI$ such that $j\neq i,\tau i$. Write $\alpha=-c_{ij},\beta=-c_{\tau i,j}$.
\begin{itemize}
\item[(i)] If $i=\tau i=\bw i,$ then $\tTa{i}(B_j)=\B_{i,j;\alpha}$ and $\tTb{i}(B_j)=\b_{i,j;\alpha}$. Explicitly, $\tTa{i}(B_j),\tTb{i}(B_j) $ are respectively given by formulas \eqref{eq:srktwo3}-\eqref{eq:srktwo4}.

\item[(ii)] If $c_{i,\tau i}=0,i=\bw i$, then $\tTa{i}(B_j)=\B_{i,\tau i,j;\alpha,\beta} $ and $\tTb{i}(B_j)=\b_{i,\tau i,j;\alpha,\beta} $. Explicitly, $\tTa{i}(B_j),\tTb{i}(B_j) $ are respectively given by formulas \eqref{eq:qsrktwo3}-\eqref{eq:qsrktwo4}.

\item[(iii)] If $c_{i,\tau i}=-1,i=\bw i$, then we have $\tTa{i}(B_j)=\B_{i,\tau i,j;\beta,\beta+\alpha,\alpha} $ and  $\tTb{i}(B_j)=\b_{i,\tau i,j;\beta,\beta+\alpha,\alpha}$. Explicitly, $\tTa{i}(B_j),\tTb{i}(B_j) $ are respectively given by formulas \eqref{eq:rktwo3}-\eqref{eq:rktwo4}.
\end{itemize}
\end{customthm}
The conjectural formulas in \cite[Conjecture 6.5]{CLW21a} and \cite[Conjecture~ 3.7]{CLW21b} for relative braid group actions are verified in full generality by Theorem~\ref{thm:B}(i)-(ii).

The recursive definitions of root vectors play an indispensable role in the proof of Theorem~\ref{thm:B}. The situations are similar for all types (i)-(iii) and we explain for type (ii). By definition, the root vectors $\B_{i,\tau i,j;m_1,m_2}$ in $\tUi$ naturally split into halves
\begin{align*}
\B_{i,\tau i,j;m_1,m_2}=\B_{i,\tau i,j;m_1,m_2}^- +\B_{i,\tau i,j;m_1,m_2}^+,
\end{align*}
and the halves satisfy the same recursions as $\B_{i,\tau i,j;m_1,m_2}$ but have different initial terms; see Definition~\ref{def:qsBij}-\ref{def:qsBij'}.
Thanks to the recursions for $\B_{i,\tau i,j;m_1,m_2}^\pm$, we can establish relations between $\B_{i,\tau i,j;m_1,m_2}^\pm$ and (rank 3) root vectors $y_{i,\tau i,j;m_1,m_2},x_{i,\tau i,j;m_1,m_2}\in \tU$ via suitable intertwiners in Proposition~\ref{prop:qsBy}-\ref{prop:qsBx}. These new intertwining relations imply that the element $\tTa{i}(B_j):=\B_{i,\tau i, j;-c_{ij},-c_{\tau i,j}}$ satisfies the desired identity \eqref{eq:intro2} and then Theorem~\ref{thm:B}(ii) is proved; see Theorem~\ref{thm:qsbraid}.


%
%
\subsubsection{$\tTa{i}$ and root vectors}
The braid group symmetries $\tTD'_{i,e},\tTD''_{i,e}$ on $\tU$ send root vectors to root vectors cf. \cite[Proposition 37.2.5]{Lus94}.
We formulate the $\imath$analog of this property for the relative braid group symmetry $\tTa{i}$ and root vectors in $\tUi$, when the corresponding vertex $i$ is of type (i)-(ii).

 \begin{customthm} {\bf C}
[Theorem~\ref{thm:basic}, Theorem~\ref{thm:qsbasic}]
  \label{thm:D}
Let $i\in\fwItau,j\in \wI$ such that $j\neq i,\tau i$.
\begin{itemize}
\item[(i)] If $i=\tau i=\bw i,$, then we have, for any $m\geq 0$,
\begin{align}
\tTT'_{i,-1}(\b_{ i,j;m})=\B_{i, j;-c_{ij}-m}.
\end{align}

\item[(ii)] If $c_{i,\tau i}=0,i=\bw i,$, then we have, for any $m_1,m_2\geq 0$,
\begin{align}
\tTT'_{i,-1}(\b_{ i, \tau i, j;m_1,m_2})=\B_{i,\tau i,j;-c_{ij}-m_1, -c_{\tau i,j}-m_2}.
\end{align}
\end{itemize}
\end{customthm}

Theorem~\ref{thm:D}(i) in the $\imath$divided power formulation was conjectured by Lu-Wang in a private communication.
By analogy, we conjecture that the symmetry $\tTa{i}$ in type (iii) also preserves root vectors.

\subsubsection{Relative braid group actions on $\imath$quantum groups $\tUi$}


For quantum symmetric pairs (of finite type), the partial quasi $K$-matrix $\tfX_{w}$ for $w\in W^\circ$ was introduced in \cite{DK19} as a product of rank one quasi $K$-matrices. It was conjectured {\em ibid.} that $\tfX_{w}$ is independent of the choice of reduced expressions of $w$ and the conjecture was proved in \cite[\S 8]{WZ22}.

We show that partial quasi $K$-matrices $\tfX_{w}$ in the Kac-Moody setting are independent of the choice of reduced expression of $w\in W^\circ$ in Proposition~\ref{prop:factor}.
We then use this property of partial quasi $K$-matrices and intertwining relations~\eqref{eq:intro2}-\eqref{eq:intro3} to prove the relative braid relations (i.e., defining relations for $\Br(W^\circ)$) for our symmetries $\tTT'_{i,e},\tTT''_{i,e}$, following similar arguments in finite type cases \cite[Theorem 9.1]{WZ22}.

 \begin{customthm} {\bf D}
[Theorem~\ref{thm:BrW0}]
  \label{thm:E}
Fix $e\in \{\pm 1\}$.
The symmetries $\tTT'_{i,e}$ ({\em resp.} $\tTT''_{i,e}$) for vertices $i\in \fwItau$ of types (i)-(iii) satisfy the relative braid relations in $\Br(W^\circ)$.
\end{customthm}
In contrast, very few relative braid relations were verified in the Hall algebra approach \cite{LW21a,LW21b}.

Relative braid group symmetries on $\Ui_\bvs$ for arbitrary parameters $\bvs$ are obtained from symmetries $\tTT'_{i,e},\tTT''_{i,e}$ via central reductions and rescaling automorphisms cf. \cite[\S 9.4]{WZ22}. One can also construct relative braid group symmetries on $\U$-modules which are compatible with those symmetries on $\Ui_\bvs$ for arbitrary parameters $\bvs$, following arguments in \cite[\S 10]{WZ22}. These two constructions are largely formal and identical to their corresponding finite type counterparts in \cite{WZ22}, and hence we will not repeat them in this paper.

\subsection{Application}
A Drinfeld type presentation for $\imath$quantum groups of split affine ADE type was constructed in \cite{LW21a} (based on \cite{BK20}), which can be viewed as both a quantization of the loop realization of $\g^\theta$ and a generalization of the Drinfeld presentation for affine quantum groups. This construction was generalized to the split affine BCFG type by the author \cite{Z22} and further extended to a quasi-split affine rank one type in \cite{LWZ22}.

The relative braid group symmetries have been proven to be crucial in all of the above constructions. The relative braid group symmetries established in \cite{LWZ22} can be recovered as a special case from the general construction in this paper; see Remark~\ref{rmk:A22}. We expect that the relative braid symmetries established in this paper will play an important role in future constructions of the Drinfeld type presentation for general quasi-split types.

\subsection{Organization} This paper is organized as follows. In Section~\ref{sec:pre}, we review the braid group actions on Drinfeld double and the basic theory of universal quantum symmetric pairs $(\tU,\tUi)$.

In Section~\ref{sec:main}, we construct relative braid group symmetries on $\tUi$ following the strategy in \cite{WZ22}. We establish symmetries $\tTa{i},\tTb{i}$ on $\tUi$ via certain intertwining relations involving quasi $K$-matrices in Theorem~\ref{thm:ibraid}. The rank one formulas of our symmetries are recalled in \S\ref{sec:pfrk1} and higher rank formulas are provided in Theorem~\ref{thm:rktwo1}. We show that our symmetries satisfy relative braid group relations in Theorem~\ref{thm:BrW0}.

The next three parallel Sections~\ref{sec:split}-\ref{sec:qsqs} are devoted to the proof of Theorem~\ref{thm:rktwo1} when the vertex $i$ is of type (i)-(iii), respectively. For types (i)-(iii), we construct higher rank root vectors in $\tUi$ via certain recursive relations in Definition~\ref{def:s}-\ref{def:s'}, Definition~\ref{def:qsBij}-\ref{def:qsBij'}, and Definition~\ref{def:qsqsBij}-\ref{def:qsqsBij'}, respectively; we then show that higher rank formulas of $\tTa{i},\tTb{i}$ are given by root vectors. The $\imath$divided power formulas for root vectors are given in Proposition~\ref{thm:idv}, Proposition~\ref{thm:dvij} and Theorem~\ref{thm:qsqsdv}, respectively.

In the last Section~\ref{sec:basic}, we show that our relative braid group symmetries send root vectors to root vectors in $\tUi$; see Theorem~\ref{thm:basic} for type (i) and Theorem~\ref{thm:qsbasic} for type (ii).

\vspace{2mm}
\noindent {\bf Acknowledgement.}
The author would like to thank his advisor Weiqiang Wang for many helpful conversations and advices. The author thanks Yaolong Shen for helpful comments and suggestions. The author also thanks anonymous referees for many valuable comments. This work is partially supported by the GRA fellowship of Wang's NSF grant DMS-2001351.

\section{Preliminaries}\label{sec:pre}
We review the braid group action on Drinfeld double quantum groups. We also review the basic theory of quantum symmetric pairs, quasi $K$-matrices, and relative Weyl groups.
\subsection{Drinfeld doubles and quantum symmetric pairs}
Let $\g$ be a symmetrizable Kac-Moody algebra associated to a generalized Cartan matrix $C=(c_{ij})_{i,j\in \I}$. Let $D=\diag(d_i|i\in \I)$ be the diagonal matrix such that $DC$ is symmetric and $d_i$ are coprime positive integers. Denote by $\alpha_i,\alpha_i^\vee,i\in \I$ the simple roots and simple coroots of $\g$. Let $\g':=[\g,\g]$ be the derived algebra.

Let $q$ be an indeterminate and $\Q(q)$ be the field of rational functions in $q$ with coefficients in $\Q$. Let $\bF$ be the algebraic closure of $\Q(q)$ and $\bF^\times:=\bF\setminus\{0\}$.

Set $q_i:=q^{d_i}$ for $i\in\I$. Denote, for $r,m \in \N$,
\[
 [r]_i =\frac{q_i^r-q_i^{-r}}{q_i-q_i^{-1}},
 \quad
 [r]_i!=\prod_{i=1}^r [i]_i, \quad \qbinom{m}{r}_i =\frac{[m]_i [m-1]_i \ldots [m-r+1]_i}{[r]_i!}.
\]

The Drinfeld double quantum group $\tU:=\tU(\g')$ is the $\bF$-algebra generated by $E_i,F_i,K_i,K_i',i\in \I$ subject to the following relations
\begin{align*}
&[K_i, K_j]=[K_i, K_j']=[K_i', K_j']=0, \quad K_i E_j =q_i^{c_{ij}} E_j K_i,  \quad K_i F_j=q_i^{-c_{ij}} F_j K_i,
\\
&[E_i,F_j]= \delta_{ij} \frac{K_i-K_i'}{q-q^{-1}},\qquad K_i' E_j=q_i^{-c_{ij}} E_j K_i', \qquad K_i' F_j=q_i^{c_{ij}} F_j K_i',
\\
& \sum_{s=0}^{1-c_{ij}} (-1)^s   \DE{s}_i E_j  \DE{1-c_{ij}-s}_i=0,
\qquad \sum_{s=0}^{1-c_{ij}} (-1)^s   \DF{s}_i F_j  \DF{1-c_{ij}-s}_i=0,
\qquad j\neq i.
\end{align*}
where $\DE{s}_i=\frac{E_i^s}{[s]_i!},\DF{s}_i=\frac{F_i^s}{[s]_i!}$.

\begin{remark}
Let $(Y,X,\langle \cdot, \cdot\rangle, \cdots)$ be a root datum of type $(\I, \cdot)$; cf. \cite{Lus94}. We assume that the root datum is $Y$-regular (see \cite[2.2.2]{Lus94} for the definition). Then we can identify the coroot lattice $\bigoplus_{i\in \I} \Z\alpha_i^\vee$ as a sublattice of $Y$.

The Drinfeld double quantum group $\tU(\g)=\langle E_i,F_i,\un{K}_\mu,\un{K}'_\mu|i\in \I, \mu\in Y\rangle$ associated to $\g$ can be formulated such that the Cartan part is parameterized by $Y$; this is a Drinfeld double analog of the quantum group in \cite[\S 3.1.1]{Lus94}. We have the following embedding
\begin{align*}
    \tU(\g')\rightarrow \tU(\g), \quad E_i \mapsto E_i,\quad F_i \mapsto F_i,\quad K_i\mapsto \un{K}_{\alpha_i^\vee}^{d_i},\quad K'_i\mapsto (\un{K}'_{\alpha_i^\vee})^{d_i}.
\end{align*}
The actions of braid group symmetries on the Cartan
part are the same as the Weyl group action on the corresponding lattices; cf. \cite[\S 37.1.3]{Lus94}. So there is
not much difference which versions to use for considering the braid group action, and we will work with $\tU(\g')$ for simplicity.
\end{remark}

\begin{remark}
Note that $K_i K_i'$ for $i\in \I$ are central in $\tU$. The well-known Drinfeld-Jimbo quantum group $\U:=\U(\g')$ is obtained from $\tU$ by taking a central reduction $\U=\tU/(K_i K_i'-1|i\in\I)$.
\end{remark}

\begin{proposition}[\text{cf. \cite[Proposition 2.1]{WZ22}}]
  \label{prop:QG1}
Let $\ba=(a_i)_{i\in \I} \in  ( \bF^\times )^\I$. There exist an automorphism $\tPsi_{\ba}$ on the $\bF$-algebra $\tU$ such that
\begin{align}
\label{tPsi}
\tPsi_{\ba}&: K_i\mapsto a_{i}^{1/2}K_i,\quad  K'_i\mapsto a_{i }^{1/2} K'_i,\quad  E_i\mapsto a_{i}^{1/2} E_i, \quad  F_i\mapsto F_i.
\end{align}
\end{proposition}

An {\em anti-linear} automorphism on $\tU$ is a $\Q$-automorphism which sends $q^{1/m}\mapsto q^{-1/m}$ for all $m>0,m\in \Z$.
\begin{proposition}[\text{cf. \cite[Proposition 2.2]{WZ22}}]
  \label{prop:QG4}
  {\quad}
\begin{enumerate}
\item
There exists an anti-linear involution $\tpsi$ on $\tU$, which fixes $E_i,F_i$ and swaps $K_i \leftrightarrow K_i'$ for $i\in \I$, called the bar involution;
\item
There exists an anti-involution $\sigma$ on $\tU$ which fixes $E_i,F_i$ and swaps $K_i \leftrightarrow K_i'$, for $i\in \I$;
\item
There exists an involution $\omega$ on $\tU$ which swaps $E_i$ and $F_i$ and swaps $K_i \leftrightarrow K_i'$ for $i\in \I$, called the Chevalley involution.
\end{enumerate}
\end{proposition}

Denote by $\tU^+$ the subalgebra of $\tU$ generated by $E_i,i\in \I$. The algebra $\tU^+$ admits an $\N\I$-grading $\tU^+=\bigoplus_{\mu\in \N\I} \tU^+_\mu$ by setting $\deg E_i=\alpha_i$.

\subsection{Braid group action on Drinfeld double $\tU$}

Let $W=\langle s_i|i\in \I\rangle$ be the Weyl group of $\g$, where $s_i$ are the simple reflections.

Denote the divided power $\frac{F}{[r]_i!}$ by $F_i^{(r) }$. Define 
\begin{align}
\label{def:f}
\begin{split}
  y_{i,j;m,e} & = \sum_{r+s=m} (-1)^r q_i^{er(m+c_{ij}-1) } F_i^{(s)} F_j  F_i^{(r)},\qquad y'_{i,j;m,e}  = \sigma(y_{i,j;m,e}),\\ 
  x_{i,j;m,e} & = \sum_{r+s=m} (-1)^r q_i^{er(-m-c_{ij}+1) } E_i^{(r)} E_j  E_i^{(s)},\qquad x'_{i,j;m,e} = \sigma(x_{i,j;m,e}).
  \end{split}
\end{align}
Note that $y_{i,j;m,e}=\sigma\omega\psi (x_{i,j;m,e})$ and $y_{i,j;m,e}=\psi(y_{i,j;m,-e})$.

\begin{remark}
In Lusztig's conventions \cite[\S 37.2.1]{Lus94}, our $y_{i,j;m,e}$ is identified with his $y_{i,j;1,m,e}$ and $x_{i,j;m,e}$ is identified with his $x_{i,j;1,m,e}$.
\end{remark}

\begin{lemma}[\text{cf. \cite[Lemma 7.1.2]{Lus94}}]
\label{lem:Lusxy}
We have, for $i\neq j\in \I$,
\begin{itemize}
\item[(1)] $ -q_i^{-c_{ij}-2m} y_{i,j;m,-1}F_i+ F_i y_{i,j;m,-1}=[m+1]_i  y_{i,j;m+1,-1}$,
\item[(2)] $ -q_i^{c_{ij}+2m}E_i x_{i,j;m,-1}+ x_{i,j;m,-1}E_i=[m+1]_i  x_{i,j;m+1,-1}$,
\item[(3)] $-  y_{i,j;m,-1} E_i+ E_i y_{i,j;m,-1}= [-c_{ij}-m+1]_i  y_{i,j;m-1,-1}K'_{i}$,
\item[(4)] $-F_i x_{i,j;m,-1} + x_{i,j;m,-1}  F_i= [-c_{ij}-m+1]_i  K_{i} x_{i,j;m-1,-1} $.
\end{itemize}
\end{lemma}

The recursive relations for $y_{i,j;m,+1},x_{i,j;m,+1}$ are obtained by applying $\psi$ to above relations. The recursive relations for $y'_{i,j;m,e},x'_{i,j;m,e}$ are obtained by applying $\sigma$ to above relations.

Let $\tTD'_{i,e},\tTD''_{i,e}$ be the braid group symmetries on $\tU$ formulated in \cite[Propositions 6.20-6.21]{LW21b}. Notations $\tTD'_{i,e},\tTD''_{i,e}$ here correspond to $\TT'_{i,e},\TT''_{i,e}$ there. We recall the (rank two) actions of $\tTD'_{i,e},\tTD''_{i,e}$ on generators $E_j,F_j$ of $\tU$ for $j\neq i\in \I$ 
\begin{align}
\label{eq:TEF}
\begin{split}
&\tTD'_{i,e}(F_j)=y_{i,j;-c_{ij},e},\qquad \tTD'_{i,e}(E_{j})=x_{i,j;-c_{ij},e},
\\
&\tTD''_{i,e}(F_j)=y'_{i,j;-c_{ij},-e},\qquad \tTD''_{i,e}(E_{j})=x'_{i,j;-c_{ij},-e}.
\end{split}
\end{align}
cf. \cite[Lemma 37.2.2]{Lus94}. These four symmetries are related via conjugations by the bar involution $\psi$ or the anti-involution $\sigma$
\begin{align*}
\tTD'_{i,e}=\sigma \tTD''_{i,-e} \sigma,\qquad \tTD'_{i,e}=\psi \tTD'_{i,-e} \psi,\qquad \tTD''_{i,e}=\psi \tTD''_{i,-e} \psi.
\end{align*}

\subsection{Quantum symmetric pairs}

Given a subset $\bI\subset \I$ of finite type, denote by $W_{\bullet}$ the parabolic subgroup of $W$ generated by $s_i,i\in \bI.$ Set $\bw$ to be the longest element of $W_\bullet$. Let $\cR_{\bullet}$ be the finite root system generated by simple roots $\alpha_i,i\in \bI.$ Similarly, $\cR_\bullet^\vee$ is the finite coroot system generated by $\alpha_i^\vee,i\in \bI.$
Let $\rho_\bullet$ be the half sum of positive roots in the root system $\cR_{\bullet}$, and $\rho_\bullet^\vee$ be the half sum of positive coroots in $\cR_\bullet^\vee$.

We recall the definition of admissible pairs from \cite[Definition 2.3]{Ko14}. An admissible pair $(\I=\wI\cup \bI,\tau)$ of Kac-Moody type consists of a partition $\wI\cup \bI$ of $\I$, and a Dynkin diagram involution $\tau$ of $\g$ such that
\begin{itemize}
\item[(1)] $\bI$ is of finite type,

\item[(2)]
 $\bw(\alpha_j) = - \alpha_{\tau j}$ for $j\in \bI$,
\item[(3)]

 If $j\in \wI$ and $\tau j =j$, then $\alpha_j(\rho_{\bullet}^\vee)\in \Z$.
\end{itemize}

Associated to each admissible pair, there is a Satake diagram. 
Involutions of the second kind on Kac-Moody algebras are classified by Satake diagrams \cite[Theorem 2.7]{Ko14}. In this paper, we shall interchangeably use
the terms: admissible pairs and Satake diagrams.


We recall the definition of (universal) quantum symmetric pairs from \cite{Let02, Ko14, LW22}. Given a symmetric pair $(\I=\wI\cup\bI,\tau )$, the {\em universal $\imath$quantum group} $\tUi$ is the subalgebra of $\tU$ generated by
\begin{align}
\label{def:iQG}
\begin{split}
&B_i=F_i+\tTD''_{\bw,+1}( E_{\tau i}) K_i', \qquad \tk_i=K_i K'_{\tau i} \qquad (i\in \wI),
\\
&E_j,F_j,K_j,K'_j\qquad (j\in \bI).
\end{split}
\end{align}
The pair $(\tU,\tUi)$ is known as the {\em universal quantum symmetric pair} associated to $(\I=\wI\cup\bI,\tau )$.

Denote $\tU^{\imath 0}$ to be the subalgebra of $\tUi$ generated by $\tk_i,K_j,K_j',i\in \wI,j\in \bI$.
Define $\tbU$ to be the subalgebra of $\tUi$ generated by $E_j,F_j,K_j,K_j',j\in \bI$. 

\subsection{Quasi $K$-matrix}
A formulation of the quasi $K$-matrix associated to $(\U,\Ui_\bvs)$ of Kac-Moody types and general parameters $\bvs$ was given in \cite[Theorem 7.4]{AV20} (see also \cite[Proposition 3.3]{Ko21}), generalizing earlier work \cite{BW18a,BK19}. Appel-Vlaar's definition of quasi $K$-matrices was reformulated in \cite{WZ22} via the anti-involution $\sigma$; bar involutions and anti-involution on $\tUi$ were constructed {\em ibid.} using quasi $K$-matrices. We recall those results from \cite{WZ22}.

 \begin{proposition}[\text{cf. \cite[Theorem 3.6]{WZ22}}]
  \label{prop:fX1}
 There is a unique element $\tfX=\sum_{\mu\in \N\I}\tfX^\mu,$ for $ \tfX^\mu\in \tU^+_\mu$, such that $\tfX^0=1$ and $\tfX$ satisfies the following intertwining relations
 \begin{align}
 \label{eq:fX2}
 \begin{split}
 B_i  \tfX &=  \tfX B_i^{\sigma},\qquad (i\in \wI),\\
 x  \tfX &= \tfX x,\qquad (x\in \tU^{\imath 0}\tbU ).
 \end{split}
 \end{align}
 Moreover, $\tfX^\mu =0$ unless $\theta (\mu) =-\mu$.
 \end{proposition}

The element $\tfX$ is known as the quasi $K$-matrix associated to $(\tU,\tUi)$.

 \begin{proposition}[\text{\cite[Proposition 3.12]{WZ22}}]
   \label{prop:sigma}
 There exists a unique anti-involution $\sigma^\imath$ of $\tUi$ such that
 \begin{align}\label{eq:newb1-2}
 \sigma^\imath(B_i) =B_i, \qquad \sigma^\imath(x) =\sigma(x),
 \quad \text{ for } i\in \wI,x\in \tU^{\imath 0}\tbU.
 \end{align}
Moreover, $\sigma^\imath$ satisfies the following intertwining relation:
 \begin{equation}
   \label{eq:newb1}
 \sigma^\imath(x)  \tfX =\tfX  \sigma(x), \qquad \text{ for all } x\in \tUi.
 \end{equation}
 \end{proposition}

Define the parameter $\bvs_\star=(\vs_{i,\star})_{i\in \wI}$ by setting
\begin{align}
 \label{eq:bvs star}
\vs_{i,\star} =(-1)^{\alpha_i( 2\rho^\vee_\bullet)}q^{(\alpha_i, w_\bullet \alpha_{\tau i}+2\rho_\bullet)},\qquad (i\in \wI).
\end{align}
Set $\tpsi_\star:= \tPsi_{\bvs_\star}\circ\tpsi$. 

\begin{proposition}[\text{\cite[Proposition 3.4]{WZ22}}]
   \label{prop:psi}
There exists a unique anti-linear involution $\tpsi^\imath$ of $\tUi$ such that
 \begin{align}\label{eq:newb9}
 \tpsi^\imath(B_i) =B_i, \qquad \tpsi^\imath(x) =\tpsi_\star(x),
\qquad
\text{ for } i\in \wI, x \in \tU^{\imath 0} \tbU.
 \end{align}
Moreover, $\tpsi^\imath$ satisfies the following intertwining relation,
 \begin{equation}\label{eq:newb10}
 \tpsi^\imath(x)   \tfX =\tfX  \tpsi_\star(x),
 \qquad
  \text{ for all }x\in \tUi.
 \end{equation}
($\tpsi^\imath$ is called a bar involution on $\tUi.$)
\end{proposition}

\subsection{Relative Weyl groups} We recall the construction of relative Weyl groups cf. \cite{Lus03,DK19}. Given a symmetric pair $(\I=\wI\cup\bI,\tau)$, set $\wItau$ to be a fixed set of representatives of $\tau$-orbits on $\wI$. The cardinality of $\wItau$ is called the (real) rank of the symmetric pair $(\I=\wI\cup\bI,\tau)$. For each $i\in \wItau$, set $\I_{\bullet,i}:= \{i,\tau i\}\cup \I_{\bullet}$ and there is a rank one Satake subdiagram $(\I_{\bullet,i}=\{i,\tau i\}\cup\I_\bullet,\tau|_{\I_{\bullet,i}})$. Let $\fwItau$ be the subset of $\wItau$ given by
\begin{align}
\label{def:fwItau}
\fwItau=\{i\in \wItau| \; \I_{\bullet,i} \text{ is of finite type}\}.
\end{align}
Then $\fwItau$ parametrizes finite type rank one Satake subdiagrams of $(\I=\wI\cup\bI,\tau)$.

 Define the subgroup $W_{\bullet,i}:=\langle s_i|i\in \I_{\bullet,i}\rangle$ of $W$ for $i\in \fwItau$. Let $w_{\bullet,i}$ be the longest element of $W_{\bullet,i}$ and define $\bs_i:=w_{\bullet,i}\bw=\bw w_{\bullet,i}$.  The {\em relative Weyl group} $W^\circ$ associated to $(\I=\wI\cup\bI,\tau)$ is defined to be
\begin{align}
\label{def:Wcirc}
W^\circ:=\langle \bs_i | i\in \fwItau\rangle.
\end{align}
It is well-known that $W^\circ $ is a Coxeter group with simple reflections $\bs_i,i\in \fwItau$ \cite{Lus03}. We denote by $\Br(W^\circ)$ the braid group associated to $W^\circ$ and call it {\em the relative braid group}.

\begin{remark}
In the Kac-Moody case, there are nontrivial symmetric pairs with trivial relative Weyl groups. For instance, consider the Satake diagram below (the underlying Dynkin diagram is of untwisted affine rank one type).
\begin{center}
\begin{tikzpicture}[baseline=0,scale=1.2]
		\node  at (-0.65,0) {$\circ$};
		\node  at (0.65,0) {$\circ$};
		\draw[-] (-0.6,0.05) to (0.6, 0.05);
		\draw[-] (-0.6,-0.05) to (0.6,- 0.05);
        \draw[bend left,<->,red] (-0.65,0.15) to (0.65,0.15);
        \node at (0,0.18) {$\textcolor{red}{\tau}$};
	\end{tikzpicture}
\end{center}
In this case, $\fwItau=\varnothing$ and then the relative Weyl group is trivial.

The corresponding $\imath$quantum group is also known as the augmented $q$-Onsager algebra \cite{BB12}.
\end{remark}

\section{Main results}\label{sec:main}
We construct relative braid group symmetries $\tTa{i},\tTb{i}$ on $\imath$quantum groups $\tUi$ in Theorem~\ref{thm:ibraid}, generalizing the finite-type construction in \cite{WZ22}. We recall closed formulas for their rank one actions from \cite{WZ22} in \S~\ref{sec:pfrk1}. The higher rank formulas of these symmetries are presented in \S~\ref{sec:pfrk2}, whose proofs occupy the coming sections. We show that these symmetries satisfy relative braid relations in Theorem~\ref{thm:BrW0}.

\subsection{New symmetries on $\tUi$}\label{sec:sym}

By the definition of $W^\circ$, it suffices to construct relative braid group symmetries associated to vertices $i\in \fwItau$.
Write $\bw i= i$ for $\bw(\alpha_i)=\alpha_i$. In this paper, we will construct such symmetries for any $i\in \fwItau$ such that $\bw i=i$. Any vertex $i\in \fwItau$ satisfying $\bw i=i$ belongs to exactly one of the following three types
\begin{itemize}
\item[(i)] $i=\tau i=\bw i,$
\item[(ii)] $c_{i,\tau i}=0,i=\bw i,$
\item[(iii)]$c_{i,\tau i}=-1,i=\bw i$.
\end{itemize}

We need a renormalization on symmetries $\tTD'_{i,-1},\tTD''_{i,+1}$, analogous to \cite[\S 4.1]{WZ22}. Let $\bvs_{\dm}=(\vs_{\dm,i})_{i\in \wI}$ be the distinguished parameters such that
\begin{align}
\label{def:bvs}
\vs_{\dm,i}= - q^{-(\alpha_i , \alpha_i+\bw\alpha_{\tau i})/2} .
\end{align}
We extend $\bvs_{\dm}$ to an $\I$-tuple of scalars by setting $\vs_{\dm,j}=1$ for $j\in \bI$.

Recall the scaling automorphism $\tPsi_{\ba}$ from Proposition~\ref{prop:QG1}. Define
\begin{align}
  \label{def:tT}
  \begin{split}
& \tT_{i,+1}'':= \tPsi_{\bvs_\diamond}^{-1} \circ \tTD''_{i,+1} \circ \tPsi_{\bvs_\diamond},
\\
 & \tT_{i,-1}':= \tPsi_{\bvs_\diamond}^{-1} \circ \tTD'_{i,-1} \circ \tPsi_{\bvs_\diamond}.
 \end{split}
\end{align}

We will mainly use the rescaled symmetries $\tT_{i,+1}'',\tT_{i,-1}'$ rather than the original ones $\tTD_{i,+1}'',\tTD_{i,-1}'$  in most parts of this paper.  We introduce the short notations
\begin{align*}
\tT_i:=\tT_{i,+1}'', \qquad \tT_i^{-1}:= \tT_{i,-1}'.
\end{align*}

The rank one quasi $K$-matrices $\tfX_i$ are defined to be the quasi $K$-matrix associated to rank one subdiagram $(\I_{\bullet,i}=\{i,\tau i\}\cup\I_\bullet,\tau|_{\I_{\bullet,i}})$.

\begin{theorem}
\label{thm:ibraid}
Let $i\in \fwItau$ be a vertex of type (i)-(iii).
\begin{itemize}
\item[(1)]
For any $x\in \tUi$, there exists a unique element $x'\in \tUi$ such that
\begin{align}\label{eq:inter}
x' \tfX_i =\tfX_i \tT_{\bs_i,-1}'(x).
\end{align}
Moreover, the map $x \mapsto x'$ is an automorphism of $\tUi$, denoted by $\tTa{i}$.

\item[(2)] For any $x\in \tUi$, there exists a unique element $x''\in \tUi$ such that
\begin{align}\label{eq:inter'}
x'' \tT_{\bs_i}(\tfX_i)^{-1} =\tT_{\bs_i}(\tfX_i)^{-1}\tT_{\bs_i,+1}''(x).
\end{align}
Moreover, the map $x \mapsto x''$ is an automorphism of $\tUi$, denoted by $\tTb{i}$.

\item[(3)] Automorphisms $\tTa{i}$ and $\tTb{i}$ are mutually inverse and they satisfy $\tTT'_{i,-1}=\sigma^\imath \circ \tTT''_{i,+1} \circ \sigma^\imath$.
\end{itemize}
\end{theorem}

Therefore, we have the following two intertwining relations
\begin{align}
\label{eq:interT}
\begin{split}
\tTa{i}(x) \tfX_i &=\tfX_i \tT_{\bs_i,-1}'(x),\\
\tTb{i}(x) \tT_{\bs_i}(\tfX_i)^{-1} &=\tT_{\bs_i}(\tfX_i)^{-1}\tT_{\bs_i,+1}''(x).
\end{split}
\end{align}

\begin{proof}
An outline of the proof is given below and a complete proof requires the developments in the coming part of this paper.

Fist of all, since $\tfX_i$ is invertible, the uniqueness in either theorem is clear.

In the next two subsections~\ref{sec:pfrk1}-\ref{sec:pfrk2}, we will show the existence of $x', x''$ for any generator $x$ of $\tUi$. Precisely, we show the existence of $x', x''$ when $x\in \tU^{\imath 0}\tbU$ or  $x=B_i,B_{\tau i}$ in Propositions~\ref{prop:pfrkone}-\ref{prop:rkone1}; we show the existence of $x', x''$ when $x=B_j$ for $j\neq i,\tau i,j\in \wI$ in Theorem~\ref{thm:rktwo1}.

Assume that $x', y' \in \tUi$ satisfy \eqref{eq:inter}, for $x,y \in \tUi$. Then it follows that $(xy)':=x'y' \in \tUi$ satisfies the identity in \eqref{eq:inter} for $xy$. Hence we have obtained a well-defined endomorphism $\tTa{i}$ on $\tUi$ which sends $x \mapsto x'$. Similarly, we have obtained a well-defined endomorphism $\tTb{i}$ on $\tUi$ which sends $x \mapsto x''$.

Finally, we show that the endomorphisms $\tTa{i},\tTb{i}$ are mutually inverses. Recall that $\tT'_{i,-1},\tT''_{i,+1}$ are mutually inverses. For any $z\in \tUi$, set $y=\tTb{i}(z)$. The second relation in \eqref{eq:interT} is written as $y \tT_{\bs_i}(\tfX_i)^{-1} =\tT_{\bs_i}(\tfX_i)^{-1}\tT_{\bs_i,+1}''(z)$, which is equivalent to
\begin{align*}
z \tfX_i &=\tfX_i \tT_{\bs_i,-1}'(y).
\end{align*}
Using the first relation \eqref{eq:interT} and the uniqueness, it must be $z=\tTa{i}(y)$. Therefore, we have proved that $\tTa{i},\tTb{i}$ are mutually inverses, which means that they are both automorphisms on $\tUi$.

The property $\tTT'_{i,-1}=\sigma^\imath \circ \tTT''_{i,+1} \circ \sigma^\imath$ is a consequence of Proposition~\ref{prop:pfrkone}(3), Proposition~\ref{prop:rkone1}(3), and Theorem~\ref{thm:rktwo1}(3).
\end{proof}

Using the bar involution $\psi^\imath$ on $\tUi$ (see Proposition~\ref{prop:psi}), we define other two variants of the symmetries
\begin{align}
\label{def:tTT2}
\tTT'_{i,+1}:= \psi^\imath\circ \tTa{i}\circ \psi^\imath,
\qquad
\tTT''_{i,-1}:= \psi^\imath\circ \tTb{i}\circ \psi^\imath.
\end{align}

Recall that $\sigma^\imath$ commutes with $\psi^\imath$. Then we have the next corollary of Theorem~\ref{thm:ibraid}(3).
\begin{corollary}
\label{thm:sigma}
Let $e=\pm 1$. The symmetries $\tTT'_{i,e}$ and $\tTT''_{i,-e}$ for $i\in \fwItau$ are mutually inverse. Moreover, we have
\begin{align}
    \tTT'_{i,e}=\sigma^\imath \circ \tTT''_{i,-e} \circ \sigma^\imath.
\end{align}
\end{corollary}


\subsection{Rank one formulation}\label{sec:pfrk1}
We prove the existence of $x',x''$ in Theorem~\ref{thm:ibraid} for $x\in \tU^{\imath 0}\tbU$ or $x=B_i,B_{\tau i}$ in this subsection. This is achieved by Proposition~\ref{prop:pfrkone}-\ref{prop:rkone1} below, whose formulations and proofs are essentially the same as corresponding statements in \cite{WZ22}. For completeness and convenience, we still sketch proofs for them.

 \begin{proposition}[\text{cf. \cite[Proposition 4.11 and Proposition 6.2]{WZ22}}]
 \label{prop:pfrkone}
Let $i\in \fwItau$.
\begin{itemize}
\item[(1)] For $x\in \tU^{\imath 0}\tbU$, the element $x':=\tT'_{\bs_i,-1}(x)\in \tUi$  satisfies \eqref{eq:inter}.

\item[(2)] For $x\in \tU^{\imath 0}\tbU$, the element $x'':=\tT''_{\bs_i,+1}(x)\in \tUi$  satisfies  \eqref{eq:inter'}.

\item[(3)] We have $\tTb{i}(x)=(\sigma^\imath \circ \tTa{i}\circ \sigma^\imath)(x)$ for any $x\in \tU^{\imath 0}\tbU$.
\end{itemize}
 \end{proposition}

In other word, on $ \tU^{\imath 0}\tbU$, the action of our symmetry $\tTa{i}$ ($resp.$ $\tTb{i}$) is the same as the action of $\tT'_{\bs_i,-1}$ ($resp.$ $\tT''_{\bs_i,+1}$).

\begin{proof}
Using $\bs_i=w_{\bullet,i} \bw$, one can show that both $\tT'_{\bs_i,-1}(x),\tT''_{\bs_i,+1}(x)\in \tU^{\imath 0}\tbU$ for any $x\in \tU^{\imath 0}\tbU$. The proposition is then a straightforward consequence of Proposition~\ref{prop:fX1}.
\end{proof}

 \begin{proposition}[\text{cf. \cite[Theorem 4.14 and Proposition 6.3]{WZ22}}]
  \label{prop:rkone1}
Let $i\in \fwItau$.
\begin{itemize}
\item[(1)] There exists a unique element $\tTa{i}(B_i)\in \tUi$ which satisfies the following intertwining relation 
\begin{align}
   \label{eq:rkone1}
\tTa{i}(B_i) \tfX_i =\tfX_i \tT'_{\bs_i,-1}(B_i).
\end{align}
More explicitly, we have $ \tTa{i}(B_i)=-q_i^{- c_{i,\tau i} }   B_{\tau_{\bullet,i} \tau i} \tk_{\tau_{\bullet,i} \tau i}^{ -1}.$

\item[(2)]  There exists a unique element $\tTb{i}(B_i)\in \tUi$ which satisfies the following intertwining relation 
\begin{align}
   \label{eq:rkone1'}
\tTb{i}(B_i) \, \tT_{\bs_i}(\tfX_i)^{-1} = \tT_{\bs_i}(\tfX_i)^{-1} \, \tT''_{\bs_i,+1}(B_i).
\end{align}
More explicitly, we have $\tTb{i}(B_i)=-q_i^{-2}  B_{\tau_{\bullet,i} \tau i}  \tk_{\tau_{\bullet,i}  i}^{-1}.$

\item[(3)] We have $\tTb{i}(B_i)=(\sigma^\imath \circ \tTa{i}\circ \sigma^\imath)(B_i)$.
\end{itemize}
 \end{proposition}

 \begin{proof}
 We give the proof for (1) here; the proof for (2) is similar and omitted; (3) follows by the explicit formulas of $ \tTa{i}(B_i),\tTb{i}(B_i)$ in (1)(2).

 When $\bw i=i$, $B_i=F_i +E_{\tau i} K_i'$ and $\bs_i$ is the longest element of the subgroup $\langle s_i,s_{\tau i}\rangle\leqslant W$. By a direct computation, we have
 \begin{align}
 \label{eq:TBi}
 \tT'_{\bs_i,-1}(B_i) =-q_i^{-c_{i,\tau i}}  B_{\tau_{\bullet,i} \tau i}^{\sigma} \tk_{\tau_{\bullet,i} \tau i}^{-1}.
 \end{align}
 By Proposition~\ref{prop:sigma}, we have $\tfX_i B_{\tau_{\bullet,i} \tau i}^{\sigma} = B_{\tau_{\bullet,i} \tau i} \tfX_i$. Since $\tk_{\tau_{\bullet,i} \tau i} \in \tU^{\imath 0}$,  $\tk_{\tau_{\bullet,i} \tau i}$ commutes with $\tfX_i$. Using these properties and \eqref{eq:TBi}, we have
 \begin{equation}
  -q_i^{- c_{i,\tau i} }   B_{\tau_{\bullet,i} \tau i} \tk_{\tau_{\bullet,i} \tau i}^{ -1}  \tfX_i
  =\tfX_i  \tT'_{\bs_i,-1}(B_i).
 \end{equation}
 Clearly, $B_{\tau_{\bullet,i} \tau i} \tk_{\tau_{\bullet,i} \tau i}^{ -1} \in \tUi$. Hence, we have proved (1).
 \end{proof}

\begin{remark}\label{rmk:rkone}
Since $\bs_i$ commutes with $\tau$, we have $\tau \tT'_{\bs_i,-1}=\tT'_{\bs_i,-1}\tau$. Recall that $\tau \tfX_i = \tfX_i$ from \cite[Proposition 3.8]{WZ22}. Then the following elements
\begin{align*}
\tTa{i}(B_{\tau i}):=\tau \tTa{i}(B_{i}),\qquad \tTb{i}(B_{\tau i}):=\tau \tTb{i}(B_{i})
\end{align*}
satisfy \eqref{eq:inter}-\eqref{eq:inter'} respectively for $x=B_{\tau i}$ cf. \cite[Remark 4.8 and \S 4.7]{WZ22}.
\end{remark}

\subsection{Higher rank formulation}\label{sec:pfrk2}

We sketch the proof of the existence of $x',x''$ in Theorem~\ref{thm:ibraid} for $x= B_j,j\neq i,\tau i,j\in \wI$ in this subsection. We need to find elements $\tTa{i}(B_j),\tTb{i}(B_j)\in \tUi$ such that
\begin{align}
\label{eq:inter2}
\tTa{i}(B_j) \tfX_i &=\tfX_i \tT_{\bs_i,-1}'(B_j),\\
\tTb{i}(B_j) \tT_{\bs_i}(\tfX_i)^{-1} &=\tT_{\bs_i}(\tfX_i)^{-1}\tT_{\bs_i,+1}''(B_j).
\label{eq:inter3}
\end{align}
The proofs depend on which of types (i)-(iii) the vertex $i\in \fwItau$ belongs to. We will construct root vectors in $\tUi$ for each of these three types in Sections~\ref{sec:split}-\ref{sec:qsqs}:
\begin{itemize}
\item[(i)] For $i=\tau i=\bw i$, we define root vectors $\B_{i,j;m},\b_{i,j;m}\in \tUi$ for $m\geq 0$ in Definition~\ref{def:s}-\ref{def:s'}.

\item[(ii)] For $c_{i,\tau i}=0,i=\bw i$, we define root vectors $\B_{i,\tau i,j;m_1,m_2},\b_{i,\tau i,j;m_1,m_2} \in \tUi$  in Definition~\ref{def:qsBij}-\ref{def:qsBij'}.

\item[(iii)] For $c_{i,\tau i}=-1,i=\bw i$, we define root vectors $\B_{i,\tau i,j;a,b,c},\b_{i,\tau i,j;a,b,c}\in\tUi$ in Definition~\ref{def:qsqsBij}-\ref{def:qsqsBij'}.
\end{itemize}

It turns out that the desired elements $\tTa{i}(B_j),\tTb{i}(B_j)$ are given by these root vectors, as formulated in the Theorem~\ref{thm:rktwo1} below.

The following lemma will be useful later for all three types (i)-(iii).
\begin{lemma}[\text{cf. \cite[Lemma 5.1]{WZ22}}]
  \label{lem:rktwo1}
For $i,j\in \wI, j\neq i,\tau i$, we have
\begin{align}
 F_j  \tfX_i &=\tfX_i  F_j,
 \label{eq:FjUp}\\
\tT'_{\bs_i,-1} (\tT_{\bw}(E_{\tau j}) K_j' )  \cdot \tfX_i
&= \tfX_i \cdot  \tT'_{\bs_i,-1} (\tT_{\bw}(E_{\tau j})  K_j' ).
 \label{eq:TEjUp}
\end{align}
\end{lemma}

For type (i), the $\imath$divided powers were formulated in \cite[(2.20)-(2.21)]{CLW21a}, generalizing \cite{BW18a, BeW18}. We recall definitions {\em loc. cit.} of $\imath$divided powers $\dv{i,\ov{p}}{m}\in\tUi$ for $m\geq 0,p\in \Z/2\Z$ as follows
\begin{align}
\label{def:idv}
\begin{split}
&\edvi{m}=\frac{1}{[m]_i!}
\begin{cases}
B_i \prod_{r=1}^{k} (B_i^2-q_i \tk_i [2r]_i^2),& \text{ if } m=2k+1,
\\
\prod_{r=1}^{k} (B_i^2-q_i \tk_i [2r-2]_i^2),& \text{ if } m=2k;
\end{cases}
\\
&\odvi{m}=\frac{1}{[m]_i!}
\begin{cases}
B_i \prod_{r=1}^{k} (B_i^2-q_i \tk_i [2r-1]_i^2),& \text{ if } m=2k+1,
\\
\prod_{r=1}^{k} (B_i^2-q_i \tk_i [2r-1]_i^2),& \text{ if } m=2k.
\end{cases}
\end{split}
\end{align}

For type (ii)-(iii), the $\imath$divided powers are the same as usual divided powers
\begin{align*}
    B_i^{(m)}=\frac{B_i^m}{[m]_i!}.
\end{align*}
These $\imath$divided powers have appeared in conjectural formulas for relative braid group actions cf. \cite[Conjecture 6.5]{CLW21a} for type (i) and in \cite[Conjecture 3.7]{CLW21b} for type (ii). These conjectures were confirmed via Hall algebras in \cite{LW21b} under assumptions that $\bI=\varnothing$ and $c_{j,\tau j}$ are even for all $j\in \I$. We shall prove these conjectures in full generality in Theorem~\ref{thm:rktwo1}(i)(ii) respectively.

\begin{theorem}
  \label{thm:rktwo1}
Let $i\in \fwItau, j\in \wI$ such that $j\neq i,\tau i$. Write $\alpha=-c_{ij},\beta=-c_{\tau i,j}$.
\begin{itemize}
\item[(i)] If $i=\tau i=\bw i,$ then the element $\tTa{i}(B_j):=\B_{i,j;\alpha}$ satisfies \eqref{eq:inter2} and the element $\tTb{i}(B_j):=\b_{i,j;\alpha}$ satisfies \eqref{eq:inter3}. Explicitly, we have
\begin{align}
\notag
\tTa{i}(B_j)=&\sum_{r+s=\alpha} (-1)^{r} q_i^{r} \dv{i,\ov{p+\alpha}}{s} B_j \dv{i,\ov{p}}{r} 
\\\label{eq:srktwo3}
&+\sum_{u\geq 1}\sum_{\substack{r+s+2u=\alpha\\ \ov{r}=\ov{p}}} (-1)^{r+u} q_i^{r+2u} \dv{i,\ov{p+\alpha}}{s} B_j \dv{i,\ov{p}}{r} \tk_i^u,
\\\notag
\tTb{i}(B_j)=&\sum_{r+s=\alpha} (-1)^{r} q_i^{r} \dv{i,\ov{p}}{r} B_j \dv{i,\ov{p+\alpha}}{s}
\\\label{eq:srktwo4}
&+\sum_{u\geq 1}\sum_{\substack{r+s+2u=\alpha\\ \ov{r}=\ov{p}}} (-1)^{r+u} q_i^{r+2u} \dv{i,\ov{p}}{r} B_j \dv{i,\ov{p+\alpha}}{s} \tk_i^u.
\end{align}

\item[(ii)] If $c_{i,\tau i}=0,i=\bw i$, then the element $\tTa{i}(B_j):=\B_{i,\tau i,j;\alpha,\beta} $ satisfies \eqref{eq:inter2} and the element $\tTb{i}(B_j):=\b_{i,\tau i,j;\alpha,\beta} $ satisfies \eqref{eq:inter3}. Explicitly, we have
\begin{align}\notag
\tTa{i}(B_j)&=\sum_{u=0}^{ \min(\alpha,\beta)} \sum_{r=0}^{\alpha-u}\sum_{s=0}^{\beta-u} (-1)^{r+s+u}q_i^{r(-u+1)+s(u+1)+u }
\\\label{eq:qsrktwo3}
  &\qquad \times B_i^{(\alpha-r-u)}B_{\tau i}^{(\beta-s-u)}B_j B_{\tau i}^{(s)}B_i^{(r)} \tk_i^u ,
  \\\notag
\tTb{i}(B_j)&=\sum_{u=0}^{ \min(\alpha,\beta)} \sum_{r=0}^{\alpha-u}\sum_{s=0}^{\beta-u} (-1)^{r+s+u}q_i^{r( -u+1)+s( u+1)+u }
\\\label{eq:qsrktwo4}
  &\qquad \times \tk_{\tau i}^u B_i^{(r)}B_{\tau i}^{(s)}B_j B_{\tau i}^{(\beta-s-u)}B_i^{(\alpha-r-u)}.
\end{align}

\item[(iii)] If $c_{i,\tau i}=-1,i=\bw i$, then the element $\tTa{i}(B_j):=\B_{i,\tau i,j;\beta,\beta+\alpha,\alpha } $ satisfies \eqref{eq:inter2} and the element $\tTb{i}(B_j):=\b_{i,\tau i,j;\beta,\beta+\alpha,\alpha } $ satisfies \eqref{eq:inter3}. Explicitly, we have
    \begin{align}\notag
&\tTa{i}(B_j)=\sum_{u,v\geq 0} \sum_{t=0}^{\beta-v} \sum_{s=0}^{\beta+\alpha-v-u} \sum_{r=0}^{\alpha-u} (-1)^{t+v+r+s+u}
 q_i^{t(-2v+1) + r(u+1)+s(v-2u+1)+uv}
 \\\label{eq:rktwo3}
&\qquad\times q_i^{-\frac{u(u-1)+v(v-1)}{2}} B_i^{(\beta-v-t)}B_{\tau i}^{(\alpha+\beta-v-u-s)}B_i^{(\alpha-u-r)} B_j B_i^{(r)} B_{\tau i}^{(s)}\tk_{\tau i}^u B_i^{(t)} \tk_i^v,
\\\notag
&\tTb{i}(B_j)=\sum_{u,v\geq 0} \sum_{t=0}^{\beta-v} \sum_{s=0}^{\beta+\alpha-v-u} \sum_{r=0}^{\alpha-u} (-1)^{t+v+r+s+u}
 q_i^{t(-2v+1) + r(u+1)+s(v-2u+1)+uv}
 \\\label{eq:rktwo4}
&\qquad\times q_i^{-\frac{u(u-1)+v(v-1)}{2}} \tk_{\tau i}^v B_i^{(t)}\tk_{i}^u B_{\tau i}^{(s)} B_i^{(r)} B_j B_i^{(\alpha-u-r)}B_{\tau i}^{(\alpha+\beta-v-u-s)} B_i^{(\beta-v-t)}.
\end{align}
\end{itemize}
\end{theorem}

\begin{proof}
We only outline the proof here; details will be included in later Sections~\ref{sec:split}-\ref{sec:qsqs}. The first statements in (i)-(iii) are respectively proved in Theorem~\ref{thm:split}, Theorem~\ref{thm:qsbraid}, and Theorem~\ref{thm:qsqs}. The explicit formulas in (i) are obtained by specializing $\imath$divided power formulations for $\B_{i,j;m},\b_{i,j;m}$ in Proposition~\ref{thm:idv} at $m=\alpha$. The explicit formulas in (ii) are obtained by specializing $\imath$divided power formulations for $\B_{i,\tau i,j;m_1,m_2},\b_{i,\tau i,j;m_1,m_2}$ in Proposition~\ref{thm:dvij} at $m_1=\alpha,m_2=\beta$. The explicit formulas in (iii) are obtained by specializing $\imath$divided power formulations for $\B_{i,\tau i,j;a,b,c},\b_{i,\tau i,j;a,b,c}$ in Theorem~\ref{thm:qsqsdv} at $a=\beta,b=\alpha+\beta,c=\alpha$.
\end{proof}

\begin{remark}
In \cite[Theorem 6.10-6.11]{LW21b}, Lu-Wang formulated four (relative) braid group symmetries $\TT'_{i,e},\TT''_{i,e}$ on $\tUi$ for $c_{i,\tau i}=0$. In fact, our symmetries can be related to theirs via a rescaling automorphism $\Phi$ on $\tUi$, which sends
\begin{align*}
\Phi: B_i \mapsto -q_i^{-1} B_i ,\quad B_{\tau i}\mapsto B_{\tau i}, \quad \tk_i\mapsto -q_i^{-1} \tk_i,\quad \tk_{\tau i}\mapsto -q_i^{-1} \tk_{\tau i},
\end{align*}
and fixes $B_j,\tk_j$ if $j\neq i,\tau i$.

One can then show that $\TT''_{i,-1}=\Phi \tTT'_{i,-1}\Phi^{-1}$ and $\TT'_{i,+1}=\Phi \tTT''_{i,+1}\Phi^{-1}$.
\end{remark}

\begin{remark}\label{rmk:A22}
Lu, Wang and the author in \cite{LWZ22} constructed relative braid group symmetries $\TT_1,\TT_0$ associated to the following quasi-split Satake diagram
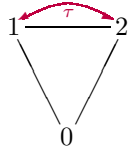
\begin{center}\setlength{\unitlength}{0.7mm}
		\begin{equation}
			\label{eq:satakerank1}
			\begin{picture}(50,13)(0,-10)
				\put(-0.5,-2){\small $1$}
				\put(20,-2){\small $2$}
				
				\put(12.5,-19){\line(1,2){8}}
				\put(9.5,-19){\line(-1,2){8}}
				\put(9.5,-23){\small $0$}
				\put(3,-.5){\line(1,0){16}}
				\color{purple}
				\qbezier(11,4)(15,3.7)(19.5,1)
				\qbezier(11,4)(7,3.7)(2.5,1)
				\put(19,1.1){\vector(2,-1){0.5}}
				\put(2,1.1){\vector(-2,-1){0.5}}
				\put(10,-.5){\small $^{\tau}$}
			\end{picture}
		\end{equation}
		\vspace{.5cm}
	\end{center}
\end{remark}
The complicated formula $\TT_1(B_0)$ in \cite[Theorem 2.8]{LWZ22} can be recovered by setting $\alpha=\beta=1$ in \eqref{eq:rktwo4}. Then the construction of relative braid group symmetries in that paper can be viewed as a special case of this paper.

\subsection{Relative braid group actions on $\tUi$}\label{sec:braidrel}


 We show that our symmetries $\tTT'_{i,-1},\tTT''_{i,+1}$ lead to relative braid group actions on $\tUi$.

 Let $w$ be any element in the relative Weyl group $\reW$ with a reduced expression
\[
\underline{w} =\bs_{i_1}\bs_{i_2}\cdots \bs_{i_m}.
\]

Following \cite{DK19,WZ22} (who worked in finite type), for $1\leq k\leq m$, we define partial quasi $K$-matrix $\tfX_{\underline{w}}$ by
\begin{align}
  \label{eq:PK}
\begin{split}
\tfX^{[k]} 
&= \tT_{\bs_{i_1}} \tT_{\bs_{i_2}} \cdots\tT_{\bs_{i_{k-1}}}(\tfX_{i_k}),
\\
\tfX_{\underline{w}} &=\tfX^{[m]}\tfX^{[m-1]}\cdots\tfX^{[1]}.
\end{split}
\end{align}
(In the notation $\tfX^{[k]}$ above, we have suppressed the dependence on $\underline{w}$.)

We formulate a property for partial quasi $K$-matrices, which generalizes \cite[Theorem 8.1(1)]{WZ22} from finite type to Kac-Moody type.
\begin{proposition}
\label{prop:factor}
Let $(\I=\wI\cup\bI,\tau)$ be a Satake diagram of Kac-Moody type. For any $w \in \reW$, the partial quasi $K$-matrix $\tfX_{\underline{w}}$ is independent of the choice of reduced expressions of $w$ and hence can be denoted by $\tfX_w$.
\end{proposition}

\begin{proof}
We claim that Proposition~\ref{prop:factor} holds for $(\I=\wI\cup\bI,\tau)$ if it holds for all finite type rank two subdiagrams of $(\I=\wI\cup\bI,\tau)$.
The proof for this claim is essentially the same as \cite[Theorem 3.17]{DK19} (Even though only finite type Satake diagrams are considered in that paper, the proof therein can be adopted for our proposition with no difficulty).

It remains to show that, for any finite type rank two subdiagrams of $(\I=\wI\cup\bI,\tau)$, Proposition~\ref{prop:factor} holds. Indeed, this claim is a consequence of \cite[Theorem 8.1]{WZ22}, where this property of partial quasi $K$-matrices is proved for any finite type Satake diagrams.
\end{proof}

For any $i\neq j\in \fwItau$, let $m_{ij}$ denote the order of $\bs_i\bs_j$ in $W^\circ$. Then the following relative braid relation holds in $\Br(W^\circ)$
\begin{align*}
\underbrace{\bs_i\bs_j \bs_i \cdots }_{m_{ij}}=\underbrace{\bs_j\bs_i\bs_j \cdots}_{m_{ij}}.
\end{align*}

\begin{theorem}
\label{thm:BrW0}
Fix $e\in \{\pm 1\}$. Let $i,j\in \fwItau$ of types (i)-(iii) such that $i\neq j$.
We have
\begin{align}
\begin{split}
\underbrace{\tTT'_{i,e} \tTT'_{j,e} \tTT'_{i,e}  \cdots }_{m_{ij}}
&=\underbrace{\tTT'_{j,e} \tTT'_{i,e} \tTT'_{j,e}  \cdots}_{m_{ij}},
\\
\underbrace{\tTT''_{i,e} \tTT''_{j,e} \tTT''_{i,e}  \cdots }_{m_{ij}}
&=\underbrace{\tTT''_{j,e} \tTT''_{i,e} \tTT''_{j,e}  \cdots}_{m_{ij}}.
\end{split}
\end{align}
\end{theorem}

\begin{proof}
Thanks to \eqref{def:tTT2} and Theorem~\ref{thm:sigma}, it suffices to show that $\tTT'_{i,-1}$ satisfy the relative braid relations. Indeed, one can adapt \cite[Proof of Theorem 9.1]{WZ22} to the Kac-Moody setting and then the desired statement is a consequence of Proposition~\ref{prop:factor} and relation~\eqref{eq:interT}.
\end{proof}

Recall that any vertex $i\in \fwItau$ is of types (i)-(iii) for a quasi-split Satake diagram.
By Theorem~\ref{thm:BrW0}, we can now construct a relative braid group action for all quasi-split quantum symmetric pairs.

\begin{corollary}\label{cor:braid}
Fix $e\in \{\pm1\}$. Let $(\tU,\tUi )$ be a quantum symmetric pair of quasi-split type.
Then there exists a relative braid group $\Br(W^\circ)$-action on $\tUi$ as automorphisms of algebras, such that the action of $\bs_i$ is given by $\tTT'_{i,e}$ ({\em resp.}  $\tTT''_{i,e}$)  for $i\in \fwItau$.
\end{corollary}

\section{Higher rank formulas for $\tau i=i=\bw i$}
\label{sec:split}
 In this section, we fix an $i\in \fwItau$ such that $\tau i=i=\bw i $. In this case, $B_i=F_i+E_i K_i',\tk_i=K_i K_{i}',$ and $\bs_i=s_i$.

We define root vectors $\B_{i,j;m},\b_{i,j;m}\in \tUi$ for $j\in\wI, j\neq i$ in Definitions~\ref{def:s}-\ref{def:s'} via recursive relations.
The $\imath$divided power formulations for these elements are obtained in Proposition~\ref{thm:idv}.
We show that $\B_{i,j;-c_{ij}},\b_{i,j;-c_{ij}}$ provide the higher rank formulas for $\tTa{i}(B_j),\tTb{i}(B_j)$ in Theorem~\ref{thm:split} and complete the proof for Theorem~\ref{thm:rktwo1}(i).

\subsection{Definitions of root vectors}

Note that, in this case, $c_{ij}=c_{i,\tau j}$.
\begin{definition}
\label{def:s}
Let $j\in \wI,j\neq i$. Let $\B_{i,j;m}^\pm $ be the elements in $\tU$ defined recursively for $m\geq-1$ as follows
 \begin{align}\notag
 &\B_{i,j;0}^-=F_j, \quad \B_{i,j;0}^+=\tT_{\bw}(E_{\tau j}) K_j', \quad \B_{i,j;-1}^\pm=0,\\\notag
 &-q_i^{-(c_{ij}+2m)} \B_{i,j;m}^\pm B_i +B_i \B_{i,j;m}^\pm\\
 =&[m+1]_i  \B_{i,j;m+1}^\pm +[- c_{ij}-m+1]_i q_i^{-2m-c_{ij}+2} \B_{i,j;m-1}^\pm \tk_i,\qquad m\geq 0.
  \label{def:Bij}
 \end{align}
\end{definition}

Set $\B_{i,j;m}:=\B_{i,j;m}^- +\B_{i,j;m}^+$. Since $\B_{i,j;0}=B_j\in \tUi$, one can recursively show that $\B_{i,j;m}\in \tUi$ for any $m\geq -1$.

\begin{definition}
\label{def:s'}
Let $j\in \wI,j\neq i$. Let $\b_{i,j;m}^\pm$ be the elements in $\tU$ defined recursively for $m\geq-1$ as follows
 \begin{align}\notag
 &\b_{i,j;0}^-=F_j,\quad \b_{i,j;0}^+=\tT_{\bw}(E_{\tau j}) K_j', \quad \b_{i,j;-1}^\pm=0,\\\notag
 &-q_i^{-(c_{ij}+2m)}B_i \b_{i,j;m}^\pm + \b_{i,j;m}^\pm B_i\\
 =&[m+1]_i  \b_{i,j;m+1}^\pm +[- c_{ij}-m+1]_i q_i^{-2m-c_{ij}+2} \b_{i,j;m-1}^\pm \tk_i,\qquad m\geq 0.
  \label{def:Bij2}
 \end{align}
\end{definition}

Set $\b_{i,j;m}:=\b_{i,j;m}^- +\b_{i,j;m}^+$. One can recursively show that $\b_{i,j;m}\in \tUi$ for any $m\geq -1$.

Recall the anti-involution $\sigma^\imath$ on $\tUi$ from Proposition~\ref{prop:sigma}.
\begin{proposition}
\label{prop:bB}
Let $j\in \wI,j\neq i$. Then $\b_{i,j;m}=\sigma^\imath ( \B_{i,j;m})$ for $m\geq 0$.
\end{proposition}

\begin{proof}
The recursive relation~\eqref{def:Bij} implies that $\sigma^\imath ( \B_{i,j;m})$ satisfies the same relation~\eqref{def:Bij2} as $\b_{i,j;m}$. Since $\b_{i,j;0}=B_j=\sigma^\imath ( \B_{i,j;0})$, this proposition follows by induction.
\end{proof}

\begin{lemma}
\label{lem:BB}
We have, for $m\geq 0,,j\in \wI, j\neq i$,
\begin{align}\label{eq:BB}
\b^-_{i, j;m}&=\sigma \big( \tfX_i^{-1} \B^-_{i,j;m} \tfX_i \big).
\end{align}
\end{lemma}

\begin{proof}
Consider the subalgebra $\tU_{[i;j]}^-$ of $\tU$ generated by $B_i,\tk_i,F_j$. It is clear from the above definitions that $\B^-_{i, j;m},\b^-_{i,j;m}\in \tU_{[i;j]}^-$. By Proposition~\ref{prop:fX1} and Lemma~\ref{lem:rktwo1}, there is a well-defined anti-automorphism $\sigma_{ij}$ on $\tU_{[i;j]}^-$, which is given by
\begin{align*}
\sigma_{ij}: x \mapsto \sigma \big( \tfX_i^{-1} x \tfX_i \big).
\end{align*}
Moreover, $\sigma_{ij}$ fixes $B_i,F_j,\tk_i$. Applying $\sigma_{ij}$ to \eqref{def:Bij}, it is clear that $ \sigma_{ij} (\B^-_{i,j;m})$ satisfies the same recursive relation as $\b^-_{i,j;m}$. Then the desired identity follows by induction.
\end{proof}

\begin{remark}
One can also formulate the relation between $\b^+_{i,j;m},\B^+_{i,j;m}$. However, it is much more complicated than \eqref{eq:BB} and hard to prove directly. We do not need the relation between $\b^+_{i,j;m},\B^+_{i,j;m}$ in this paper. The situation is similar when the vertex $i$ is of type (ii)-(iii).
\end{remark}

%
%
\subsection{An $\imath$divided power formulation}
\label{sec:idv}
In \cite[\S 6.1]{CLW21a}, elements $\ty_{i,j;n,m,\ov{p},\ov{t},e},\ty'_{i,j;n,m,\ov{p},\ov{t},e}$ in $\tUi$ were defined via $\imath$divided powers.
 Comparing Definition~\ref{def:s} and \cite[Theorem 6.2]{CLW21a}, it is clear that our $\B_{i,j;m}$ ({\em resp.} $\b_{i,j;m}$) and their $\ty'_{i,j;1,m,\ov{p},\ov{t},1}$ ({\em resp.} $\ty_{i,j;1,m,\ov{p},\ov{t},1}$) satisfy the same recursive relations, which implies that
\begin{align}
\B_{i,j;m}=\ty'_{i,j;1,m,\ov{p},\ov{t},1},\qquad \b_{i,j;m}=\ty_{i,j;1,m,\ov{p},\ov{t},1},\qquad (m\geq 0,\ov{p},\ov{t}\in \Z/2\Z).
\end{align}
We thus obtained $\imath$divided power formulations for $\B_{i,j;m},\b_{i,j;m}$ in the next proposition, by specializing $\imath$divided power formulas of $\ty_{i,j;n,m,\ov{p},\ov{t},e},\ty'_{i,j;n,m,\ov{p},\ov{t},e} $ at $n=1,e=1$.

\begin{proposition}
\label{thm:idv}
Let $i\in \fwItau,j\in \wI$ such that $\tau i=i=\bw i,j\neq i$.
\begin{itemize}
\item[(1)] The element $\B_{i,j;m}$ in Definition~\ref{def:s} admits an $\imath$divided power formulation:
for $m+c_{ij}$ odd,
\begin{align*}
\B_{i,j;m}=&\sum_{u\geq 0} (q_i \tk_i)^u \Big\{ \sum_{\substack{r+s+2u=m\\ \ov{r}=\ov{p}+\odd}} (-1)^r q_i^{-(m+c_{ij})(r+u)+r}\qbinom{\frac{m+c_{ij}-1}{2}}{u}_{q_i^2} \dv{i,\ov{p}}{s} B_j \dv{i,\ov{p+c_{ij}}}{r}+
\\
&+\sum_{\substack{r+s+2u=m\\\ov{r}=\ov{p} }} (-1)^r q_i^{-(m+c_{ij}-2)(r+u)-r}\qbinom{\frac{m+c_{ij}-1}{2}}{u}_{q_i^2} \dv{i,\ov{p}}{s} B_j \dv{i,\ov{p+c_{ij}}}{r}\Big\},
\end{align*}
and for $m+c_{ij}$ even,
\begin{align*}
\B_{i,j;m}=&\sum_{u\geq 0} (q_i \tk_i)^u \Big\{ \sum_{\substack{r+s+2u=m \\ \ov{r}=\ov{p}+\odd}} (-1)^r q_i^{-(m+c_{ij}-1)(r+u)}\qbinom{\frac{m+c_{ij}}{2}}{u}_{q_i^2} \dv{i,\ov{p}}{s} B_j \dv{i,\ov{p+c_{ij}}}{r}+
\\
&+\sum_{\substack{r+s+2u=m\\\ov{r}=\ov{p} }} (-1)^r q_i^{-(m+c_{ij}-1)(r+u)}\qbinom{\frac{m+c_{ij}-2}{2}}{u}_{q_i^2} \dv{i,\ov{p}}{s} B_j \dv{i,\ov{p+c_{ij}}}{r}\Big\}.
\end{align*}
\item[(2)] The element $\b_{i,j;m}$ in Definition~\ref{def:s'} admits an $\imath$divided power formulation:
for $m+c_{ij}$ odd,
\begin{align*}
\b_{i,j;m}=&\sum_{u\geq 0} (q_i \tk_i)^u \Big\{ \sum_{\substack{r+s+2u=m\\ \ov{r}=\ov{p}+\odd}} (-1)^r q_i^{-(m+c_{ij})(r+u)+r}\qbinom{\frac{m+c_{ij}-1}{2}}{u}_{q_i^2} \dv{i,\ov{p}}{r} B_j \dv{i,\ov{p+c_{ij}}}{s}+
\\
&+\sum_{\substack{r+s+2u=m\\\ov{r}=\ov{p} }} (-1)^r q_i^{-(m+c_{ij}-2)(r+u)-r}\qbinom{\frac{m+c_{ij}-1}{2}}{u}_{q_i^2} \dv{i,\ov{p}}{r} B_j \dv{i,\ov{p+c_{ij}}}{s}\Big\},
\end{align*}
and for $m+c_{ij}$ even,
\begin{align*}
\b_{i,j;m}=&\sum_{u\geq 0} (q_i \tk_i)^u \Big\{ \sum_{\substack{r+s+2u=m \\ \ov{r}=\ov{p}+\odd}} (-1)^r q_i^{-(m+c_{ij}-1)(r+u)}\qbinom{\frac{m+c_{ij}}{2}}{u}_{q_i^2} \dv{i,\ov{p}}{r} B_j \dv{i,\ov{p+c_{ij}}}{s}+
\\
&+\sum_{\substack{r+s+2u=m\\\ov{r}=\ov{p} }} (-1)^r q_i^{-(m+c_{ij}-1)(r+u)}\qbinom{\frac{m+c_{ij}-2}{2}}{u}_{q_i^2} \dv{i,\ov{p}}{r} B_j \dv{i,\ov{p+c_{ij}}}{s}\Big\}.
\end{align*}
\end{itemize}
\end{proposition}

\subsection{Intertwining properties}
\label{sec:sinter}

We write $y_{i,j;m},x_{i,j;m}$ for $y_{i,j;m,-1},x_{i,j;m,-1}$ respectively.
\begin{proposition}
\label{prop:By}
We have for any $m\geq 0, j\in \wI, j\neq i$,
\begin{align}
\label{eq:By1}
\B_{i,j;m}^- \tfX_i = \tfX_i   y_{i,j;m}
\end{align}
\end{proposition}

\begin{proof}
We use an induction on $m$. The base cases $m=0$ is a consequence of Lemma~\ref{lem:rktwo1}.

Suppose that \eqref{eq:By1} is true for $1,2,\cdots, m $. Note that $B_i^\sigma=F_i+K_i E_i.$ We have, by induction hypothesis and Proposition~\ref{prop:fX1},
\begin{align*}
&\tfX_i^{-1} (-q_i^{-(c_{ij}+2m)}\B_{i,j;m}^- B_i +B_i \B_{i,j;m}^-)\tfX_i\\
=
 & -q_i^{-(c_{ij}+2m)} y_{i,j;m}B_i^\sigma + B_i^\sigma y_{i,j;m}\\
=
 & -q_i^{-(c_{ij}+2m)} y_{i,j;m} F_i + F_i y_{i,j;m} -q_i^{-(c_{ij}+2m)} y_{i,j;m}K_i E_i + K_i E_i y_{i,j;m}\\
=
 & -q_i^{-(c_{ij}+2m)} y_{i,j;m} F_i + F_i y_{i,j;m} +K_i (-  y_{i,j;m}E_i+   E_i y_{i,j;m}).
\end{align*}
Using Lemma~\ref{lem:Lusxy} to simplify the RHS of above formula, we obtain
\begin{align*}
\text{RHS}=&[m+1]_i  y_{i,j;m+1}+[-c_{ij}-m+1]_i K_i y_{i,j;m-1}K'_{i}\\
=&[m+1]_i  y_{i,j;m+1}+[- c_{ij}-m+1]_i q_i^{-2m-c_{ij}+2}y_{i,j;m-1} K_i K'_{i}.
\end{align*}
Combining the above two computations, we have the following identity
\begin{align}\notag
&-q_i^{-(c_{ij}+2m)}\B_{i,j;m}^- B_i +B_i \B_{i,j;m}^- \\
=&\tfX_i ([m+1]_i  y_{i,j;m+1}+[- c_{ij}-m+1]_i q_i^{-2m-c_{ij}+2}y_{i,j;m-1} K_i K'_{i})\tfX_i^{-1}.
\label{eq:By2}
\end{align}
On the other hand, by definition \eqref{def:Bij}, we have
\begin{align}
\label{eq:By3}
&-q_i^{-(c_{ij}+2m)}\B_{i,j;m}^- B_i +B_i \B_{i,j;m}^-\\\notag
 =&[m+1]_i  \B_{i,j;m+1}^- +[- c_{ij}-m+1]_i q_i^{-2m-c_{ij}+2} \B_{i,j;m-1}^- \tk_i.
\end{align}
Comparing \eqref{eq:By2} with \eqref{eq:By3} and then using the induction hypothesis, we deduce that \eqref{eq:By1} is true for $m+1$ as desired.
\end{proof}

\begin{proposition}
\label{prop:Bx}
We have for any $m\geq 0, j\in \wI, j\neq i$,
\begin{align}
\label{eq:Bx}
\B_{i,j;m}^+=(-1)^m q_i^{ -2m(c_{ij}+m-1)} \tT_{\bw}( x_{i,\tau j;m}) K'_{j}(K_i')^m.
\end{align}
\end{proposition}

\begin{proof}
We use an induction on $m$. The base cases $m=0,1$ are verified by straightforward computations.


We denote $\tT_{\bw}(x_{i,\tau j;m})$ by $x_{i,\bw\tau j;m}$ in the proof. Let $R_m$ denote the RHS \eqref{eq:Bx}, i.e., $R_m=(-1)^m q_i^{ -2m(c_{ij}+m-1)} x_{i,\bw\tau j;m}K'_{j}(K_i')^m$. It suffices to show that $R_m$ satisfies the same recursive relation as $\B_{i,j;m}^+$.

Let $Q_m$ denote the element $Q_m=  x_{i,\bw\tau j;m} K'_{j}(K_i')^m$. We first obtain the recursive relation for $Q_m$ as follows
\begin{align*}
& -q_i^{-(c_{ij}+2m)}Q_m B_i +B_i Q_m  \\
=&-q_i^{ -(c_{ij}+2m)} x_{i,\bw\tau j;m}K'_{j}(K_i')^m( F_i +E_i K_i')+ (F_i +E_i K_i') x_{i,\bw\tau j;m}K'_{j}(K_i')^m \\
=&-\big(x_{i,\bw\tau j;m}F_i- F_i  x_{i,\bw\tau j;m}\big)K'_{j}(K_i')^m \\
 &- q_i^{- 2(c_{ij}+2m)}(x_{i,\bw\tau j;m} E_i -q_i^{c_{ij}+2m}E_i x_{i,\bw\tau j;m}) K'_{j}(K_i')^{m+1}.
\end{align*}
Recall that $c_{i,\tau j}=c_{ij}$ in this case. Applying Lemma~\ref{lem:Lusxy} to the above formula, we have
\begin{align*}
RHS=&-[-c_{ij}-m+1]_i K_i x_{i,\bw\tau j;m-1} K'_{j}(K_i')^m\\
    &- q_i^{-2(c_{ij}+2m)}[m+1]_i x_{i,\bw\tau j;m+1} K'_{j}(K_i')^{m+1}\\
=&-[-c_{ij}-m+1]_i q_i^{c_{ij}+2m-2}x_{i,\bw\tau j;m-1}  K'_{j}(K_i')^{m-1}\tk_i\\
 &- q_i^{-2(c_{ij}+2m)}[m+1]_i x_{i,\bw\tau j;m+1} K'_{j}(K_i')^{m+1}.
\end{align*}
Combining the above two formulas, we have
\begin{align}\notag
 &-q_i^{-(c_{ij}+2m)}Q_m B_i +B_i Q_m \\\label{eq:Bx2}
 =&-[-c_{ij}-m+1]_i q_i^{c_{ij}+2m-2} Q_{m-1}\tk_i- q_i^{-2(c_{ij}+2m)}[m+1]_i Q_{m+1}.
\end{align}
Note that $R_m=(-1)^m q_i^{ -2m(c_{ij}+m-1)} Q_m$. Then, by \eqref{eq:Bx2}, we have
\begin{align}
 -q_i^{-(c_{ij}+2m)}R_m B_i +B_i R_m =&[m+1]_i R_{m+1}+[-c_{ij}-m+1]_i q_i^{-(c_{ij}+2m-2)} R_{m-1}\tk_i. \label{eq:Bx3}
\end{align}
Comparing \eqref{eq:Bx3} with \eqref{def:Bij}, it is clear that $R_m$ satisfies the same recursive relation as $\B_{i,j;m}^+$.

Therefore, $\B_{i,j;m}^+=R_m$ for $m\geq 0$.
\end{proof}

We next formulate the relations between $\b_{i,j;m}^\pm$ and $y'_{i,j;m},x'_{i,j;m}$.
\begin{proposition}
\label{prop:B'y}
We have, for $m\geq 0, j\in \wI, j\neq i$,
\begin{align}
\label{eq:B'y}
\b_{i,j;m}^- = y'_{i,j;m}.
\end{align}
\end{proposition}

\begin{proof}
This proposition is a consequence of Proposition~\ref{prop:By} and Lemma~\ref{lem:BB}.
\end{proof}

Set $\widehat{B}_i:=-q_i^{-2}\tT_{i}(B_i \tk_i^{-1})=q_i^{-2} F_i K_i K_i'^{-1}+E_i K_i'$, and then we have $\widehat{B}_i\tT_i(\tfX_i)=\tT_i(\tfX_i) B_i$, following \cite[\S 6.4]{WZ22}.
.

\begin{proposition}
\label{prop:B'x}
We have, for $m\geq 0, j\in \wI, j\neq i$,
\begin{align}
\label{eq:B'x}
\b_{i,j;m}^+ \tT_i(\tfX_i)^{-1} =(-1)^m  q_i^{ -2m(c_{ij}+m-1)} \tT_i(\tfX_i)^{-1} \tT_{\bw}(x'_{i,\tau j;m}) K_j' (K_i')^m.
\end{align}
\end{proposition}

\begin{proof}
We denote $\tT_{\bw}(x'_{i,\tau j;m})$ by $x'_{i,\bw\tau j;m}$ in the proof. Let $R_m$ denote RHS \eqref{eq:B'x} i.e.,
\begin{align*}
R_m=(-1)^m  q_i^{ -2m(c_{ij}+m-1)} \tT_i(\tfX_i)^{-1} x'_{i,\bw \tau j;m} K_j' (K_i')^m \tT_i(\tfX_i).
\end{align*}
We use an induction on $m$. By definition, $\b_{i,j;0}^+=\tT_{\bw}(E_{\tau j})K_j'=R_0$ and $\b_{i,j;-1}^+=0=R_{-1}$. Hence, it suffices to show that $R_m$ satisfies the same recursive relation as $\b_{i,j;m}^+$.

The recursive relation for $x'_{i,\tau j;m}$ is obtained by applying $\sigma$ to Lemma~\ref{lem:Lusxy}(2)(4). Since $\bw i=i$, both $E_i,F_i$ are fixed by $\tT_{\bw}$ and hence $x'_{i,\bw\tau j;m}$ satisfy the same recursive relation as $x'_{i,\tau j;m}$. Thus, we have
\begin{align}
\label{eq:Lusx'}
\begin{split}
-q_i^{c_{ij}+2m} x'_{i,\bw\tau j;m}E_i+ E_i x'_{i,\bw\tau j;m} &=[m+1]_i  x'_{i,\bw\tau j;m+1},\\
- x'_{i,\bw\tau j;m} F_i +  F_i x'_{i,\bw\tau j;m} &= [-c_{ij}-m+1]_i  x'_{i, \bw\tau j;m-1} K_{i}'.
\end{split}
\end{align}

Let $Q_m:=\tT_i(\tfX_i)^{-1} x'_{i,\bw\tau j;m} K_j' (K_i')^m \tT_i(\tfX_i)$.
We first formulate the recursive relation for $Q_m$ as follows
\begin{align*}
&\tT_i(\tfX_i) \big(-q_i^{-(c_{ij}+2m)} B_i Q_m + Q_m B_i\big) \tT_i(\tfX_i)^{-1}\\
=&-q_i^{-(c_{ij}+2m)} \widehat{B}_i x'_{i,\bw\tau j;m} K_j' (K_i')^m  +  x'_{i,\bw\tau j;m} K_j' (K_i')^m  \widehat{B}_i \\
=&-q_i^{-(c_{ij}+2m)} (q_i^{-2} F_i K_i K_i'^{-1}+E_i K_i') x'_{i,\bw\tau j;m} K_j' (K_i')^m \\
 & + x'_{i,\bw\tau j;m} K_j' (K_i')^m  (q_i^{-2} F_i K_i K_i'^{-1}+E_i K_i')\\
=&-q_i^{ c_{ij}+2m -2} (-F_i x'_{i,\bw\tau j;m}+ x'_{i,\bw\tau j;m} F_i)K_i K_j' (K_i')^{m-1}\\
 & -q_i^{-2 c_{ij}-4m} (E_i x'_{i,\bw\tau j;m} -q_i^{c_{ij}+2m} x'_{i,\bw\tau j;m}E_i )K_j' (K_i')^{m+1}.
\end{align*}
Now applying \eqref{eq:Lusx'} to the RHS, we obtain
\begin{align*}
\text{RHS }&=-q_i^{ c_{ij}+2m -2}  [-c_{ij}-m+1]_i  x'_{i,\bw\tau j;m-1} K_j' (K_i')^{m-1} \tk_i\\
&\quad +  -q_i^{-2 c_{ij}-4m} [m+1]_i  x'_{i,\bw\tau j;m+1} K_j' (K_i')^{m+1} .
\end{align*}
Combining the above two formulas, we conclude that $Q_m$ satisfies
\begin{align}\notag
-q_i^{-(c_{ij}+2m)} B_i Q_m + Q_m B_i=&-q_i^{ c_{ij}+2m -2}  [-c_{ij}-m+1]_i  Q_{m-1} \tk_i \\
&-q_i^{-2 c_{ij}-4m} [m+1]_i  Q_{m+1}.
\end{align}
Hence, by definition, $R_m$ satisfies recursive relation
\begin{align}
-q_i^{-(c_{ij}+2m)} B_i R_m + R_m B_i= q_i^{- c_{ij}-2m +2}  [-c_{ij}-m+1]_i  R_{m-1} \tk_i+  [m+1]_i  R_{m+1},
\end{align}
which is the same as the defining recursive relation for $\b_{i,j;m}^+$. Therefore, $\b_{i,j;m}^+=R_m$ for $m\geq 0$.
\end{proof}

\subsection{Proof of Theorem~\ref{thm:rktwo1}(i)}
By \eqref{eq:TEF}, the actions for $\tTD'_{i,-1}$ are given by
\begin{align}
\tTD'_{i,-1}(F_j)=y_{i,j;-c_{ij}},\qquad\tTD'_{i,-1}(E_{\tau j})=x_{i,\tau j;-c_{ij}}\qquad \tTD'_{i,e}( K_j') = K_j' (K_i')^{-c_{ij}}.
\end{align}
Recall the rescaled symmetries $\tT'_{i,-1}$ from \eqref{def:tT}. In this case, since $\tau i=i=\bw i$, $\vs_{i,\dm}=-q_i^{-2}$. Then we have
\begin{align}
\label{eq:tTBj}
\begin{split}
\tT'_{i,-1}(F_j)&=y_{i,j;-c_{ij}},\\
 \tT'_{i,-1}\big(\tT_{\bw}(E_{\tau j}) K_j'\big)&=(-1)^{c_{ij}} q_i^{ -2c_{ij}} \tT_{\bw}(x_{i,\tau j;-c_{ij}})K_j' (K_i')^{-c_{ij}},
 \end{split}
\end{align}
 where the second formula follows from that $\tT_{\bw} \tT'_{i,-1}=\tT'_{i,-1}\tT_{\bw}$.

 We also have analogous formulas for $\tT''_{i,+1}$
\begin{align}
\label{eq:tTBj'}
\begin{split}
\tT''_{i,+1}(F_j)&=y'_{i,j;-c_{ij}},\\
 \tT''_{i,+1}\big(\tT_{\bw}(E_{\tau j}) K_j'\big)&=(-1)^{c_{ij}} q_i^{ -2c_{ij}} \tT_{\bw}(x'_{i,\tau j;-c_{ij}})K_j' (K_i')^{-c_{ij}}.
 \end{split}
\end{align}

Recall elements $\B_{i,j;m},\b_{i,j;m}$ defined in Definitions~\ref{def:s}-\ref{def:s'}.

\begin{theorem}
\label{thm:split}
Let $ j\neq i\in \wI$.
\begin{itemize}
\item[(1)] The element $\B_{i,j;-c_{ij}}\in \tUi$ satisfies
\begin{align}
\B_{i,j;-c_{ij}} \tfX_i=\tfX_i \tT'_{i,-1}(B_j).
\end{align}

\item[(2)] The element $\b_{i,j;-c_{ij}}\in \tUi$ satisfies
\begin{align}
\b_{i,j;-c_{ij}}\tT_i(\tfX_i)^{-1} =\tT_i(\tfX_i)^{-1} \tT''_{i,+1}(B_j).
\end{align}

\item[(3)] $\b_{i,j;-c_{ij}}=\sigma^\imath(\B_{i,j;-c_{ij}})$.
\end{itemize}
\end{theorem}

In other word, the element $\tTa{i}(B_j):=\B_{i,j;-c_{ij}}$ satisfies \eqref{eq:inter2} and $\tTb{i}(B_j):=\b_{i,j;-c_{ij}}$ satisfies \eqref{eq:inter3}. Hence, we have proved the first statement in Theorem~\ref{thm:rktwo1}(i).

\begin{proof}
We prove (1). By Lemma~\ref{lem:rktwo1} and \eqref{eq:tTBj}, we have
\begin{align*}
\tfX_i \tT'_{i,-1}(B_j)=&  \tfX_i \tT'_{i,-1}(F_j)+\tfX_i \tT'_{i,-1}(E_j K_j')\\
=& \tfX_i \tT'_{i,-1}(F_j)+\tT'_{i,-1}(E_j K_j') \tfX_i\\
=&\tfX_i y_{i,j;-c_{ij}}+(-1)^{c_{ij}} q_i^{ -2c_{ij}} \tT_{\bw}(x_{i,\tau j;-c_{ij}})K_j' (K_i')^{-c_{ij}}\tfX_i.
\end{align*}
On the other hand, setting $m=-c_{ij}$ in Proposition~\ref{prop:By}-\ref{prop:Bx}, we have
\begin{align*}
\B_{i,j;-c_{ij}}^- \tfX_i=\tfX_i y_{i,j;-c_{ij}},\qquad  \B_{i,j;-c_{ij}}^+ =(-1)^{c_{ij}} q_i^{ -2c_{ij}} \tT_{\bw}(x_{i,\tau j;-c_{ij}})K_j' (K_i')^{-c_{ij}}.
\end{align*}
Therefore, by the above two formulas, we have
\begin{align*}
\tfX_i \tT'_{i,-1}(B_j)=& \B_{i,j;-c_{ij}}^- \tfX_i+ \B_{i,j;-c_{ij}}^+ \tfX_i = \B_{i,j;-c_{ij}} \tfX_i.
\end{align*}

(2) is proved by similar arguments above, using intertwining relations in Proposition~\ref{prop:B'y}-\ref{prop:B'x} and \eqref{eq:tTBj'}.

(3) is a consequence of Proposition~\ref{prop:bB}.
\end{proof}

%
%
\section{Higher rank formulas for $c_{i,\tau i}=0,\bw i=i$}
\label{sec:qs}
Throughout this section, we fix $i\in \fwItau$ such that $c_{i,\tau i}=0,\bw i=i$. Since $\tau $ commutes with $\bw$, we also have $\bw \tau i=\tau i$. In this case,
  we have $q_i=q_{\tau i}, B_i=F_i+E_{\tau i} K_i', \bs_i=s_i s_{\tau i}.$

We define higher rank root vectors $\B_{i,\tau i,j;m_1,m_2},\b_{i,\tau i,j;m_1,m_2}\in \tUi$ in Definitions~\ref{def:qsBij}-\ref{def:qsBij'} via recursive relations.
The divided power formulations for these elements are obtained in Proposition~\ref{thm:dvij}.
We show that $\B_{i,\tau i,j;-c_{ij},-c_{\tau i,j}},\b_{i,\tau i,j;-c_{ij},-c_{\tau i,j}}$ provide the higher rank formulas for $\tTa{i}(B_j),\tTb{i}(B_j)$ in Theorem~\ref{thm:qsbraid} and complete the proof for Theorem~\ref{thm:rktwo1}(ii).

\subsection{Definitions of root vectors}

\begin{definition}
\label{def:qsxy}
Define elements $y_{i,\tau i, j;m_1,m_2},x_{i,\tau i, j;m_1,m_2}$ for $m_1,m_2\geq 0,j\neq i,\tau i, j\in \wI$ as follows
\begin{align*}
y_{i,\tau i, j;m_1,m_2}&=\sum_{r=0}^{m_1}\sum_{s=0}^{m_2} (-1)^{r+s}  q_i^{-r(m_1 +c_{ij}-1) }q_{\tau i}^{-s(m_2 +c_{\tau i,j}-1) } F_i^{(m_1 -r)}F_{\tau_i}^{(m_2 -s)} F_j  F_{\tau i}^{(s)}F_i^{(r)}\\
x_{i,\tau i, j;m_1,m_2}&=\sum_{r=0}^{m_1}\sum_{s=0}^{m_2} (-1)^{r+s} q_i^{r(m_1+c_{ij}-1) }q_{\tau i}^{s(m_2+c_{\tau i,j}-1) } E_i^{(r)}E_{\tau i}^{(s)} E_j  E_{\tau i}^{(m_2-s)}E_i^{(m_1-r)}.
\end{align*}
\end{definition}

\begin{remark}
\label{rmk:qsxy}
Recall the elements $y_{i,j;m},x_{i,j;m}$ in Section~\ref{sec:split}. We have
\begin{align*}
y_{i,\tau i, j;m_1,m_2}&=\sum_{r=0}^{m_1} (-1)^{r}  q_i^{-r(m_1 +c_{ij}-1) }F_i^{(m_1 -r)} y_{\tau i,j;m_2} F_i^{(r)}\\
&=\sum_{s=0}^{m_2} (-1)^{s}  q_i^{-s(m_2 +c_{\tau i,j}-1) }F_{\tau i}^{(m_2 -s)} y_{  i,j;m_1} F_{\tau i}^{(s)}.
\end{align*}
In particular, $y_{i,\tau i, j;m ,0}=y_{i,j;m}$ and $y_{i,\tau i, j; 0,m }=y_{\tau i,j;m}$.
Similarly, we have $x_{i,\tau i, j;m,0}=x_{i,j;m}$ and $x_{i,\tau i, j; 0,m}=x_{ \tau i,j;m}$.
\end{remark}

\begin{definition}
\label{def:xy'}
Define elements $y'_{i,\tau i, j;m_1,m_2},x'_{i,\tau i, j;m_1,m_2}$ for $m_1,m_2\geq 0,j\neq i,\tau i,j\in \wI$ as follows
\begin{align*}
y'_{i,\tau i, j;m_1,m_2}&=\sigma(y_{i,\tau i, j;m_1,m_2}),\qquad x'_{i,\tau i, j;m_1,m_2}= \sigma(x_{i,\tau i, j;m_1,m_2}).
\end{align*}
\end{definition}

Set $y_{i,\tau i, j;m_1,m_2}=0$ and $x_{i,\tau i, j;m_1,m_2}=0$, if $m_1<0$ or $m_2<0$. Similar for $y'_{i,\tau i, j;m_1,m_2},x'_{i,\tau i, j;m_1,m_2}$.

\begin{lemma}
\label{lem:Lus-qs1}
We have, for $j\neq i,\tau i,j\in \wI$,$m_1,m_2\in \Z$
\begin{itemize}
\item[(1)] $ -q_i^{-(c_{ij}+2m_1)} y_{i,\tau i, j;m_1,m_2} F_i+ F_i y_{i,\tau i, j;m_1,m_2}=[m_1+1]_i\; y_{i,\tau i, j;m_1+1,m_2}$.
\item[(2)] $ -q_{\tau i}^{-(c_{\tau i,j}+2m_2)} y_{i,\tau i, j;m_1,m_2} F_{\tau i}+ F_{\tau i} y_{i,\tau i, j;m_1,m_2}=[m_2+1]_{\tau i}\; y_{i,\tau i, j;m_1,m_2+1}$.
\item[(3)] $ -y_{i,\tau i,j;m_1,m_2} E_i+ E_i y_{i,\tau i,j;m_1,m_2}= [-c_{ij}-m_1+1]_i \;  y_{i,\tau i,j;m_1 -1, m_2} K'_{i}$.
\item[(4)] $ -y_{i,\tau i,j;m_1,m_2} E_{\tau i}+ E_{\tau i} y_{i,\tau i,j;m_1,m_2}= [-c_{\tau i,j}-m_2+1]_{\tau i} \; y_{i,\tau i,j;m_1, m_2-1} K'_{\tau i}$.
\end{itemize}
\end{lemma}

\begin{proof}
Clearly, $[F_{\tau i}, E_i]=[F_i,E_{\tau i}]=0$. Since $c_{i,\tau i}=0$, we have $[F_i,F_{\tau i}]=0$. Then these four identities are immediate consequences of Lemma~\ref{lem:Lusxy} and Remark~\ref{rmk:qsxy}.
\end{proof}

\begin{lemma}
\label{lem:Lus-qs2}
We have, for $j\neq i,\tau i,j\in \wI$,$m_1,m_2\in \Z$
\begin{itemize}
\item[(1)] $ -q_i^{c_{ij}+2m_1} E_i x_{i,\tau i, j;m_1,m_2}+ x_{i,\tau i,j;m_1,m_2}E_i=[m_1+1]_i  x_{i,\tau i,j;m_1+1,m_2}$.
\item[(2)] $ -q_{\tau i}^{c_{\tau i,j}+2m_2} E_{\tau i} x_{i,\tau i,j;m_1,m_2}+ x_{i,\tau i,j;m_1,m_2}E_{\tau i}=[m_2+1]_{\tau i}  x_{i,\tau i,j;m_1,m_2+1}$.
\item[(3)] $ -F_i  x_{i,\tau i,j;m_1,m_2} +  x_{i,\tau i,j;m_1,m_2}  F_i= [-c_{ij}-m_1+1]_i \; K_{i}  x_{i,\tau i,j;m_1 -1,m_2} $.
\item[(4)] $ -F_{\tau i}  x_{i,\tau i,j;m_1,m_2} +  x_{i,\tau i,j;m_1,m_2}  F_{\tau i}= [-c_{\tau i,j}-m_2+1]_i \; K_{\tau i}  x_{i,\tau i,j;m_1,m_2-1} $.
\end{itemize}
\end{lemma}

\begin{proof}
By definition $x_{i,\tau i,j;m_1,m_2}=\sigma \omega\psi (y_{i,\tau i,j;m_1,m_2})$. These four identities are obtained by applying $\sigma \omega\psi$ to those four identities in Lemma~\ref{lem:Lus-qs1}.
\end{proof}

\begin{definition}
\label{def:qsBij}
Let $j\neq i,\tau i,j\in \wI$. Define $\B_{i,\tau i, j;m_1,m_2}^\pm $ for $m_1,m_2 \geq -1$ to be the elements in $\tU$ determined by the following two recursive relations,
 \begin{align}\notag
&-q_i^{-(c_{ij}+2m_1)} \B_{i,\tau i, j;m_1,m_2}^\pm B_i + B_i \B_{i,\tau i, j;m_1,m_2}^\pm\\
=&[m_1+1]_i\; \B_{i,\tau i, j;m_1+1,m_2}^\pm+q_i^{-(c_{ij}+2m_1)} [-c_{\tau i,j}-m_2+1]_{i} \; \B_{i,\tau i, j;m_1,m_2-1}^\pm \tk_i,
  \label{def:qsBij1}
  \\\notag
&-q_i^{-(c_{\tau i,j}+2m_2)}\B_{i,\tau i, j;m_1,m_2}^\pm B_{\tau i} + B_{\tau i} \B_{i,\tau i, j;m_1,m_2}^\pm \\
=&[m_2+1]_{i}\; \B_{i,\tau i, j;m_1,m_2+1}^\pm +q_i^{-(c_{\tau i,j}+2m_2)} [-c_{ i,j}-m_1+1]_{i} \; \B_{i,\tau i, j;m_1-1,m_2}^\pm \tk_{\tau i},
\label{def:qsBij2}
 \end{align}
 where we set
\begin{align}
\label{def:qsBij3}
 \B_{i,\tau i, j;m_1,-1}^\pm=\B_{i,\tau i, j;-1,m_2}^\pm=0,\qquad \B_{i,\tau i, j;0,0}^-= F_j,\qquad \B_{i,\tau i, j;0,0}^+=\tT_{\bw}(E_{\tau j}) K_j'.
\end{align}
\end{definition}

Set $\B_{i,\tau i,j;m_1,m_2}=\B_{i,\tau i,j;m_1,m_2}^- +\B_{i,\tau i,j;m_1,m_2}^+$. Since $\B_{i,\tau i,j,0,0}=B_j\in \tUi$, one can inductively show that $\B_{i,\tau i,j;m_1,m_2}\in \tUi$ for $m_1,m_2\geq 0$.

\begin{definition}
\label{def:qsBij'}
Let $j\neq i,\tau i,j\in \wI$. Define $\b_{i,\tau i, j;m_1,m_2}^\pm $ for $m_1,m_2\geq -1$ to be the elements in $\tU$ determined by the following two recursive relations,
 \begin{align}\notag
&-q_i^{-(c_{ij}+2m_1)} B_i\b_{i,\tau i, j;m_1,m_2}^\pm + \b_{i,\tau i, j;m_1,m_2}^\pm B_i\\
=&[m_1+1]_i\; \b_{i,\tau i, j;m_1+1,m_2}^\pm + q_i^{-(c_{\tau i,j}+2m_2-2)} [-c_{\tau i,j}-m_2+1]_{i} \;\b_{i,\tau i, j;m_1,m_2-1}^\pm\tk_{\tau i},
  \label{def:qsBij'1}
  \\\notag
&-q_i^{-(c_{\tau i,j}+2m_2)}B_{\tau i} \b_{i,\tau i, j;m_1,m_2}^\pm + \b_{i,\tau i, j;m_1,m_2}^\pm B_{\tau i}\\
=&[m_2+1]_{i}\; \b_{i,\tau i, j;m_1,m_2+1}^\pm +q_i^{-(c_{ij}+2m_1-2)} [-c_{ i,j}-m_1+1]_{i} \;\b_{i,\tau i, j;m_1-1,m_2}^\pm \tk_i,
\label{def:qsBij'2}
 \end{align}
 where we set
\begin{align}
\label{def:qsBij'3}
 \b_{i,\tau i, j;m_1,-1}^\pm=\b_{i,\tau i, j;-1,m_2}^\pm=0,\qquad \b_{i,\tau i, j;0,0}^-= F_j,\qquad \b_{i,\tau i, j;0,0}^+=\tT_{\bw}(E_{\tau j}) K_j'.
\end{align}
\end{definition}

Set $\b_{i,\tau i,j;m_1,m_2}=\b_{i,\tau i,j;m_1,m_2}^- +\b_{i,\tau i,j;m_1,m_2}^+$. One can also inductively show that $\b_{i,\tau i,j;m_1,m_2}\in \tUi$ for $m_1,m_2\geq 0$.

Recall the anti-involution $\sigma^\imath$ on $\tUi$ from Proposition~\ref{prop:sigma}.
\begin{proposition}
\label{prop:qsbB}
Let $j\in \wI,j\neq i,\tau i$. Then $\b_{i,j;m_1,m_2}=\sigma^\imath ( \B_{i,j;m_1,m_2})$ for $m_1,m_2\geq 0$.
\end{proposition}

\begin{proof}
By Definition~\ref{def:qsBij}-\ref{def:qsBij'}, it is clear that $\sigma^\imath ( \B_{i,j;m_1,m_2})$ satisfies the same recursive relations as $\b_{i,j;m_1,m_2}$. Since $\b_{i,j;0,0}=B_j=\sigma^\imath ( \B_{i,j;0,0})$, this proposition follows by induction.
\end{proof}

\begin{lemma}
\label{lem:qsBB}
Let $j\in \wI, j\neq i,\tau i$. We have, for $m_1,m_2\geq 0,$
\begin{align*}
\b^-_{i,\tau i, j;m_1,m_2}&=\sigma \big( \tfX_i^{-1} \B^-_{i,\tau i, j;m_1,m_2} \tfX_i \big).
\end{align*}
\end{lemma}

\begin{proof}
Consider the subalgebra $\tU_{[i;j]}^-$ of $\tU$ generated by $B_i,B_{\tau i},\tk_i,\tk_{\tau i},F_j$. It is clear from the above definitions that $\B^-_{i,\tau i, j;m_1,m_2},\b^-_{i,\tau i, j;m_1,m_2}\in \tU_{[i;j]}^-$. By Proposition~\ref{prop:fX1} and Lemma~\ref{lem:rktwo1}, there is a well-defined anti-automorphism $\sigma_{ij}$ on $\tU_{[i;j]}^-$, which is given by
\begin{align*}
\sigma_{ij}: x \mapsto \sigma \big( \tfX_i^{-1} x \tfX_i \big).
\end{align*}
Moreover, $\sigma_{ij}$ fixes $B_i,B_{\tau i},F_j$ and sends $\tk_i\leftrightarrow \tk_{\tau i}$. Applying $\sigma_{ij}$ to \eqref{def:qsBij1}-\eqref{def:qsBij2}, it is clear that $ \sigma_{ij} (\B^-_{i,\tau i, j;m_1,m_2})$ satisfies the same recursive relations as $\b^-_{i,\tau i, j;m_1,m_2}$. Then the desired identity follows by induction.
\end{proof}

\subsection{A divided power formulation}

In this subsection, we derive the formulas of root vectors $\B_{i,\tau i,j; m_1,m_2},\b_{i,\tau i,j; m_1,m_2}$, in terms of divided powers of generators of $\tUi$, from their recursive Definitions~\ref{def:qsBij}-\ref{def:qsBij'}.
Denote $$[\tk_i;a]:=\frac{\tk_i q_i^a-\tk_{\tau i} q_i^{-a}}{q_i -q_i^{-1}}.$$

\begin{lemma}
\label{lem:J95B}
We have for any $m\geq 0$,
\begin{align*}
B_{\tau i} B_i^{(m)}-B_i^{(m)} B_{\tau i}= B_i^{(m-1)} [\tk_i;1-m].
\end{align*}
\end{lemma}

\begin{proof}
For $c_{i,\tau i}=0$, $B_i,B_{\tau i}$ satisfy the following relation
\begin{align*}
[B_{\tau i},B_i] = \frac{\tk_i  -\tk_{\tau i} }{q_i -q_i^{-1}}.
\end{align*}
Using this relation, one can then prove this lemma by induction on $m$ c.f. \cite[\S 1.3]{Ja95}.
\end{proof}

\begin{proposition}
\label{thm:dvij}
Let $j\in \wI, j\neq i,\tau i$. The elements $\B_{i,\tau i,j; m_1,m_2},\b_{i,\tau i,j; m_1,m_2}$ defined in Definitions~\ref{def:qsBij}-\ref{def:qsBij'} admit following formulas
\begin{align}\notag
\B_{i,\tau i,j; m_1,m_2}=&\sum_{u=0}^{\min(m_1,m_2)} \sum_{r=0}^{m_1-u}\sum_{s=0}^{m_2-u} (-1)^{r+s+u}q_i^{r(\alpha-m_1-u+1)+s(\beta-m_2+u+1)+u(\alpha-m_1+1)}\times
\\& \times \qbinom{\beta-m_2+u}{u}_i B_i^{(m_1-r-u)}B_{\tau i}^{(m_2-s-u)}B_j B_{\tau i}^{(s)}B_i^{(r)} \tk_i^u .
\label{def:dvij}
\\\notag
\b_{i,\tau i,j; m_1,m_2}=&\sum_{u=0}^{\min(m_1,m_2)} \sum_{r=0}^{m_1-u}\sum_{s=0}^{m_2-u} (-1)^{r+s+u}q_i^{r(\alpha-m_1-u+1)+s(\beta-m_2+u+1)+u(\alpha-m_1+1)}\times
\\& \times \qbinom{\beta-m_2+u}{u}_i \tk_{\tau i}^u B_i^{(r)}B_{\tau i}^{(s)}B_j B_{\tau i}^{(m_2-s-u)}B_i^{(m_1-r-u)}.
\label{def:dvij2}
\end{align}
\end{proposition}

\begin{proof}
The second formula \eqref{def:dvij2} is obtained by applying $\sigma_\imath$ to the first formula. Hence, it suffices to show that the elements $\B_{i,\tau i,j; m_1,m_2}$ defined by \eqref{def:dvij} satisfy the recursive relations \eqref{def:qsBij1}-\eqref{def:qsBij2}.

(1) Indeed, we have 
\begin{align*}
& \text{LHS }\eqref{def:qsBij1}-\text{ RHS }\eqref{def:qsBij1}
\\
=&-q_i^{\alpha-2m_1} \B_{i,\tau i, j;m_1,m_2}  B_i + B_i \B_{i,\tau i, j;m_1,m_2}
 \\
 & -[m_1+1]_i\; \B_{i,\tau i, j;m_1+1,m_2}-q_i^{\alpha-2m_1} [\beta-m_2+1]_{i} \; \B_{i,\tau i, j;m_1,m_2-1} \tk_i
\\
= &\sum_{u=0}^{ \min(m_1,m_2)} \sum_{r=0}^{m_1-u}\sum_{s=0}^{m_2-u} (-1)^{r+s+u} q_i^{r(\alpha-m_1-u+1)+s(\beta-m_2+u+1)+u(\alpha-m_1+1)} \varepsilon_{r,s,u}
\\& \times \qbinom{\beta-m_2+u}{u}_i B_i^{(m_1-r-u)}B_{\tau i}^{(m_2-s-u)}B_j B_{\tau i}^{(s)}B_i^{(r)} \tk_i^u .
\end{align*}
where the scalar $ \varepsilon_{r,s,u}$ is given by
\begin{align*}
\varepsilon_{r,s,u}=&[r]_i q_i^{-2u-(\alpha-m_1-u+1)+\alpha-2m_1} +[m_1-r-u+1]_i-[m_1+1]_i q_i^{-r-u}
\\
&+[u]_i q_i^{-(\alpha-m_1+1)+r}
\\
=&0.
\end{align*}
Hence, we have proved that $\B_{i,\tau i,j; m_1,m_2}$ defined by \eqref{def:dvij} satisfy \eqref{def:qsBij1}.

(2) We next show that $\B_{i,\tau i,j; m_1,m_2}$  satisfy \eqref{def:qsBij2}. By Lemma~\ref{lem:J95B}, we have
\begin{align*}
& \text{LHS }\eqref{def:qsBij2}-\text{ RHS }\eqref{def:qsBij2}
\\
=&-q_i^{\beta-2m_2 }\B_{i,\tau i, j;m_1,m_2}  B_{\tau i} + B_{\tau i} \B_{i,\tau i, j;m_1,m_2}
\\
 &-[m_2+1]_{ i}\; \B_{i,\tau i, j;m_1,m_2+1} -q_i^{\beta- 2m_2} [\alpha-m_1+1]_{i} \; \B_{i,\tau i, j;m_1-1,m_2}  \tk_{\tau i}
\\
=&\sum_{u\geq 0}  \sum_{r=0}^{m_1-u}\sum_{s=0}^{m_2+1-u} (-1)^{r+s+u}q_i^{r(\alpha-m_1-u+1)+s(\beta-m_2+u+1)+u(\alpha-m_1+1)} \xi_{r,s,u}
\\
& \qquad\times \qbinom{\beta-m_2+u}{u}_i B_i^{(m_1-r-u)}B_{\tau i}^{(m_2+1-s-u)}B_j B_{\tau i}^{(s)}B_i^{(r)} \tk_i^u
\\
&+\sum_{u\geq 0}  \sum_{r=0}^{m_1-1-u}\sum_{s=0}^{m_2-u} (-1)^{r+s+u}q_i^{r(\alpha-m_1-u+1)+s(\beta-m_2+u+1)+u(\alpha-m_1+1)} \xi'_{r,s,u}
\\& \qquad\times \qbinom{\beta-m_2+u}{u}_i B_i^{(m_1-1-r-u)}B_{\tau i}^{(m_2 -s-u)}B_j B_{\tau i}^{(s)}B_i^{(r)} \tk_i^u \tk_{\tau i}.
\end{align*}
where the scalars $\xi_{r,s,u},\xi'_{r,s,u}$ are given by
\begin{align*}
 \xi_{r,s,u}=& [s]_i q_i^{-(\beta-m_2+u+1)+2u+\beta-2m_2}+[m_2+1-s-u]_i -q_i^{-s} [m_2+1]_i \frac{[\beta-m_2]_i}{[\beta-m_2+u ]_i}
 \\
 &+\frac{q_i^{\beta+u-2m_2-s-1}-q_i^{-\beta-u+2m_2-s+1}}{q_i-q_i^{-1}}\frac{[u]_i}{[\beta-m_2+u]_i}
 \\
 =&q_i^{-s}\frac{[m_2+1-u]_i[\beta-m_2+u]_i -[m_2+1]_i[\beta-m_2]_i+[\beta+u-2m_2-1]_i[u]_i}{[\beta-m_2+u ]_i}
 \\
 =&0,
 \\
 \xi'_{r,s,u}=&-q_i^{\beta- 2m_2+r+u} [\alpha-m_1+1]_i +\frac{q_i^{\beta+\alpha+r+u-m_1-2m_2+1}-q_i^{\beta-\alpha+r+u+m_1-2m_2-1}}{q_i -q_i^{-1}}
 \\
 =&0.
\end{align*}
Hence, we have proved that $\B_{i,\tau i,j; m_1,m_2}$ defined by \eqref{def:dvij} satisfy \eqref{def:qsBij2}.
\end{proof}



\begin{remark}
By Remark~\ref{rmk:qsxy} and Propositions~\ref{prop:qsBy}-\ref{prop:qsBx},\ref{prop:qsBy'}-\ref{prop:qsBx'}, we have
\begin{align*}
\B_{i,\tau i,j;m_1,m_2}=0, \quad \b_{i,\tau i,j;m_1,m_2}=0,\quad \text{ if } m_1>-c_{ij}, \quad \text{or } m_2>-c_{\tau i,j}.
\end{align*}
Furthermore, according to the divided power formulations, the Serre relations in $\tUi$ are given by
\begin{align}
\B_{i,\tau i,j;-c_{ij}+1,0}=0,\qquad \B_{i,\tau i,j;0,-c_{\tau i,j}+1}=0,\qquad j\neq i,\tau i.
\end{align}
\end{remark}

\subsection{Intertwining properties}
\label{sec:qsinter}
We establish precise intertwining relations between those elements $\B_{i,\tau i,j;m_1,m_2}^\pm$ ({\em resp.} $\b_{i,\tau i,j;m_1,m_2}^\pm$) and elements $y_{i,\tau i, j;m_1,m_2},x_{i,\tau i, j;m_1,m_2}$ ({\em resp.} $y'_{i,\tau i, j;m_1,m_2},x'_{i,\tau i, j;m_1,m_2}$). These relations will be the key for the construction of relative braid group action on $\tUi$.

\begin{proposition}
\label{prop:qsBy}
Let $j\in \wI$ such that $j\neq i,\tau i$. We have, for $m_1,m_2\geq 0, $
\begin{align}
\label{eq:qsBy1}
\B_{i,\tau i, j;m_1,m_2}^- \tfX_i=\tfX_i y_{i,\tau i, j;m_1,m_2}.
\end{align}
\end{proposition}

\begin{proof}

Let $R_{m_1,m_2} $ denote $\tfX_i y_{i,\tau i, j;m_1,m_2}\tfX_i^{-1}$. By Lemma~\ref{lem:rktwo1}, $R_{0,0}=F_j=\B_{i,\tau i, j;0,0}^-$. Moreover, by definition, $y_{i,\tau i, j;m_1,-1}=y_{i,\tau i, j;-1,m_2}=0=R_{m_1,-1}=R_{-1,m_2}$. Hence, it suffices to prove that $R_{m_1,m_2}$ satisfies the same recursive relations as $\B_{i,\tau i, j;m_1,m_2}^-$.

Recall that $B_i^\sigma=F_i+K_i E_{\tau i}$ . We have, by Proposition~\ref{prop:fX1},
\begin{align*}
&\tfX_i^{-1}\big(-q_i^{-(c_{ij}+2m_1)} R_{m_1,m_2} B_i + B_i R_{m_1,m_2} \big)\tfX_i\\
=& -q_i^{-(c_{ij}+2m_1)} y_{i,\tau i, j;m_1,m_2} B_i^\sigma + B_i^\sigma y_{i,\tau i, j;m_1,m_2}\\
=& -q_i^{-(c_{ij}+2m_1)} y_{i,\tau i, j;m_1,m_2} (F_i+K_i E_{\tau i}) + (F_i+K_i E_{\tau i}) y_{i,\tau i, j;m_1,m_2}\\
=& -q_i^{-(c_{ij}+2m_1)} y_{i,\tau i, j;m_1,m_2}  F_i + F_i  y_{i,\tau i, j;m_1,m_2}\\
 & +q_i^{-(c_{ij}+2m_1)} (-y_{i,\tau i, j;m_1,m_2} E_{\tau i}+ E_{\tau i} y_{i,\tau i, j;m_1,m_2})K_i.
\end{align*}
Now using Lemma~\ref{lem:Lus-qs1} to simplify the RHS of above formula, we have
\begin{align*}
\text{RHS}=& [m_1+1]_i\; y_{i,\tau i, j;m_1+1,m_2}+\\
 & +q_i^{-(c_{ij}+2m_1)} [-c_{\tau i,j}-m_2+1]_{\tau i} \; y_{i,\tau i,j;m_1, m_2-1} K_i K'_{\tau i}.
\end{align*}
Combining the above two formulas, we have
\begin{align}\notag
&-q_i^{-(c_{ij}+2m_1)} R_{m_1,m_2} B_i + B_i R_{m_1,m_2}\\
=&[m_1+1]_i\; R_{m_1+1,m_2}+q_i^{-(c_{ij}+2m_1)} [-c_{\tau i,j}-m_2+1]_{i} \; R_{m_1,m_2-1} K_i K'_{\tau i}.
\label{eq:qsBy2}
\end{align}
The following variant of \eqref{eq:qsBy2} can be obtained by a similar strategy
\begin{align}\notag
&-q_i^{-(c_{\tau i,j}+2m_2)} R_{m_1,m_2} B_{\tau i} + B_{\tau i} R_{m_1,m_2}\\
=&[m_2+1]_{i}\; R_{m_1,m_2+1}+q_i^{-(c_{\tau i,j}+2m_2)} [-c_{i,j}-m_1+1]_{i} \; R_{m_1-1,m_2} K'_i K_{\tau i}.
\label{eq:qsBy3}
\end{align}
Comparing \eqref{eq:qsBy2}-\eqref{eq:qsBy3} with \eqref{def:qsBij1}-\eqref{def:qsBij2}, it is clear that $R_{m_1,m_2}$ satisfies the same recursive relations as $\B_{i,\tau i, j;m_1,m_2}^-$. Therefore, we have proved \eqref{eq:qsBy1}.
\end{proof}

\begin{proposition}
\label{prop:qsBx}
Let $j\in \wI$ such that $j\neq i,\tau i$. We have, for $m_1,m_2\geq 0$,
\begin{align}\notag
\B_{ i, \tau i, j;m_1,m_2}^+ =&(-1)^{m_1+m_2} q_i^{-(m_1+m_2)(c_{ij}+c_{\tau i,j}+m_1+m_2-1)} \times
\\
&\qquad\times\tT_{\bw}(x_{\tau i,i, \tau j;m_1,m_2}) K'_j (K_i')^{m_1} (K_{\tau i}')^{m_2}.
\label{eq:qsBx1}
\end{align}
(Note the shift of indices on the right-hand side.)
\end{proposition}

\begin{proof}
Let $P_{m_1,m_2}$ denote the RHS \eqref{eq:qsBx1} and $x_{\tau i,i, \bw\tau j;m_1,m_2}$ denote $\tT_{\bw}(x_{\tau i,i, \tau j;m_1,m_2})$. By definition \eqref{def:qsBij3}, $\B_{ i, \tau i, j;0,0}^+=\tT_{\bw}(E_{\tau j} )K_j'= x_{\tau i,i, \bw\tau j;0,0}  K'_j=P_{0,0}$. Moreover, by definition, $x_{\tau i,i, \tau j;-1,m_2}=x_{\tau i,i, \tau j;m_1,-1}=0$ and $\B_{ i, \tau i, j;-1,m_2}^+=\B_{ i, \tau i, j;m_1,-1}^+=0$. Thus, it suffices to show that $P_{m_1,m_2}$ satisfies the same recursive relations as $\B_{ i, \tau i, j;m_1,m_2}^+ $.

Let $Q_{m_1,m_2}$ denote $x_{\tau i,i, \bw\tau j;m_1,m_2} K'_j (K_i')^{m_1} (K_{\tau i}')^{m_2}$. We first formulate the recursive relations for $Q_{m_1,m_2}$. We have
\begin{align*}
&\quad-q_i^{-(c_{ij}+2m_1)} Q_{m_1,m_2} B_i + B_i Q_{m_1,m_2} \\
&=-q_i^{-(c_{ij}+2m_1)}x_{\tau i,i, \bw\tau j;m_1,m_2} K_j' (K_i')^{m_1} (K_{\tau i}')^{m_2} (F_i+E_{\tau i} K_i')\\
 &\quad+(F_i+E_{\tau i} K_i') x_{\tau i,i, \bw\tau j;m_1,m_2} K'_j (K_i')^{m_1} (K_{\tau i}')^{m_2}\\
&=-(x_{\tau i,i, \bw\tau j;m_1,m_2}F_i - F_i x_{\tau i,i, \bw\tau j;m_1,m_2})K_j' (K_i')^{m_1} (K_{\tau i}')^{m_2}\\
 &\quad-q_i^{-(c_{ij}+c_{\tau i,j}+2m_1+2m_2)}\big[x_{\tau i,i, \bw\tau j;m_1,m_2}, E_{\tau i} \big]_{q_i^{c_{i,j}+2 m_1}} K_j' (K_i')^{m_1+1} (K_{\tau i}')^{m_2}.
\end{align*}
Since $\bw i=i$, both $F_i,E_{\tau i}$ are fixed by $\tT_{\bw}$. Then the recursion involving $x_{\tau i,i, \bw\tau j;m_1,m_2}$ and $F_i$ ($resp.$ $ E_{\tau i}$) is the same as the recursion involving $x_{\tau i,i,\tau j;m_1,m_2}$ and  $F_i$ ($resp.$ $E_{\tau i}$). By Lemma~\ref{lem:Lus-qs2}, one can obtain those recursions for $x_{\tau i,i, \bw\tau j;m_1,m_2}$ and then the RHS of the above formula is simplified as below
\begin{align*}
\text{RHS} &=-[-c_{\tau i,j} -m_2+1]_i K_i x_{\tau i,i, \bw\tau j;m_1,m_2-1} K_j' (K_i')^{m_1} (K_{\tau i}')^{m_2}\\
&\quad -q_i^{-(c_{ij}+c_{\tau i,j}+2m_1+2m_2)}[m_1+1]_i x_{\tau i,i, \bw\tau j;m_1+1,m_2} K_j' (K_i')^{m_1+1} (K_{\tau i}')^{m_2}\\
&=-q_i^{2m_2-2+c_{\tau i,j}}[-c_{\tau i,j} -m_2+1]_i  x_{\tau i,i, \bw\tau j;m_1,m_2-1} K_j' (K_i')^{m_1} (K_{\tau i}')^{m_2-1}\tk_{ i}\\
&\quad -q_i^{-(c_{ij}+c_{\tau i,j}+2m_1+2m_2)}[m_1+1]_i x_{\tau i,i,\bw \tau j;m_1+1,m_2} K_j' (K_i')^{m_1+1} (K_{\tau i}')^{m_2}.
\end{align*}
Combining the above two formulas, we have
\begin{align*}
&-q_i^{-(c_{ij}+2m_1)} Q_{m_1,m_2} B_i + B_i Q_{m_1,m_2} \\
=&-q_i^{-(c_{ij}+c_{\tau i,j}+2m_1+2m_2)}[m_1+1]_i Q_{m_1+1,m_2}+q_i^{2m_2-2+c_{\tau i,j}}[c_{\tau i,j} +m_2-1]_i Q_{m_1,m_2-1}\tk_{ i}.
\end{align*}
Hence, $P_{m_1,m_2}$ satisfies the following recursive relation,
\begin{align}\notag
&-q_i^{-(c_{ij}+2m_1)} P_{m_1,m_2} B_i + B_i P_{m_1,m_2} \\
=& [m_1+1]_i P_{m_1+1,m_2}+q_i^{-(c_{ij}+2m_1) }[-c_{\tau i,j} -m_2+1]_i P_{m_1,m_2-1}\tk_{ i}.
\label{eq:qsBx2}
\end{align}
Comparing \eqref{eq:qsBx2} and \eqref{def:qsBij1}, it is clear that $P_{m_1,m_2}$ satisfies the defining recursive relation \eqref{def:qsBij1} for $\B_{ i, \tau i, j;m_1,m_2}^+ $. Using a similar strategy, one can show that $P_{m_1,m_2}$ satisfies the other defining recursive relation \eqref{def:qsBij2} for $\B_{ i, \tau i, j;m_1,m_2}^+ $. Therefore, $\B_{ i, \tau i, j;m_1,m_2}^+=P_{m_1,m_2}$ for any $m_1,m_2\geq 0$.
\end{proof}

We next formulate the relation between $y'_{i,\tau i, j;m_1,m_2},x'_{i,\tau i, j;m_1,m_2}$ in Definition~\ref{def:xy'} and $\b_{i,\tau i, j;m_1,m_2}^\pm$ in Definition~\ref{def:qsBij'}.
\begin{proposition}
\label{prop:qsBy'}
Let $j\in \wI$ such that $j\neq i,\tau i$. We have, for $m_1,m_2\geq 0$,
\begin{align}
\label{eq:qsBy'1}
\b_{i,\tau i, j;m_1,m_2}^-  =  y'_{i,\tau i, j;m_1,m_2}.
\end{align}
\end{proposition}

\begin{proof}
This proposition is a consequence of Proposition~\ref{prop:qsBy} and Lemma~\ref{lem:qsBB}.
\end{proof}

\begin{proposition}
\label{prop:qsBx'}
Let $j\in \wI$ such that $j\neq i,\tau i$. We have, for $m_1,m_2\geq 0$,
\begin{align}\notag
\b_{ i, \tau i, j;m_1,m_2}^+=&(-1)^{m_1+m_2} q_i^{-(m_1+m_2)(c_{ij}+c_{\tau i,j}+m_1+m_2-1)} \times\\
&\times\tT_{\bs_i}(\tfX_i)^{-1} \tT_{\bw}(x'_{\tau i,i, \tau j;m_1,m_2}) K'_j (K_i')^{m_1} (K_{\tau i}')^{m_2}\tT_{\bs_i}(\tfX_i).
\label{eq:qsBx'1}
\end{align}
\end{proposition}

\begin{proof}
Let $P_{m_1,m_2}$ denote the RHS \eqref{eq:qsBx'1} and $x'_{\tau i,i, \bw\tau j;m_1,m_2}$ denote $\tT_{\bw}(x'_{\tau i,i, \tau j;m_1,m_2}) $. By Lemma~\ref{lem:rktwo1}, $P_{0,0} =\tT_{\bw}(E_{\tau j}) K_j'=\b_{ i, \tau i, j;0,0}^+$. Moreover, by definition, we have $P_{-1,m}=P_{m,-1}=0$ and $\b_{ i, \tau i, j;-1,m_2}^+=\b_{ i, \tau i, j;m_1,-1}^+=0$. Hence, it suffices to show that $P_{m_1,m_2}$ satisfies the defining recursive relations for $\b_{ i, \tau i, j;m_1,m_2}^+$.

Applying $\sigma$ to Lemma~\ref{lem:Lus-qs2}(1)(4) and then shifting the indices $i,j$ to $\tau i,\tau j$, we obtain the recursions for $ x'_{\tau i,i,\tau j;m_1,m_2}$. Since $F_i,E_{\tau i}$ are fixed by $\tT_{\bw}$, recursions for $ x'_{\tau i,i,\bw\tau j;m_1,m_2}$ are the same as recursions for $ x'_{\tau i,i,\tau j;m_1,m_2}$. Thus, we have
\begin{align}
\label{eq:x'}
\begin{split}
&-  x'_{\tau i,i,\bw\tau j;m_1,m_2} F_{ i}+ F_{i} x'_{\tau i,i,\bw\tau j;m_1,m_2}= [-c_{\tau i,j}-m_2+1]_i \; x'_{\tau i,i,\bw\tau j;m_1,m_2-1} K'_{ i},\\
&-q_i^{c_{ij}+2m_1} x'_{\tau i, i,\bw\tau j;m_1,m_2} E_{\tau i}+ E_{\tau i} x'_{\tau i, i,\bw\tau j;m_1,m_2}=[m_1+1]_i  x'_{\tau i,i,\bw\tau j;m_1+1,m_2} .
\end{split}
\end{align}

Let $Q_{m_1,m_2}$ denote $\tT_{\bs_i}(\tfX_i)^{-1} x'_{\tau i,i, \bw\tau j;m_1,m_2} K'_j (K_i')^{m_1} (K_{\tau i}')^{m_2}\tT_{\bs_i}(\tfX_i)$. We first formulate the recursive relation for $Q_{m_1,m_2}$.

In the case $c_{i,\tau i}=0$, set
\begin{align}
\widehat{B}_i=-\tT_{\bs_i}(B_{\tau i}\tk_{\tau i}^{-1})=F_i K_{\tau i} K_{\tau i}'^{-1} + E_{\tau i} K_i'.
\end{align}
Then, due to \cite[\S 6.4]{WZ22}, $\widehat{B}_i$ satisfies $\widehat{B}_i\tT_{\bs_i}(\tfX_i)=\tT_{\bs_i}(\tfX_i) B_i$.

We compute
\begin{align*}
 &\tT_{\bs_i}(\tfX_i)\big(-q_i^{-(c_{ij}+2m_1)} B_i Q_{m_1,m_2}+Q_{m_1,m_2} B_i\big) \tT_{\bs_i}(\tfX_i)^{-1}\\
=&-q_i^{-(c_{ij}+2m_1)} \widehat{B}_i x'_{\tau i,i, \bw\tau j;m_1,m_2} K'_j (K_i')^{m_1} (K_{\tau i}')^{m_2}\\
 &+x'_{\tau i,i, \bw\tau j;m_1,m_2} K'_j (K_i')^{m_1} (K_{\tau i}')^{m_2} \widehat{B}_i\\
=&-q_i^{c_{ij}+2m_1} (F_i x'_{\tau i,i,\bw \tau j;m_1,m_2} -x'_{\tau i,i,\bw \tau j;m_1,m_2}F_i) K'_j (K_i')^{m_1-1} (K_{\tau i}')^{m_2-1} \tk_{\tau i}\\
 &-q_i^{-c_{\tau i,j}-2m_2} (q_i^{-c_{ij}-2m_1} E_{\tau i} x'_{\tau i,i, \bw\tau j;m_1,m_2}-x'_{\tau i,i, \bw\tau j;m_1,m_2} E_{\tau i}) K'_j (K_i')^{m_1+1} (K_{\tau i}')^{m_2}\\
=&-q_i^{c_{ij}+2m_1}[-c_{\tau i,j}-m_2+1]_i \; x'_{\tau i,i,\bw\tau j;m_1,m_2-1}  K'_j (K_i')^{m_1} (K_{\tau i}')^{m_2-1} \tk_{\tau i}\\
 &-q_i^{-c_{ij}-c_{\tau i,j}-2m_1-2m_2}[m+1]_i x'_{\tau i,i,\bw \tau j;m_1+1,m_2} K'_j (K_i')^{m_1+1} (K_{\tau i}')^{m_2},
\end{align*}
where the last equality follows by applying \eqref{eq:x'}.

The above computation shows that $Q_{m_1,m_2}$ satisfies the following recursive relation
\begin{align}
&-q_i^{-(c_{ij}+2m_1)} B_i Q_{m_1,m_2}+Q_{m_1,m_2} B_i\\\notag
&=-q_i^{c_{ij}+2m_1}[-c_{\tau i,j}-m_2+1]_i Q_{m_1,m_2-1}\tk_{\tau i}-q_i^{-c_{ij}-c_{\tau i,j}-2m_1-2m_2}[m+1]_i Q_{m_1+1,m_2}.
\end{align}
By definition, $P_{m_1,m_2}=(-1)^{m_1+m_2} q_i^{-(m_1+m_2)(c_{ij}+c_{\tau i,j}+m_1+m_2-1)} Q_{m_1,m_2}$. Hence, $P_{m_1,m_2}$ satisfies the following relation
\begin{align}
\label{eq:qsBx'2}
&-q_i^{-(c_{ij}+2m_1)} B_i P_{m_1,m_2}+P_{m_1,m_2} B_i\\\notag
&=[m+1]_i P_{m_1+1,m_2}+q_i^{-(c_{\tau i,j}+2m_2-2)}[-c_{\tau i,j}-m_2+1]_i P_{m_1,m_2-1}\tk_{\tau i}.
\end{align}
Comparing \eqref{eq:qsBx'2} with \eqref{def:qsBij'1}, it is clear that $P_{m_1,m_2}$ satisfies the defining recursive relation \eqref{def:qsBij'1} for $\b_{ i, \tau i, j;m_1,m_2}^+$. Using a similar strategy, one can show that $P_{m_1,m_2}$ satisfies the other defining recursive relation \eqref{def:qsBij'2} for $\b_{ i, \tau i, j;m_1,m_2}^+$. Therefore, we conclude that $\b_{ i, \tau i, j;m_1,m_2}^+=P_{m_1,m_2}$ for $m_1,m_2 \geq 0$.
\end{proof}

\subsection{Proof of Theorem~\ref{thm:rktwo1}(ii)}
In the case $c_{i,\tau i}=0$, $\bs_i=s_i s_{\tau i}$. It follows by \cite[\S 37.2]{Lus94} that, for $j\neq i,\tau i$,
\begin{align}
\label{eq:qstTF}
&\tTD'_{\bs_i,-1}(F_j)=y_{i,\tau i,j;-c_{ij},-c_{\tau i,j}},\qquad\tTD'_{\bs_i,-1}(E_{\tau j})=x_{\tau i, i,\tau j;-c_{ij},-c_{\tau i,j}},
\\
&\tTD''_{\bs_i,+1}(F_j)=y'_{i,\tau i,j;-c_{ij},-c_{\tau i,j}},\qquad\tTD''_{\bs_i,+1}(E_{\tau j})=x'_{\tau i, i,\tau j;-c_{ij},-c_{\tau i,j}}.
\label{eq:qstTF'}
\end{align}
Recall the rescaled symmetries $\tT'_{i,-1},\tT''_{i,+1}$ from \eqref{def:tT}. In the case $c_{i,\tau i}=0$, $\vs_{i,\dm}=-q_i^{-1}$. By \eqref{eq:qstTF}, we have
\begin{align}
\label{eq:qstTBj}
\begin{split}
&\tT'_{\bs_i,-1}(F_j)=y_{i,\tau i,j;-c_{ij},-c_{\tau i,j}},\\
&\tT'_{\bs_i,-1}( \tT_{\bw}(E_{\tau j}) K_j')=(-q_i)^{ -c_{ij}-c_{\tau i,j}} \tT_{\bw}(x_{\tau i, i,\tau j;-c_{ij},-c_{\tau i,j}}) K_j' (K_i')^{-c_{ij}}(K_{\tau i}')^{-c_{\tau i,j}},
 \end{split}
\end{align}
 where the second formula follows from $ \tT'_{\bs_i,-1}\tT_{\bw}=\tT_{\bw}\tT'_{\bs_i,-1}$.

By \eqref{eq:qstTF'}, we have analogous formulas for the symmetry $\tT''_{\bs_i,+1}$
\begin{align}
\label{eq:qstTBj'}
\begin{split}
&\tT''_{\bs_i,+1}(F_j)=y'_{i,\tau i,j;-c_{ij},-c_{\tau i,j}},\\
&\tT''_{\bs_i,+1}( \tT_{\bw}(E_{\tau j}) K_j')=(-q_i)^{ -c_{ij}-c_{\tau i,j}} \tT_{\bw}(x'_{\tau i, i,\tau j;-c_{ij},-c_{\tau i,j}}) K_j' (K_i')^{-c_{ij}}(K_{\tau i}')^{-c_{\tau i,j}}.
 \end{split}
\end{align}

Recall elements $ \B_{i,\tau i, j;m_1,m_2}, \b_{i,\tau i, j;m_1,m_2}$ defined in Definitions~\ref{def:qsBij}-\ref{def:qsBij'}.

\begin{theorem}
\label{thm:qsbraid}
Let $ j\in \wI, j\neq i,\tau i$.
\begin{itemize}
\item[(1)] The element $\B_{i,\tau i, j;-c_{ij},-c_{\tau i,j}}\in \tUi$ satisfies
\begin{align}
 \B_{i,\tau i, j;-c_{ij},-c_{\tau i,j}}\tfX_i=\tfX_i \tT'_{\bs_i,-1}(B_j).
\end{align}

\item[(2)] The element $\b_{i,\tau i, j;-c_{ij},-c_{\tau i,j}}\in \tUi$ satisfies
\begin{align}
 \b_{i,\tau i, j;-c_{ij},-c_{\tau i,j}}\tT_{\bs_i}(\tfX_i)^{-1}=\tT_{\bs_i}(\tfX_i)^{-1} \tT''_{\bs_i,+1}(B_j).
\end{align}

\item[(3)] $ \b_{i,\tau i, j;-c_{ij},-c_{\tau i,j}}=\sigma^\imath(\B_{i,\tau i, j;-c_{ij},-c_{\tau i,j}})$.
\end{itemize}
\end{theorem}

In other word, the element $\tTa{i}(B_j):=\B_{i,\tau i, j;-c_{ij},-c_{\tau i,j}}$ satisfies \eqref{eq:inter2} and the element $\tTb{i}(B_j):=\b_{i,\tau i, j;-c_{ij},-c_{\tau i,j}}$ satisfies \eqref{eq:inter3}. Hence, we have proved the first statement in Theorem~\ref{thm:rktwo1}(ii).

\begin{proof}
We prove (1). By Lemma~\ref{lem:rktwo1} and \eqref{eq:qstTBj}, we have
\begin{align*}
\tfX_i \tT'_{\bs_i,-1}(B_j)=&  \tfX_i \tT'_{\bs_i,-1}(F_j)+\tfX_i \tT'_{\bs_i,-1}\big(\tT_{\bw}(E_j) K_j'\big)\\
=& \tfX_i \tT'_{\bs_i,-1}(F_j)+\tT'_{\bs_i,-1}\big(\tT_{\bw}(E_j) K_j'\big) \tfX_i\\
=&\tfX_i y_{i,\tau i,j;-c_{ij},-c_{\tau i,j}}\\
&+(-q_i)^{ -c_{ij}-c_{\tau i,j}} \tT_{\bw}(x_{\tau i, i,\tau j;-c_{ij},-c_{\tau i,j}}) K_j' (K_i')^{-c_{ij}}(K_{\tau i}')^{-c_{\tau i,j}}\tfX_i.
\end{align*}
On the other hand, setting $m_1=-c_{ij},m_2=-c_{\tau i,j}$ in Proposition~\ref{prop:qsBy}-\ref{prop:qsBx}, we have
\begin{align*}
& \B_{i,\tau i,j;-c_{ij},-c_{\tau i,j}}^- \tfX_i=\tfX_i y_{i,\tau i,j;-c_{ij},-c_{\tau i,j}},\\
& \B_{i,\tau i,j;-c_{ij},-c_{\tau i,j}}^+ = (-q_i)^{ -c_{ij}-c_{\tau i,j}} \tT_{\bw}(x_{\tau i, i,\tau j;-c_{ij},-c_{\tau i,j}}) K_j' (K_i')^{-c_{ij}}(K_{\tau i}')^{-c_{\tau i,j}}.
\end{align*}
Therefore, by the above formulas, we obtain the desired identity
\begin{align*}
\tfX_i \tT'_{\bs_i,-1}(B_j)=&  \B_{i,\tau i,j;-c_{ij},-c_{\tau i,j}}^-  \tfX_i+  \B_{i,\tau i,j;-c_{ij},-c_{\tau i,j}}^+  \tfX_i =  \B_{i,\tau i,j;-c_{ij},-c_{\tau i,j}} \tfX_i.
\end{align*}

We prove (2). By Lemma~\ref{lem:rktwo1} and \eqref{eq:qstTBj'}, we have
\begin{align*}
&\tT_{\bs_i}(\tfX_i)^{-1} \tT''_{\bs_i,+1}(B_j)\tT_{\bs_i}(\tfX_i)
\\
=&\tT_{\bs_i}(\tfX_i)^{-1} \tT_{\bs_i}(F_j)\tT_{\bs_i}(\tfX_i)+ \tT_{\bs_i}(\tfX_i)^{-1} \tT_{\bs_i}\big(\tT_{\bw}(E_j) K_j'\big)\tT_{\bs_i}(\tfX_i)
\\
=&\tT_{\bs_i}(F_j) + \tT_{\bs_i}(\tfX_i)^{-1} \tT_{\bs_i}\big(\tT_{\bw}(E_j) K_j'\big)\tT_{\bs_i}(\tfX_i)
\\
=&y'_{i,\tau i,j;-c_{ij},-c_{\tau i,j}}
\\
&+(-q_i)^{ -c_{ij}-c_{\tau i,j}}\tT_{\bs_i}(\tfX_i)^{-1} \tT_{\bw}(x'_{\tau i, i,\tau j;-c_{ij},-c_{\tau i,j}}) K_j' (K_i')^{-c_{ij}}(K_{\tau i}')^{-c_{\tau i,j}}\tT_{\bs_i}(\tfX_i).
\end{align*}
On the other hand, setting $m_1=-c_{ij},m_2=-c_{\tau i,j}$ in Proposition~\ref{prop:qsBy'}-\ref{prop:qsBx'}, we have
\begin{align*}
&\b_{i,\tau i,j;-c_{ij},-c_{\tau i,j}}^-=y'_{i,\tau i,j;-c_{ij},-c_{\tau i,j}},
\\
& \b_{i,\tau i,j;-c_{ij},-c_{\tau i,j}}^+\\
=&(-q_i)^{ -c_{ij}-c_{\tau i,j}}\tT_{\bs_i}(\tfX_i)^{-1} \tT_{\bw}(x'_{\tau i, i,\tau j;-c_{ij},-c_{\tau i,j}}) K_j' (K_i')^{-c_{ij}}(K_{\tau i}')^{-c_{\tau i,j}}\tT_{\bs_i}(\tfX_i).
\end{align*}
Therefore, by above two formulas, we obtain the desired identity
\begin{align*}
&\tT_{\bs_i}(\tfX_i)^{-1} \tT''_{\bs_i,+1}(B_j)\tT_{\bs_i}(\tfX_i)=\b_{i,\tau i,j;-c_{ij},-c_{\tau i,j}}^- +\b_{i,\tau i,j;-c_{ij},-c_{\tau i,j}}^+=\b_{i,\tau i,j;-c_{ij},-c_{\tau i,j}}.
\end{align*}

The statement (3) is a consequence of Proposition~\ref{prop:qsbB}.
\end{proof}

\section{Higher rank formulas for $c_{i,\tau i}=-1,\bw i=i$}
\label{sec:qsqs}
In this section, we fix $i\in \fwItau$ such that $c_{i,\tau i}=-1,\bw i=i$. Since $\bw$ commutes with $\tau$, $\bw\tau i=\tau i$.
In this case, we have $B_i=F_i+E_{\tau i}K_i'$ and $\bs_i=s_i s_{\tau i} s_i = s_{\tau i} s_i s_{\tau i}$. %

We define higher rank root vectors $\B_{i,\tau i,j;a,b,c},\b_{i,\tau i,j;a,b,c}\in \tUi$ in Definitions~\ref{def:qsqsBij}-\ref{def:qsqsBij'} via recursive relations.
We show that the higher rank formulas $\tTa{i}(B_j),\tTb{i}(B_j)$ are given by these root vectors in Theorem~\ref{thm:qsqs} and complete the proof for Theorem~\ref{thm:rktwo1}(iii).
The divided power formulations for $\B_{i,\tau i,j;a,b,c},\b_{i,\tau i,j;a,b,c}$ are obtained in Theorem~\ref{thm:qsqsdv}.

\subsection{Definitions of root vectors}
Let $\ad$ be the adjoint action on $\tU$, explicitly given by
\begin{align*}
\ad(E_i) u & = E_i u - K_i u K_i^{-1}E_i,
\\
\ad(F_i) u & = (F_i u -u F_i) K_i'^{-1},
\\
\ad(K_i) u & = K_i u K_i^{-1}.
\end{align*}
Set $\wad:=\omega\psi \circ \ad \circ \omega\psi$ and $\sad:=\sigma \circ \ad \circ \sigma$, where $\omega,\psi,\sigma$ are the Chevalley involution, the bar involution, and the anti-involution on $\tU$.

\begin{definition}
\label{def:qsqsxy}
Define elements $y_{i,\tau i, j;a,b,c},x_{i,\tau i, j;a,b,c},y'_{i,\tau i, j;a,b,c},x'_{i,\tau i, j;a,b,c}$ for $a,b,c\geq 0,  j \neq i,\tau i$ as follows
\begin{align*}
y_{i,\tau i, j;a,b,c}&=\wad(E_{ i}^{(a)}E_{\tau i}^{(b)}E_i^{(c)})F_j, \qquad y'_{i,\tau i, j;a,b,c}=\wsad(E_{ i}^{(a)}E_{\tau i}^{(b)}E_i^{(c)})F_j,
\\
x_{i,\tau i, j;a,b,c}&=\sad (E_{ i}^{(a)}E_{\tau i}^{(b)}E_i^{(c)})E_j, \qquad x'_{i,\tau i, j;a,b,c}=\ad (E_{ i}^{(a)}E_{\tau i}^{(b)}E_i^{(c)})E_j.
\end{align*}
\end{definition}

Denote $E_{\tau i}^{(a,b,c)}:=E_{ i}^{(a)}E_{\tau_i}^{(b)}E_i^{(c)}$ and $[K_i;x]:=\frac{K_i q_i^x-K_i^{-1} q_i^{-x}}{q_i-q_i^{-1}}$.
\begin{lemma}
\label{lem:EE}
We have, for $a,b,c\geq 0$,
\begin{align}
\begin{split}
    E_{ i} E_{\tau i}^{(a,b,c)}
    &= [a+1]E_{\tau i}^{(a+1,b,c)},
    \\
    E_{\tau i} E_{\tau i}^{(a,b,c)}
    &= [b-a+1 ] E_{\tau i}^{(a,b+1,c )} + [c+1] E_{\tau i}^{(a-1,b+1,c+1)}.
\end{split}
    \label{eq:EE}
\end{align}
\end{lemma}

\begin{proof}
The first identity is obvious. The second identity follows from a standard though lengthy induction.
\end{proof}

\begin{lemma}\label{lem:FE}
We have, for $a,b,c\geq 0$,
\begin{align}
\begin{split}
[F_i,E_{\tau i}^{(a,b,c)}]&=-E_{\tau i}^{(a-1,b,c)} [K_i;a-b+2c-1] - E_{\tau i}^{(a,b,c-1)} [K_i;c-1],
\\
[F_{\tau i},E_{\tau i}^{(a,b,c)}]&=-E_{\tau i}^{(a,b-1,c)}[K_{\tau i};b-c-1].
\end{split}
\end{align}
\end{lemma}

\begin{proof}
The first identity follows from the following relations
\begin{align*}
[F_i,E_{\tau i}]=0,\qquad [F_i, E_i^{(m)}]=-[m] E_i^{(m-1)}[K_i;m-1],\qquad [K_i;m]E_i=E_i[K_i;m+2],
\end{align*}
cf. \cite[1.3,1.6]{Ja95}. One can prove the second identity via similar relations.
\end{proof}

We write $[A,B]_x:=AB-xBA$ for scalars $x$.
\begin{lemma}
\label{lem:qsy}
We have, for $a,b,c\geq 0$,
\begin{itemize}
\item[(1)] $\big[F_i, y_{i,\tau i, j;a,b,c}\big]_{q_i^{b-2a-2c-c_{ij}}} =[a+1]_i y_{i,\tau i, j;a+1,b,c}$.
\item[(2)] $\big[F_{\tau i} , y_{i,\tau i, j;a,b,c}\big]_{q_i^{-2b+a+c-c_{\tau i,j}}} = [b-a+1 ]_i y_{i,\tau i, j;a,b+1,c} + [c+1]_i y_{i,\tau i, j;a-1,b+1,c+1}$.
\item[(3)] $[E_{i}, y_{i,\tau i, j;a,b,c}] = [-c_{ij}-a+b-2c+1]_i y_{i,\tau i, j;a-1,b,c}K_i' + [ -c_{ij}-c+1]_i y_{i,\tau i, j;a,b,c-1}K_i' $.
\item[(4)] $[E_{\tau i}, y_{i,\tau i, j;a,b,c}]= [-c_{\tau i,j}-b+c+1]_i y_{i,\tau i, j;a,b-1,c} K_{\tau i}'$.
\end{itemize}
\end{lemma}

\begin{proof}

We give a detailed proof for (2).
On one hand, for $u\in \tU$, we have $\wad(E_{\tau i})u=F_{\tau i} u - K_{\tau i} u K_{\tau i}^{-1} F_{\tau i}$, which implies that
\begin{align*}
\wad(E_{\tau i})y_{i,\tau i, j;a,b,c}=\big[F_{\tau i} , y_{i,\tau i, j;a,b,c}\big]_{q_i^{-2b+a+c-c_{\tau i,j}}}.
\end{align*}
On the other hand, by definition of $y_{i,\tau i, j;a,b,c}$ and Lemma~\ref{lem:EE}, we have
\begin{align*}
&\wad(E_{\tau i})y_{i,\tau i, j;a,b,c}
\\
=&\wad(E_{\tau i})\wad(E_{ i}^{(a)}E_{\tau i}^{(b)}E_i^{(c)})F_j
\\
=&[b-a+1 ]_i \wad(E_{ i}^{(a)}E_{\tau i}^{(b+1)}E_i^{(c)})F_j+ [c+1]_i \wad(E_{ i}^{(a-1)}E_{\tau i}^{(b+1)}E_i^{(c+1)})F_j
\\
=&[b-a+1 ]_i y_{i,\tau i, j;a,b+1,c} + [c+1]_i y_{i,\tau i, j;a-1,b+1,c+1}.
\end{align*}
The identity (2) follows by above two formulas.

The identity (1) is obtained by considering the action of $\wad(E_i)$ on $y_{i,\tau i, j;a,b,c}$ via similar arguments. Identities (3)-(4) are obtained by respectively considering the action of $\wad(F_i),\wad(F_{\tau i})$ on $y_{i,\tau i,j;a,b,c}$ and using Lemma~\ref{lem:FE}. We omit details for them.
\end{proof}

\begin{lemma}
\label{lem:qsx}
We have, for $a,b,c\geq 0$,
\begin{itemize}
\item[(1)] $ x_{i,\tau i, j;a,b,c} E_i-q_i^{-b+2a+2c+c_{ij}}E_i x_{i,\tau i, j;a,b,c} =[a+1]_i x_{i,\tau i, j;a+1,b,c}$.
\item[(2)] $ \big[x_{i,\tau i, j;a,b,c}, E_{\tau i}\big]_{q_i^{ 2b-a-c+c_{\tau i,j}}} = [b-a+1 ]_i x_{i,\tau i, j;a,b+1,c} + [c+1]_i x_{i,\tau i, j;a-1,b+1,c+1}$.
\item[(3)] $[x_{i,\tau i, j;a,b,c}, F_{i}] = [-c_{ij}-a+b-2c+1]_i K_i x_{i,\tau i, j;a-1,b,c}  + [ -c_{ij}-c+1]_i K_i  x_{i,\tau i, j;a,b,c-1}$.
\item[(4)] $[x_{i,\tau i, j;a,b,c}, F_{\tau i}]= [-c_{\tau i,j}-b+c+1]_i K_{\tau i} x_{i,\tau i, j;a,b-1,c}$.
\end{itemize}
\end{lemma}

\begin{proof}
By definition, we have
\begin{align*}
x_{i,\tau i, j;a,b,c}=\sigma \omega \psi (y_{i,\tau i, j;a,b,c}).
\end{align*}
Then these four identities are obtained by applying $\sigma\omega \psi$ to those four identities in Lemma~\ref{lem:qsy}.
\end{proof}

\begin{definition}
\label{def:qsqsBij}
Let $j\in \wI$ such that $j\neq i,\tau i$. Define $\B^\pm_{i,\tau i, j;a,b,c}$ to be the elements in $\tU$ determined by the following recursive relations
\begin{align}\notag
& B_i  \B^\pm_{i,\tau i, j;a,b,c} -q_i^{b-2a-2c-c_{ij}} \B^\pm_{i,\tau i, j;a,b,c}B_i
\\\label{def:qsqsBij1}
=&[a+1]_i\B^\pm_{i,\tau i, j;a+1,b,c}+q_i^{b-2a-2c-c_{ij}-1}[-c_{\tau i,j}-b+c+1]_i \B^\pm_{i,\tau i, j;a,b-1,c}\tk_i,
\end{align}
and
\begin{align}\notag
&\quad B_{\tau i}  \B^\pm_{i,\tau i, j;a,b,c}-q_i^{-2b+a+c-c_{\tau i,j}} \B^\pm_{i,\tau i, j;a,b,c}B_{\tau i}
\\\label{def:qsqsBij2}
&=[b-a+1 ]_i \B^\pm_{i,\tau i, j;a,b+1,c} + [c+1]_i \B^\pm_{i,\tau i, j;a-1,b+1,c+1}
\\\notag
&+q_i^{a-2b+c-c_{\tau i,j}-1} \big([-c_{ij}-a+b-2c+1]_i \B^\pm_{i,\tau i, j;a-1,b,c}  +[-c_{ij}-c+1]_i \B^\pm_{i,\tau i, j;a,b,c-1}\big)\tk_{\tau i},
\end{align}
where we set $\B^\pm_{i,\tau i, j;a,b,c}=0$ if either one of $a,b,c$ is negative, and set
\begin{align}
\B^-_{i,\tau i, j;0,0,0}=F_j,\qquad \B^+_{i,\tau i, j;0,0,0}=\tT_{\bw}(E_{\tau j}) K_j'.
\end{align}
\end{definition}

\begin{definition}
\label{def:qsqsBij'}
Let $ j\in \wI$ such that $j\neq i,\tau i$. Define $\b^\pm_{i,\tau i, j;a,b,c}$ to be the elements in $\tU$ determined by the following recursive relations
\begin{align}\notag
& \b^\pm_{i,\tau i, j;a,b,c} B_i  -q_i^{b-2a-2c-c_{ij}} B_i \b^\pm_{i,\tau i, j;a,b,c}
\\\label{def:qsqsBij1'}
=&[a+1]_i\b^\pm_{i,\tau i, j;a+1,b,c}+q_i^{b-2a-2c-c_{ij}-1}[-c_{\tau i,j}-b+c+1]_i\tk_{\tau i} \b^\pm_{i,\tau i, j;a,b-1,c},
\end{align}
and
\begin{align}\notag
& \qquad\b^\pm_{i,\tau i, j;a,b,c}B_{\tau i} -q_i^{-2b+a+c-c_{\tau i,j}} B_{\tau i}\b^\pm_{i,\tau i, j;a,b,c}
\\\label{def:qsqsBij2'}
&=[b-a+1 ]_i \b^\pm_{i,\tau i, j;a,b+1,c} + [c+1]_i  \b^\pm_{i,\tau i, j;a-1,b+1,c+1}
\\\notag
&+q_i^{a-2b+c-c_{\tau i,j}-1} \tk_{i}\big([-c_{ij}-a+b-2c+1]_i  \b^\pm_{i,\tau i, j;a-1,b,c} + [-c_{ij}-c+1]_i  \b^\pm_{i,\tau i, j;a,b,c-1}\big),
\end{align}
where we set $\b^\pm_{i,\tau i, j;a,b,c}=0$ if either one of $a,b,c$ is negative, and set
\begin{align}
\b^-_{i,\tau i, j;0,0,0}=F_j,\qquad \b^+_{i,\tau i, j;0,0,0}=\tT_{\bw}(E_{\tau j}) K_j'.
\end{align}
\end{definition}

Define $\B_{i,\tau i,j;a,b,c}:=\B^-_{i,\tau i,j;a,b,c}+\B^+_{i,\tau i,j;a,b,c}$. Similarly, define $\b_{i,\tau i,j;a,b,c}$.
Since
\begin{align*}
\B_{i,\tau i,j;0,0,0}=\b_{i,\tau i,j;0,0,0}=B_j\in \tUi,
\end{align*}
it follows from the above recursive definitions that $\B_{i,\tau i,j;a,b,c},\b_{i,\tau i,j;a,b,c}\in \tUi$ for any $a,b,c$.

Recall the anti-involution $\sigma^\imath$ on $\tUi$ from Proposition~\ref{prop:sigma}.
\begin{proposition}
\label{prop:qsqsbB}
Let $j\in \wI,j\neq i,\tau i$. Then $\b_{i,j;a,b,c}=\sigma^\imath ( \B_{i,j;a,b,c})$ for $a,b,c\geq 0$.
\end{proposition}

\begin{proof}
By Definition~\ref{def:qsqsBij}-\ref{def:qsqsBij'}, $\sigma^\imath ( \B_{i,j;a,b,c})$ satisfies the same recursive relations as $\b_{i,j;a,b,c}$. Since $\b_{i,j;0,0,0}=B_j=\sigma^\imath ( \B_{i,j;0,0,0})$, this proposition follows.
\end{proof}

\begin{lemma}
\label{lem:qsqsBB}
We have, for $a,b,c\geq 0,j\neq i,\tau i,j\in \wI$,
\begin{align*}
\b^-_{i,\tau i, j;a,b,c}&=\sigma \big( \tfX_i^{-1} \B^-_{i,\tau i, j;a,b,c} \tfX_i \big).
\end{align*}
\end{lemma}

\begin{proof}
Consider the subalgebra $\tU_{[i;j]}^-$ of $\tU$ generated by $B_i,B_{\tau i},\tk_i,\tk_{\tau i},F_j$. It is clear from the above definitions that $\B^-_{i,\tau i, j;a,b,c},\b^-_{i,\tau i, j;a,b,c}\in \tU_{[i;j]}^-$. By Proposition~\ref{prop:fX1} and Lemma~\ref{lem:rktwo1}, there is a well-defined anti-automorphism $\sigma_{ij}$ on $\tU_{[i;j]}^-$, which is given by
\begin{align*}
\sigma_{ij}: x \mapsto \sigma \big( \tfX_i^{-1} x \tfX_i \big).
\end{align*}
Moreover, $\sigma_{ij}$ fixes $B_i,B_{\tau i},F_j$ and sends $\tk_i\leftrightarrow \tk_{\tau i}$. Applying $\sigma_{ij}$ to \eqref{def:qsqsBij1}-\eqref{def:qsqsBij2}, it is clear that $ \sigma_{ij} (\B^-_{i,\tau i, j;a,b,c})$ satisfies the same recursive relations as $\b^-_{i,\tau i, j;a,b,c}$. Then the desired identity follows by induction.
\end{proof}

\subsection{Intertwining properties}
\label{sec:qsqsinter}

We formulate the intertwining relations between elements $\B^\pm_{i,\tau i,j;a,b,c}$ and $y_{i,\tau i,j;a,b,c},x_{\tau i, i,\tau j;a,b,c}$.
\begin{proposition}
\label{prop:qsqsBy}
We have,  for $a,b,c\geq 0,j\neq i,\tau i,j\in \wI$,
\begin{align}
\label{eq:qsqsBy1}
\B_{i,\tau i, j;a,b,c}^- \tfX_i = \tfX_i y_{i,\tau i, j;a,b,c}.
\end{align}
\end{proposition}

\begin{proof}
Let $R_{a,b,c} $ denote $\tfX_i y_{i,\tau i, j;a,b,c}\tfX_i^{-1}$. By Lemma~\ref{lem:rktwo1}, $R_{0,0,0}=F_j=\B_{i,\tau i, j;0,0,0}^-$. Moreover, by definition, if either one of $a,b,c$ is negative, then $y_{i,\tau i, j; a,b,c}=y_{i,\tau i, j;a,b,c}=0$. Hence, it suffices to prove that $R_{a,b,c}$ satisfies the same recursive relations as $\B_{i,\tau i, j;a,b,c}^-$.

Recall that $B_i^\sigma=F_i+K_i E_{\tau i}$. We have, by Proposition~\ref{prop:fX1},
\begin{align*}
&\tfX_i^{-1}\big(B_i R_{a,b,c} -q_i^{b-2a-2c-c_{ij}} R_{a,b,c} B_i \big) \tfX_i
\\
=&B_i^\sigma y_{i,\tau i, j;a,b,c}-q_i^{b-2a-2c-c_{ij}} y_{i,\tau i, j;a,b,c}B_i^\sigma
\\
=&F_i  y_{i,\tau i, j;a,b,c}-q_i^{b-2a-2c-c_{ij}} y_{i,\tau i, j;a,b,c}F_i
\\
&+q_i^{b-2a-2c-c_{ij}-1} (E_{\tau i} y_{i,\tau i, j;a,b,c}- y_{i,\tau i, j;a,b,c}E_{\tau i}) K_i
\\
=&[a+1]_i y_{i,\tau i, j;a+1,b,c}+q_i^{b-2a-2c-c_{ij}-1}[-c_{\tau i,j}-b+c+1]_i y_{i,\tau i, j;a,b-1,c}\tk_i,
\end{align*}
where the last step follow from Lemma~\ref{lem:qsy}(1)(4).
This computation shows that the element $R_{a,b,c} $ satisfies \eqref{def:qsqsBij1}.

For $B_{\tau i}^\sigma=F_{\tau i}+K_{\tau i} E_i$, by Proposition~\ref{prop:fX1}, we similarly have
\begin{align*}
&\quad\tfX_i^{-1}\big(B_{\tau i}  R_{a,b,c} -q_i^{-2b+c+a-c_{\tau i,j}} R_{a,b,c} B_{\tau i}\big) \tfX_i
\\
& =B_{\tau i}^\sigma y_{i,\tau i, j;a,b,c}-q_i^{-2b+c+a-c_{\tau i,j}} y_{i,\tau i, j;a,b,c} B_{\tau i}^\sigma
\\
&=F_{\tau i}  y_{i,\tau i, j;a,b,c}-q_i^{-2b+c+a-c_{\tau i,j}} y_{i,\tau i, j;a,b,c} F_{\tau i}
\\
&\quad+q_i^{c+a-2b-c_{\tau i,j}-1}(E_{ i}  y_{i,\tau i, j;a,b,c}- y_{i,\tau i, j;a,b,c} E_{ i}) K_{\tau i}
\\
&=[b-a+1 ]_i y_{i,\tau i, j;a,b+1,c} + [c+1]_i y_{i,\tau i, j;a-1,b+1,c+1}
\\
&+q_i^{c+a-2b-c_{\tau i,j}-1} \big([-c_{ij}-a+b-2c+1]_i y_{i,\tau i, j;a-1,b,c}  +[-c_{ij}-c+1]_i y_{i,\tau i, j;a,b,c-1}\big)\tk_{\tau i}.
\end{align*}
This computation shows that the element $R_{a,b,c} $ satisfies \eqref{def:qsqsBij2}. Therefore, we have proved \eqref{eq:qsqsBy1} for any $a,b,c\geq 0$.
\end{proof}

\begin{proposition}
\label{prop:qsqsBx}
We have,  for $a,b,c\geq 0,j\neq i,\tau i,j\in \wI$,
\begin{align}
\label{eq:qsqsBx1}
\B^+_{i,\tau i, j;a,b,c} =(-1)^{a+b+c} q_i^{-\frac{1}{2}(a+b+c)(a+b+c-1+2c_{ij}+2c_{\tau i,j})} \tT_{\bw}(x_{\tau i,i,\tau j; a,b,c}) K_j' (K'_i)^{a+c} (K_{\tau i}')^b.
\end{align}
\end{proposition}

\begin{proof}
Let $P_{a,b,c}$ denote RHS \eqref{eq:qsqsBx1} and $x_{\tau i,i,\bw\tau j; a,b,c}$ denote $\tT_{\bw}(x_{\tau i,i,\tau j; a,b,c})$. It is clear that $P_{0,0,0}=\tT_{\bw}(E_{\tau j}) K_j'=\B^+_{i,\tau i, j;0,0,0}$. Moreover, if either one of $a,b,c$ is negative, then $P_{a,b,c}=0=\B^+_{i,\tau i, j;a,b,c}$. Hence, it suffices to show that $P_{a,b,c}$ also satisfies the defining recursive relations \eqref{def:qsqsBij1}-\eqref{def:qsqsBij2} for $B^+_{i,\tau i, j;a,b,c}$.

Let $Q_{a,b,c}$ denote $x_{\tau i,i,\bw\tau j; a,b,c} K_j' (K'_i)^{a+c} (K_{\tau i}')^b$. Applying $\tau$ to Lemma~\ref{lem:qsx}(1)(4), we obtain two recursions for $x_{\tau i,i,\tau j; a,b,c}$. Since $F_i,E_{\tau i}$ are fixed by $\tT_{\bw}$, $x_{\tau i,i,\bw\tau j; a,b,c}$ also satisfies the same recursions. To this end, we have
\begin{align}
\label{eq:qsqsBx4}
\begin{split}
&x_{\tau i, i, \bw\tau j;a,b,c} E_{\tau i}-q_i^{-b+2a+2c+c_{ij}}E_{\tau i} x_{\tau i, i, \bw\tau j;a,b,c} =[a+1]_i x_{i,\tau i, \bw\tau j;a+1,b,c},
\\
&[x_{\tau i,i,\bw\tau j;a,b,c}, F_{i}]= [-c_{\tau i,j}-b+c+1]_i K_{  i} x_{\tau i,i,\bw\tau j;a,b-1,c}.
\end{split}
\end{align}

We formulate the recursive relation for $Q_{a,b,c}$ as follows
\begin{align*}
&\quad B_i Q_{a,b,c} -q_i^{b-2a-2c-c_{ij}} Q_{a,b,c} B_i
\\
&=( F_i x_{\tau i,i,\bw\tau j; a,b,c} -x_{\tau i,i,\bw\tau j; a,b,c} F_i )K_j' (K'_i)^{a+c} (K_{\tau i}')^b
\\
&+q_i^{-a-c-b-c_{\tau i,j}-c_{ij}}(q_i^{2a+2c-b+c_{ij}}E_{\tau i} x_{\tau i,i,\bw\tau j; a,b,c} -x_{\tau i,i,\bw\tau j; a,b,c} E_{\tau i})K_j' (K'_i)^{a+c+1} (K_{\tau i}')^b
\\
&= -[-c_{\tau i,j}-b+c+1]_i q_i^{-a-c+2b-2+c_{\tau i,j}} x_{\tau i,i,\bw\tau j;a,b-1,c} K_j' (K'_i)^{a+c} (K_{\tau i}')^{b-1} \tk_i
\\
&\quad-q_i^{-a-c-b-c_{\tau i,j}-c_{ij}}[a+1]_i x_{i,\bw\tau i, j;a+1,b,c} E_{\tau i})K_j' (K'_i)^{a+c+1} (K_{\tau i}')^b,
\end{align*}
where we used \eqref{eq:qsqsBx4} in the last step.

The above computation implies that $Q_{a,b,c}$ satisfies the following recursive relation
\begin{align}
&B_i Q_{a,b,c} -q_i^{b-2a-2c-c_{ij}} Q_{a,b,c} B_i
\\\notag
=&-q_i^{-a-c-b-c_{\tau i,j}-c_{ij}}[a+1]_i Q_{a+1,b,c} -q_i^{-a-c+2b-2+c_{\tau i,j}}[-c_{\tau i,j}-b+c+1]_i Q_{a,b-1,c}\tk_i.
\end{align}
Similarly, one can show that $Q_{a,b,c}$ also satisfies the following recursive relation
\begin{align}\notag
&B_{\tau i}  Q_{ a,b,c} - q_i^{-2b+a+c-c_{\tau i,j}} Q_{ a,b,c} B_{\tau i}
\\
=&-q_i^{2a+2c-2-b+c_{ij}}\big([-c_{ij}-a+b-2c+1]_i  Q_{a-1,b,c}  + [-c_{ij}-c+1]_i   Q_{a,b,c-1} \big)  \tk_{\tau i}
\\\notag
&-q_i^{-b- a- c-c_{ij}-c_{\tau i,j}}\big([b-a+1 ]_i Q_{a,b+1,c} + [c+1]_i Q_{a-1,b+1,c+1}\big).
\end{align}

Since $P_{a,b,c}=(-1)^{a+b+c} q_i^{-\frac{1}{2}(a+b+c)(a+b+c-1+2c_{ij}+2c_{\tau i,j})} Q_{a,b,c} $, we obtain that $P_{a,b,c}$ satisfy the following two relations
\begin{align}\notag
& B_i P_{a,b,c} -q_i^{b-2a-2c-c_{ij}} P_{a,b,c}B_i
\\
\label{eq:qsqsBx2}
=&[a+1]_i P_{a+1,b,c}+q_i^{b-2a-2c-c_{ij}-1}[-c_{\tau i,j}-b+c+1]_i P_{a,b-1,c}\tk_i,
\end{align}
and
\begin{align}\notag
&B_{\tau i} P_{a,b,c}-q_i^{-2b+a+c-c_{\tau i,j}} P_{a,b,c}B_{\tau i}
\\
\label{eq:qsqsBx3}
=&[b-a+1 ]_i P_{a,b+1,c} + [c+1]_i  P_{a-1,b+1,c+1}
\\\notag
&+q_i^{-2b+c+a-c_{\tau i,j}-1} \big([-c_{ij}-a+b-2c+1]_i  P_{a-1,b,c}  +[-c_{ij}-c+1]_i P_{a,b,c-1}\big)\tk_{\tau i}.
\end{align}
These two relations tell that $P_{a,b,c}$ satisfy the recursive relations \eqref{def:qsqsBij1}-\eqref{def:qsqsBij2}. Therefore, $P_{a,b,c}=\B_{i,\tau i,\tau j;a,b,c}^+$ for any $a,b,c\geq 0$.
\end{proof}

We next formulate intertwining relations between elements $\b^\pm_{i,\tau i,j;a,b,c}$ and $y'_{i,\tau i,j;a,b,c}$, $x'_{\tau i, i,\tau j;a,b,c}$.
\begin{proposition}
\label{prop:qsqsBx'}
We have, for $a,b,c\geq 0, j\neq i,\tau i,j\in \wI$,
\begin{align}
\label{eq:qsqsBy1'}
\b^-_{i,\tau i,j;a,b,c}=y'_{i,\tau i,j;a,b,c},
\\\notag
\b^+_{i,\tau i, j;a,b,c} =(-1)^{a+b+c} &q_i^{-\frac{1}{2}(a+b+c)(a+b+c-1+2c_{ij}+2c_{\tau i,j})}\times
\\
&\times \tT_{\bs_i}(\tfX_i)^{-1} \tT_{\bw}(x_{\tau i,i,\tau j; a,b,c}) K_j' (K'_i)^{a+c} (K_{\tau i}')^b \tT_{\bs_i}(\tfX_i).
\label{eq:qsqsBx1'}
\end{align}
\end{proposition}

\begin{proof}
The firs identity is obtained by applying Lemma~\ref{lem:qsqsBB} to Proposition~\ref{prop:qsqsBy}. We prove the second identity.

Let $P_{a,b,c}$ denote the RHS \eqref{eq:qsqsBx1'} and $x'_{\tau i,i,\bw\tau j;a,b,c}$ denote $\tT_{\bw}(x_{\tau i,i,\tau j; a,b,c})$. By Lemma~\ref{lem:rktwo1}, $P_{0,0,0} =\tT_{\bw}(E_{\tau j}) K_j'=\b_{ i, \tau i, j;0,0,0}^+$. Moreover, $P_{a,b,c}=0=\b_{ i, \tau i, j;a,b,c}^+$ if either one of $a,b,c$ is negative. Hence, it suffices to show that $P_{a,b,c}$ satisfies the same recursive relations for $\b_{ i, \tau i, j;a,b,c}^+$.

Applying $\sigma$ to Lemma~\ref{lem:qsx}(1)(4) and then shifting the indices $i,j$ to $\tau i,\tau j$, we have two recursions for $x'_{\tau i, i,\tau j;a,b,c} $. Since $F_i,E_{\tau i}$ are fixed by $\tT_{\bw}$, $x'_{\tau i, i,\bw\tau j;a,b,c}$ satisfies the same recursions. To this end, we obtain
\begin{align}
\label{eq:qsx'}
\begin{split}
&- x'_{\tau i,i,\bw\tau j;a,b,c} F_{ i}+ F_{i} x'_{\tau i,i,\bw\tau j;a,b,c}= [-c_{\tau i,j}-b+c+1]_i \; x'_{\tau i,i,\bw\tau j;a,b-1,c} K'_{i},\\
&- q_i^{-b+2a+2c+c_{ij}} x'_{\tau i, i,\bw\tau j;a,b,c} E_{\tau i}+ E_{\tau i} x'_{\tau i, i,\bw\tau j;a,b,c}=[a+1]_i  x'_{\tau i,i,\bw\tau j;a+1,b,c} .
\end{split}
\end{align}

In the case $c_{i,\tau i}=-1$, set
\begin{align*}
\widehat{B}_i:= -q_i \tT_{\bs_i}(B_i \tk_i^{-1})= q_i F_i K_{\tau i} K_{\tau i}'^{-1}+ E_{\tau i} K'_i.
\end{align*}
It follows from \cite[\S 6.4]{WZ22} that $\widehat{B}_i \tT_{\bs_i}(\tfX_i)=\tT_{\bs_i}(\tfX_i) B_i$.

Let $Q_{a,b,c}$ denote $\tT_{\bs_i}(\tfX_i)^{-1} x'_{\tau i,i, \bw\tau j;a,b,c} K'_j (K_i')^{a+c} (K_{\tau i}')^{b}\tT_{\bs_i}(\tfX_i)$. We first formulate the recursive relations for $Q_{m_1,m_2}$ as follows
\begin{align*}
&\quad\tT_{\bs_i}(\tfX_i) \big( Q_{a,b,c} B_i  -q_i^{b-2a-2c-c_{ij}} B_i Q_{a,b,c}\big)\tT_{\bs_i}(\tfX_i)^{-1}
\\
&=x'_{\tau i,i, \bw\tau j;a,b,c}  K'_j (K_i')^{a+c} (K_{\tau i}')^{b} \widehat{B}_i
\\
&\quad -q_i^{b-2a-2c-c_{ij}}\widehat{B}_i  x'_{\tau i,i, \bw\tau j;a,b,c} K'_j (K_i')^{a+c} (K_{\tau i}')^{b}
\\
&=q_i^{2a+2c-b+c_{ij}+1}\big(x'_{\tau i,i, \bw\tau j;a,b,c}F_i-F_i x'_{\tau i,i,\bw \tau j;a,b,c}\big) K'_j (K_i')^{a+c-1} (K_{\tau i}')^{b-1}\tk_{\tau i}
\\
&+q_i^{-a-b-c-c_{ij}-c_{\tau i,j}}\big(q_i^{2a+2c-b+c_{ij}}x'_{\tau i, i,\bw\tau j;a,b,c} E_{\tau i}- E_{\tau i} x'_{\tau i, i,\bw\tau j;a,b,c}\big) K'_j (K_i')^{a+c+1} (K_{\tau i}')^{b}
\\
&= q_i^{2a+2c-b+c_{ij}+1}[-c_{\tau i,j}-b+c+1]_i \; x'_{\tau i,i,\bw\tau j;a,b-1,c} K'_j (K_i')^{a+c } (K_{\tau i}')^{b-1}\tk_{\tau i}
\\
&\quad-q_i^{-a-b-c-c_{ij}-c_{\tau i,j}}[a+1]_i \; x'_{\tau i,i,\bw\tau j;a+1,b,c}K'_j (K_i')^{a+c+1} (K_{\tau i}')^{b},
\end{align*}
where the last step follows by applying \eqref{eq:qsx'}.

The above computation shows that $Q_{a,b,c}$ satisfies the next recursive relation
\begin{align*}
 &Q_{a,b,c} B_i  -q_i^{b-2a-2c-c_{ij}} B_i Q_{a,b,c}
 \\
 =&q_i^{2a+2c-b+c_{ij}+1}[-c_{\tau i,j}-b+c+1]_i Q_{a,b-1,c}\tk_{\tau i} -q_i^{-a-b-c-c_{ij}-c_{\tau i,j}}[a+1]_i Q_{a+1,b,c}.
\end{align*}
Similarly, one can show that $Q_{a,b,c}$ also satisfies
\begin{align*}
&\quad Q_{ a,b,c}B_{\tau i} -q_i^{-2b+a+c-c_{\tau i,j}} B_{\tau i}Q_{ a,b,c}
\\
&=-q_i^{-a-b-c-c_{ij}-c_{\tau i,j}}\big([b-a+1 ]_i Q_{ a,b+1,c} + [c+1]_i  Q_{ a-1,b+1,c+1}\big)
\\
&-q_i^{2a-b+2c+c_{i,j}-2} \tk_{i}\big([-c_{ij}-a+b-2c+1]_i  Q_{i,\tau i, j;a-1,b,c}  +[-c_{ij}-c+1]_i  Q_{i,\tau i, j;a,b,c-1}\big).
\end{align*}

By definition, $P_{a,b,c}=(-1)^{a+b+c} q_i^{-\frac{1}{2}(a+b+c)(a+b+c-1+2c_{ij}+2c_{\tau i,j})} Q_{a,b,c}$. Then $P_{a,b,c}$ satisfies the following recursive relations
\begin{align}
\label{eq:qsqsBx2'}
 &P_{a,b,c} B_i  -q_i^{b-2a-2c-c_{ij}} B_i P_{a,b,c}
 \\\notag
 =&[a+1]_i P_{a+1,b,c}- q_i^{a+c-2b-c_{\tau i,j}+2}[-c_{\tau i,j}-b+c+1]_i P_{a,b-1,c}\tk_{\tau i}.
\end{align}
and
\begin{align}\notag
&P_{a,b,c}B_{\tau i} -q_i^{-2b+a+c-c_{\tau i,j}} B_{\tau i}P_{a,b,c}
\\\label{eq:qsqsBx3'}
=&[b-a+1 ]_i P_{i,\tau i, j;a,b+1,c} + [c+1]_i  P_{a-1,b+1,c+1}
\\\notag
&+q_i^{-2b+c+a-c_{\tau i,j}-1} \tk_{i}\big([-c_{ij}-a+b-2c+1]_i  P_{a-1,b,c}  +[-c_{ij}-c+1]_i P_{a,b,c-1}\big).
\end{align}
These two relations \eqref{eq:qsqsBx2'}-\eqref{eq:qsqsBx3'} indicate that $P_{a,b,c}$ also satisfies the defining recursive relations \eqref{def:qsqsBij1'}-\eqref{def:qsqsBij2'} for $\B^+_{i,\tau i,j;a,b,c}$. Therefore, we have proved $P_{a,b,c}=\B^+_{i,\tau i,j;a,b,c}$ for arbitrary $a,b,c\geq 0$.
\end{proof}


%
%
\subsection{Proof of Theorem~\ref{thm:rktwo1}(iii)}
Recall $\bs_i=s_i s_{\tau i} s_i=s_{\tau i} s_i s_{\tau i}$ when $c_{i,\tau i}=-1$.
We first formulate the relation between elements $y_{i,\tau i,j;a,b,c},x_{i,\tau i,j;a,b,c}$ in Definition~\ref{def:qsqsxy} and (unrescaled) Lusztig symmetries $\tTD'_{i,-1}$ on $\tU$.

For a subset $J\subset I$, denote by $\tU_J$ the subalgebra of $\tU$ generated by $E_j,F_j,K_j,K_j'$, $j\in J$.

\begin{lemma}
\label{lem:Tad}
Let $w\in W$ with a reduced expression $w=s_{i_1}\cdots s_{i_k}$. For any $j\not\in \{i_1\cdots i_k\}$, we have
\begin{align}
\tTD'_{w,-1}(E_j)=\sad(E_{i_1}^{(a_1)}\cdots E_{i_k}^{(a_k)})E_j
\end{align}
where $a_s=-\langle s_{i_k}s_{i_{k-1}}\cdots s_{i_{s+1}}(\alpha_{i_s}^\vee),\alpha_j\rangle$ for $1\leq s \leq k$.
\end{lemma}

\begin{proof}

We prove this lemma by induction on $k=l(w)$. For $k=1$, this result is well-known; see \cite[8.14(6)]{Ja95}.

Suppose that $k>1$. Set $w'=s_{i_2}\cdots s_{i_k}$ and $x:=\sad(E_{i_2}^{(a_2)}\cdots E_{i_k}^{(a_k)})E_j$. By induction hypothesis, we have $\tTD'_{w',-1}(E_j)= x$. It remains to show that
\begin{align}
\tTD'_{i_1,-1}(x)=\sad(E_{i_1}^{(a_1)})x.
\end{align}

We consider the $\tU_{i_1,\cdots,i_k}$-module $ \sad (\tU_{i_1,\cdots,i_k})E_j$ and denote its irreducible submodule containing $E_j$ by $M_j$. Since $j\not\in \{i_1\cdots i_k\}$, we have $\sad(F_{i_s})E_j=0$ and then $E_j$ is the lowest weight vector for the $M_j$. Note that both $x $ and $\sad(E_{i_1}^{(a_1 )})x $ are extremal weight vectors in $M_j$. Then we must have
\begin{align}
\label{eq:Tad}
\sad(F_{i_1}) x =0,\qquad \sad(E_{i_1}^{(a_1+1)}) x =0.
\end{align}

On the other hand, for any integrable $\tU$-module $V$, recall that
\begin{align*}
\tTD'_{i,-1}v=\sum_{a-b+c=m} (-1)^b q_i^{ac-b} F_{i_1}^{(a)}E_{i_1}^{(b)}F_{i_1}^{(c)}v,\qquad \forall v\in V_\lambda,
\end{align*}
where $m=\langle \alpha_i^\vee,\lambda \rangle$. Using similar arguments in \cite[8.9-8.10]{Ja95} and \eqref{eq:Tad}, we obtain
\begin{align*}
\tTD'_{i_1,-1}(xv)=\sad(E_{i_1}^{(a_1)})x\tTD'_{i_1,-1}(v)
\end{align*}
for any vector $v$ in any integrable $\tU$-module. Therefore, we have proved this lemma by induction.
\end{proof}

Write $\alpha,\beta$ for $-c_{ij},-c_{\tau i,j}$ respectively. By Lemma~\ref{lem:Tad}, we have
\begin{align}
\label{eq:qsqstT}
&\tTD'_{\bs_i,-1}(F_j)=y_{i,\tau i,j;\beta, \alpha+\beta,\alpha},\qquad
\tTD'_{\bs_i,-1}(E_{\tau j})=x_{\tau i,i,\tau j;\beta, \alpha+\beta,\alpha}.
\\
&\tTD''_{\bs_i,+1}(F_j)=y'_{i,\tau i,j;\beta, \alpha+\beta,\alpha},\qquad
\tTD''_{\bs_i,+1}(E_{\tau j})=x'_{\tau i,i,\tau j;\beta, \alpha+\beta,\alpha}.
\label{eq:qsqstT'}
\end{align}

Since $c_{i,\tau i}=-1$, $\vs_{i,\diamond}=-q_i^{-1/2}$. By \eqref{eq:qsqstT}, the action of (rescaled) symmetries $\tT'_{\bs_i,-1}$ is given by
\begin{align}
\label{eq:qsbraid1}
\begin{split}
\tT'_{\bs_i,-1}(F_j)&=y_{i,\tau i,j;\beta, \alpha+\beta,\alpha},
\\
\tT'_{\bs_i,-1} \big(\tT_{\bw}(E_{\tau j})K_j'\big)&=q_i^{\alpha+\beta}\tT_{\bw}( x_{\tau i,i,\tau j;\beta, \alpha+\beta,\alpha})K_j' (K_i' K_{\tau i}')^{\alpha+\beta},
\end{split}
\end{align}
where the second formula follows from $\tT'_{\bs_i,-1} \tT_{\bw}=\tT_{\bw}\tT'_{\bs_i,-1} $.

By \eqref{eq:qsqstT'}, we obtain analogous formulas for $\tT''_{\bs_i,+1}$ below
\begin{align}
\label{eq:qsbraid1'}
\begin{split}
\tT''_{\bs_i,+1}(F_j)&=y'_{i,\tau i,j;\beta, \alpha+\beta,\alpha},
\\
\tT''_{\bs_i,+1} \big(\tT_{\bw}(E_{\tau j})K_j'\big)&=q_i^{\alpha+\beta}\tT_{\bw}( x'_{\tau i,i,\tau j;\beta, \alpha+\beta,\alpha})K_j' (K_i' K_{\tau i}')^{\alpha+\beta},
\end{split}
\end{align}

Recall elements $\B_{i,\tau i,j; a,b,c},\b_{i,\tau i,j; a,b,c}$ defined in Definitions~\ref{def:qsqsBij}-\ref{def:qsqsBij'}.

\begin{theorem}
\label{thm:qsqs}
Let $j\in \wI, j\neq i,\tau i$.
\begin{itemize}
\item[(1)] The element $\B_{i,\tau i,j;\beta, \alpha+\beta,\alpha}$ satisfies
\begin{align}
\B_{i,\tau i,j;\beta, \alpha+\beta,\alpha} \tfX_i =\tfX_i \tT'_{\bs_i,-1}(B_j).
\end{align}
\item[(2)] The element $\b_{i,\tau i,j;\beta, \alpha+\beta,\alpha}$ satisfies
\begin{align}
\b_{i,\tau i,j;\beta, \alpha+\beta,\alpha} \tT_{\bs_i}(\tfX_i)^{-1}=\tT_{\bs_i}(\tfX_i)^{-1} \tT''_{\bs_i,+1}(B_j).
\end{align}
\item[(3)] $\b_{i,\tau i,j;\beta, \alpha\beta,\alpha}=\sigma^\imath(\B_{i,\tau i,j;\beta, \alpha\beta,\alpha})$.
\end{itemize}
\end{theorem}

In other word, the element $\tTa{i}(B_j):=\B_{i,\tau i,j;\beta, \alpha\beta,\alpha}$ satisfies \eqref{eq:inter2} and the element $\tTb{i}(B_j):=\b_{i,\tau i,j;\beta, \alpha\beta,\alpha}$ satisfies \eqref{eq:inter3}. Hence, we have proved Theorem~\ref{thm:rktwo1}(iii).

\begin{proof}
We prove (1). By Lemma~\ref{lem:rktwo1} and \eqref{eq:qsbraid1}, we have
\begin{align*}
\tfX_i \tT'_{\bs_i,-1}(B_j)=&  \tfX_i \tT'_{\bs_i,-1}(F_j)+\tfX_i \tT'_{\bs_i,-1}(E_j K_j')\\
=& \tfX_i \tT'_{\bs_i,-1}(F_j)+\tT'_{\bs_i,-1}(E_j K_j') \tfX_i\\
=&\tfX_i y_{i,\tau i,j;\beta,\beta+\alpha,\alpha}+q_i^{\alpha+\beta} \tT_{\bw}(x_{\tau i,i,j;\beta,\beta+\alpha,\alpha}) K_j' (K_i' K_{\tau i}')^{\alpha+\beta}\tfX_i
\end{align*}
On the other hand, setting $a=\beta,b=\beta+\alpha,c=\alpha$ in Proposition~\ref{prop:qsqsBy}-\ref{prop:qsqsBx}, we have
\begin{align*}
& \B_{i,\tau i,j;\beta,\beta+\alpha,\alpha}^- \tfX_i=\tfX_i y_{i,\tau i,j;\beta,\beta+\alpha,\alpha},\\
& \B_{i,\tau i,j;\beta,\beta+\alpha,\alpha}^+ =  q_i^{ \alpha+\beta} \tT_{\bw}(x_{\tau i,i,j;\beta,\beta+\alpha,\alpha} ) K_j' (K_i' K_{\tau i}')^{\alpha+\beta}.
\end{align*}
Therefore, by the above formulas, we have
\begin{align*}
\tfX_i \tT'_{\bs_i,-1}(B_j)=&  \B_{i,\tau i,j;\beta,\beta+\alpha,\alpha}^-  \tfX_i+  \B_{i,\tau i,j;\beta,\beta+\alpha,\alpha}^+  \tfX_i =  \B_{i,\tau i,j;\beta,\beta+\alpha,\alpha} \tfX_i.
\end{align*}

Similarly, one can obtain (2) using Proposition~\ref{prop:qsqsBx'} and \eqref{eq:qsbraid1'}.

The statement (3) is a consequence of Proposition~\ref{prop:qsqsbB}.
\end{proof}

\subsection{A divided power formulation} In this subsection, we derive divided power formulations for root vectors $\B_{i,\tau i,j;a,b,c},\b_{i,\tau i,j;a,b,c}$ from their recursive Definitions~\ref{def:qsqsBij}-\ref{def:qsqsBij'}.

\begin{lemma}
\label{lem:qsqsdv}
We have for any $a\geq 0$,
\begin{align*}
\B_{i,\tau i,j;a,0,0}=\B_{i,\tau i,j;0,0,a}.
\end{align*}
\end{lemma}

\begin{proof}
It suffices to show that $\B_{i,\tau i,j;a,0,0}^\pm=\B_{i,\tau i,j;0,0,a}^\pm$. By Definition~\ref{def:qsqsxy}, we have
\begin{align*}
y_{i,\tau i,j;a,0,0}=y_{i,\tau i,j;0,0,a},\qquad x_{i,\tau i,j;a,0,0}=x_{i,\tau i,j;0,0,a}.
\end{align*}
Then the desired identities follow from Proposition~\ref{prop:qsqsBy}-\ref{prop:qsqsBx}.
\end{proof}


\begin{theorem}
\label{thm:qsqsdv}
The elements $\B_{i,\tau i,j;a,b,c},\b_{i,\tau i,j;a,b,c} $ for $a,b,c\geq 0$ admit the following divided power formulations
\begin{align}\notag
&\B_{i,\tau i,j;a,b,c}=\sum_{u,v\geq 0} \sum_{t=0}^{a-v} \sum_{s=0}^{b-v-u} \sum_{r=0}^{c-u} (-1)^{t+v+r+s+u}\times
\\\notag
&\quad\times  q_i^{t(b-2c+\alpha-a-2v+1)+ v(b-2c+\alpha-a-\frac{v-1}{2})+ r(\alpha-c+u+1)+s(c+\beta-b+v-2u+1)+u(c+\beta-b+v-\frac{u-1}{2})}
\\\label{eq:qsqsdv1}
&\quad\times\qbinom{\beta-b+c+v}{v} \qbinom{\alpha-c+u}{u}B_i^{(a-v-t)}B_{\tau i}^{(b-v-u-s)}B_i^{(c-u-r)} B_j B_i^{(r)} B_{\tau i}^{(s)}\tk_{\tau i}^u B_i^{(t)} \tk_i^v,
\\\notag
&\b_{i,\tau i,j;a,b,c}=\sum_{u,v\geq 0} \sum_{t=0}^{a-v} \sum_{s=0}^{b-v-u} \sum_{r=0}^{c-u} (-1)^{t+v+r+s+u}\times
\\\notag
&\quad\times  q_i^{t(b-2c+\alpha-a-2v+1)+ v(b-2c+\alpha-a-\frac{v-1}{2})+ r(\alpha-c+u+1)+s(c+\beta-b+v-2u+1)+u(c+\beta-b+v-\frac{u-1}{2})}
\\\label{eq:qsqsdv2}
&\quad\times\qbinom{\beta-b+c+v}{v} \qbinom{\alpha-c+u}{u}\tk_{\tau i}^v B_i^{(t)} \tk_{i}^u B_{\tau i}^{(s)}B_i^{(r)}B_j B_i^{(c-u-r)} B_{\tau i}^{(b-v-u-s)}B_i^{(a-v-t)}.
\end{align}
\end{theorem}

\begin{proof}
Recall from Proposition~\ref{prop:qsqsbB} that $\b_{i,\tau i,j;a,b,c}=\sigma^\imath(\B_{i,\tau i,j;a,b,c})$. The second formula is obtained by applying $\sigma^\imath$ to the first one.

The first formula for $\B_{i,\tau i,j;a,b,c}$ is derived from recursive relations \eqref{def:qsqsBij1}-\eqref{def:qsqsBij2} in Definition~\ref{def:qsqsBij} via three steps.
\begin{enumerate}
\item Recall that $\B_{i,\tau i,j;0,0,0}=B_j$. Setting $b=c=0$ in \eqref{def:qsqsBij1}, we have a recursive relation for $\B_{i,\tau i,j;a,0,0}$. Using this relation and an induction on $a$, we obtain the formula of $\B_{i,\tau i,j;a,0,0}$ as follows
\begin{align*}
\B_{i,\tau i,j;a,0,0}=&\sum_{r=0}^a (-1)^r q_i^{r(\alpha-c+1)} B_i^{(a-r)} B_j B_i^{(r)}.
\end{align*}
By Lemma~\ref{lem:qsqsdv}, $\B_{i,\tau i,j;0,0,a}=\B_{i,\tau i,j;a,0,0}$ is given by the same formula.

\item Setting $a=0$ in \eqref{def:qsqsBij2}, we have a recursive relation for $\B_{i,\tau i,j;0,b,c}$. Using this relation and an induction on $b$, we can write $\B_{i,\tau i,j;0,b,c}$ in terms of $\B_{i,\tau i,j;0,0,c-u}$ for $0\leq u\leq \min(b,c)$ as follows
\begin{align*}
\B_{i,\tau i,j;0,b,c}=&\sum_{u=0}^{ \min(b,c)} \sum_{s=0}^{b-u}  (-1)^{s+u} q_i^{s(c+\beta-b-2u+1)+u(c+\beta-b-\frac{u-1}{2})}
\\
&\times \qbinom{\alpha-c+u}{u}B_{\tau i}^{(b-u-s)}\B_{i,\tau i,j;0,0,c-u} B_{\tau i}^{(s)}\tk_{\tau i}^u.
\end{align*}

\item Using \eqref{def:qsqsBij1} and an induction on $a$, we can write $\B_{i,\tau i,j;a,b,c}$ in terms of $\B_{i,\tau i,j;0,b-v,c}$ for $0\leq v\leq \min(a,b)$ as follows
\begin{align*}
\B_{i,\tau i,j;a,b,c}=&\sum_{v= 0}^{ \min(a,b)}\sum_{t=0}^{a-v} (-1)^{t+v} q_i^{t(b-2c+\alpha-a-2v+1)+ v(b-2c+\alpha-a-\frac{v-1}{2})}\times
\\
&\times \qbinom{\beta-b+c+v}{v} B_i^{(a-v-t)} \B_{i,\tau i,j;0,b-v,c} B_i^{(t)} \tk_i^v.
\end{align*}
\end{enumerate}
Now the desired formula \eqref{eq:qsqsdv1} is obtained by combining formulas in steps (1)-(3).
\end{proof}

\section{Symmetry $\tTT'_{i,-1}$ and root vectors}\label{sec:basic}

The braid group symmetries $\tTD'_{i,e},\tTD''_{i,e}$ on $\tU$ send root vectors to root vectors
\begin{align}
\label{eq:Troot}
\begin{split}
&\tTD'_{i,e}(x'_{i,j;m,e})=x_{i,j;-m-c_{ij},e},\qquad
\tTD''_{i,-e}(x_{i,j;m,e})=x'_{i,j;-m-c_{ij},e}.
\\
&\tTD'_{i,e}(y'_{i,j;m,e})=y_{i,j;-m-c_{ij},e},\qquad
\tTD''_{i,-e}(y_{i,j;m,e})=y'_{i,j;-m-c_{ij},e}.
\end{split}
\end{align}
cf. \cite[Proposition 37.2.5]{Lus94}.

We will show that our symmetry $\tTa{i}$ analogously sends root vectors in $\tUi$ to root vectors in $\tUi$. Precisely, when $ i=\tau i=\bw i$, the actions of $\tTa{i}$ on root vectors $ \b_{i,j;m}$ are formulated in Theorem~\ref{thm:basic}; when  $c_{i,\tau i}=0,i=\bw i$, the actions of  $\tTa{i}$ on root vectors $ \b_{i,\tau i,j;m_1,m_2}$ are formulated in Theorem~\ref{thm:qsbasic}.

\subsection{Case $ i=\tau i=\bw i$}
Note that $c_{ij}=c_{i,\tau j}$ when $i=\tau i$. Recall that $\tT'_{i,-1}=\tPsi_{\bvs_\dm }^{-1 } \tTD'_{i,-1}\tPsi_{\bvs_\dm }$ and $\vs_{i,\dm}=-q_i^{-2}$ in this case. Then we have
\begin{align}
\label{eq:BB'0}
\begin{split}
\tT'_{i,-1}(y'_{i,j;m})&=y_{i,j;-c_{ij}-m},\\
\tT'_{i,-1}\Big(\tT_{\bw}(x'_{i,\tau j;m}) K_j' (K_i')^m \Big)&=\vs_{i,\dm}^{c_{ij}+2m} \tT_{\bw}(x_{i,\tau j;-c_{ij}-m})K_j' (K_i')^{-c_{ij}-m},
\end{split}
\end{align}
where the second identity uses the commutativity of $\tT'_{i,-1},\tT_{\bw}$.

\begin{theorem}
\label{thm:basic}
We have, for $m\geq 0, j\neq i\in \wI$,
\begin{align}
\tTT'_{i,-1}(\b_{i,j;m})=\B_{i,j;-c_{ij}-m}.
\end{align}
\end{theorem}

\begin{proof}
By Proposition~\ref{prop:B'x} and \eqref{eq:BB'0}, we have
\begin{align}\notag
&\quad\tfX_i  \tT'_{i,-1}(\b_{i,j;m}^+)  \tfX_i ^{-1}\\\notag
& =(-1)^m  q_i^{ -2m(c_{ij}+m-1)} \tT'_{i,-1}\Big(\tT_{\bw}(x'_{i,\tau j;m}) K_j' (K_i')^m\Big)\\
&= (-1)^{c_{ij}+m} q_i^{-2 (c_{ij}+m)(m+1)} \tT_{\bw}(x_{i,\tau j;-c_{ij}-m})K_j' (K_i')^{-c_{ij}-m}.\label{eq:BB'1}
\end{align}
By the definition of $\tTT'_{i,-1}$, we have
\begin{align*}
\tTT'_{i,-1}(\b_{i,j;m})
&\overset{\quad}{=}\tfX_i \big(\tT'_{i,-1}(\b_{i,j;m}^+)+\tT'_{i,-1}(\b_{i,j;m}^-)\big) \tfX_i^{-1}\\
&\overset{(\star)}{=}\tfX_i \big(\tT'_{i,-1}(\b_{i,j;m}^+)+\tT'_{i,-1}(y'_{i,j;m} )\big) \tfX_i^{-1}\\
&\overset{(\dagger)}{=}(-1)^{c_{ij}+m} q_i^{-2 (c_{ij}+m)(m+1)} x_{i,\tau j;-c_{ij}-m}K_j' (K_i')^{-c_{ij}-m}\\
&\quad +\tfX_i y_{i,j;-c_{ij}-m}\tfX_i^{-1}\\
&\overset{(\ddagger)}{=}\B_{i,j;-c_{ij}-m}^+ + \B_{i,j;-c_{ij}-m}^-\\
&\overset{\quad}{=}\B_{i,j;-c_{ij}-m},
\end{align*}
where the equality ($\star$) follows from Proposition~\ref{prop:B'y}; the equality ($\dagger$) follows from \eqref{eq:BB'1} and \eqref{eq:BB'0}; the equality ($\ddagger$) follows form Propositions~\ref{prop:By}-\ref{prop:Bx}.
\end{proof}

\subsection{Case $c_{i,\tau i}=0,\bw i=i$}
Recall the root vectors $y'_{i,\tau i,j;m_1,m_2},x'_{\tau i, i,\tau j;m_1,m_2}$ in Definition~\ref{def:xy'}. We first show that the symmetry $\tT'_{\bs_i,-1}$ on $\tU$ sends these root vectors to root vectors, as a generalization of \cite[Proposition 37.2.5]{Lus94}. We then formulate the $\imath$analog of this property for the $\tTT'_{i,-1}$-action on $\tUi$ in Theorem~\ref{thm:qsbasic}.
\begin{proposition}
\label{prop:qsxy}
Let $i,j\in \I,j\neq i,\tau i$. We have
\begin{align}
\label{eq:qsyy}
&\tT'_{\bs_i,-1}(y'_{i,\tau i,j;m_1,m_2})=y_{i,\tau i, j; -c_{ij}-m_1,-c_{\tau i,j}-m_2},\\
&\tT'_{\bs_i,-1}(x'_{\tau i, i,\tau j;m_1,m_2} )=\vs_{i,\dm}^{(c_{ij}+c_{\tau i,j})/2+m_1+m_2} x_{\tau i, i, \tau j; -c_{ij}-m_1,-c_{\tau i,j}-m_2},\\
&\tT'_{\bs_i,-1}(K'_j (K_i')^{m_1} (K_{\tau i}')^{m_2})=\vs_{i,\dm}^{(c_{ij}+c_{\tau i,j})/2+m_1+m_2}K'_j (K_i')^{-c_{ij}-m_1} (K_{\tau i}')^{-c_{\tau i,j}-m_2}.
\end{align}
\end{proposition}

\begin{proof}
All three formulas are proved by similar computations; we only give the proof for the first formula here. Recall that, since $c_{i,\tau i}=0$, we have $\bs_i=s_i s_{\tau i}$ and $\vs_{i,\dm}=-q_i^{-1}$. The element $y'_{i,\tau i,j;m_1,m_2}$ admits a reformulation similar to Remark~\ref{rmk:qsxy}
\begin{align}
\label{eq:qsyy2}
y'_{i,\tau i,j;m_1,m_2}=\sum_{s=0}^{m_2} (-1)^{s}  q_i^{-s(m_2 +c_{\tau i,j}-1) }F_{\tau i}^{(s)}  y'_{  i,j;m_1,-1} F_{\tau i}^{(m_2 -s)}
\end{align}
By \cite[Proposition 37.2.5]{Lus94}, we have $\tT'_{i,-1}(y'_{  i,j;m_1,-1})=y_{  i,j;-c_{ij}-m_1,-1}$. By \eqref{eq:qsyy2}, we have
\begin{align*}
\tT'_{i,-1}(y'_{i,\tau i,j;m_1,m_2})&=\sum_{s=0}^{m_2} (-1)^{s}  q_i^{-s(m_2 +c_{\tau i,j}-1) }F_{\tau i}^{(s)} \tT'_{i,-1}(y'_{  i,j;m_1,-1}) F_{\tau i}^{(m_2 -s)}\\
&=\sum_{s=0}^{m_2} (-1)^{s}  q_i^{-s(m_2 +c_{\tau i,j}-1) }F_{\tau i}^{(s)} y_{  i,j;-c_{ij}-m_1,-1}  F_{\tau i}^{(m_2 -s)}\\
&=\sum_{r=0}^{-c_{ij}-m_1} (-1)^{r}  q_i^{r(m_1 +1) }F_i^{(-c_{ij}-m_1 -r)} y'_{\tau i,j;m_2,-1} F_i^{(r)}.
\end{align*}
Now applying $\tT'_{\tau i,-1}$ to the above formula and using \cite[Proposition 37.2.5]{Lus94} again, we obtain
\begin{align*}
\tT'_{\bs_i,-1}(y'_{i,\tau i,j;m_1,m_2})&=\sum_{r=0}^{-c_{ij}-m_1} (-1)^{r}  q_i^{r(m_1 +1) }F_i^{(-c_{ij}-m_1 -r)} \tT'_{\tau i,-1}(y'_{\tau i,j;m_2,-1}) F_i^{(r)}\\
&=\sum_{r=0}^{-c_{ij}-m_1} (-1)^{r}  q_i^{r(m_1 +1) }F_i^{(-c_{ij}-m_1 -r)}  y_{\tau i,j;-c_{\tau i,j}-m_2,-1}  F_i^{(r)}\\
&=y_{i,\tau i, j; -c_{ij}-m_1,-c_{\tau i,j}-m_2},
\end{align*}
where the last equality follows by Remark~\ref{rmk:qsxy}.
\end{proof}

Using  $\tT'_{\bs_i,-1} \tT_{\bw}=\tT_{\bw}\tT'_{\bs_i,-1}$ and Proposition~\ref{prop:qsxy}, we have the following formula
\begin{align}\notag
&\tT'_{\bs_i,-1}(\tT_{\bw}(x'_{\tau i, i,\tau j;m_1,m_2})K'_j (K_i')^{m_1} (K_{\tau i}')^{m_2})\\
=&\vs_{i,\dm}^{c_{ij}+c_{\tau i,j}+2m_1+2m_2} \tT_{\bw}(x_{\tau i, i, \tau j; -c_{ij}-m_1,-c_{\tau i,j}-m_2})K'_j (K_i')^{-c_{ij}-m_1} (K_{\tau i}')^{-c_{\tau i,j}-m_2}.
\label{eq:qsxx}
\end{align}

\begin{theorem}
\label{thm:qsbasic}
Let $i,j\in \I,j\neq i,\tau i$. We have
\begin{align}
\tTT'_{i,-1}(\b_{ i, \tau i, j;m_1,m_2})=\B_{i,\tau i,j;-c_{ij}-m_1, -c_{\tau i,j}-m_2}.
\end{align}
\end{theorem}

\begin{proof}
Recall that $\tTT'_{i,-1}(x)= \tfX_i \tT'_{\bs_i,-1}(x) \tfX_i^{-1}$ for any $x\in \tUi$. By Proposition~\ref{prop:qsBx'} and \eqref{eq:qsxx}, we have
\begin{align*}
&\tfX_i\tT'_{\bs_i,-1}(\b_{ i, \tau i, j;m_1,m_2}^+)\tfX_i^{-1}\\
= &(-1)^{m_1+m_2} q_i^{-(m_1+m_2)(c_{ij}+c_{\tau i,j}+m_1+m_2-1)}  \tT'_{\bs_i,-1}\big( \tT_{\bw}(x'_{\tau i,i, \tau j;m_1,m_2}) K'_j (K_i')^{m_1} (K_{\tau i}')^{m_2}\big)\\
= &(-1)^{m_1+m_2+c_{ij}+c_{\tau i,j}} q_i^{-(m_1+m_2+1)(c_{ij}+c_{\tau i,j}+m_1+m_2 )}\times \\
 &\times \tT_{\bw}(x_{\tau i, i, \tau j; -c_{ij}-m_1,-c_{\tau i,j}-m_2})K'_j (K_i')^{-c_{ij}-m_1} (K_{\tau i}')^{-c_{\tau i,j}-m_2}\\
=&\B_{i,\tau i,j; -c_{ij}-m_1,-c_{\tau i,j}-m_2}^+,
\end{align*}
where the last step follows from Proposition~\ref{prop:qsBx}.

On the other hand, by Proposition~\ref{prop:qsBy'}, \eqref{eq:qsyy} and Proposition~\ref{prop:qsBy}, we have
\begin{align*}
&\tfX_i \tT'_{\bs_i,-1}(\b_{ i, \tau i, j;m_1,m_2}^-)\tfX_i^{-1}\\
=& \tfX_i \tT'_{\bs_i,-1}(y'_{ i, \tau i, j;m_1,m_2})\tfX_i^{-1}\\
=& \tfX_i  y_{i,\tau i, j; -c_{ij}-m_1,-c_{\tau i,j}-m_2} \tfX_i^{-1}\\
=& \B_{i,\tau i,j; -c_{ij}-m_1,-c_{\tau i,j}-m_2}^-.
\end{align*}

Using the above two formulas, we have
\begin{align*}
\tTT'_{i,-1}(\b_{ i, \tau i, j;m_1,m_2})=&\tfX_i\tT'_{\bs_i,-1}(\b_{ i, \tau i, j;m_1,m_2}^- +\b_{ i, \tau i, j;m_1,m_2}^+)\tfX_i^{-1}\\
=&\B_{i,\tau i,j; -c_{ij}-m_1,-c_{\tau i,j}-m_2}^- + \B_{i,\tau i,j; -c_{ij}-m_1,-c_{\tau i,j}-m_2}^+\\
=& \B_{i,\tau i,j; -c_{ij}-m_1,-c_{\tau i,j}-m_2},
\end{align*}
as desired.
\end{proof}

\end{document}